\documentclass[12pt]{amsart}
\usepackage{amssymb,latexsym,mathtools}
\usepackage[margin=1.25in]{geometry}
\usepackage{mathrsfs}
\usepackage{mathtools}
\usepackage{graphicx}
\usepackage{esint}
\usepackage{braket}
\usepackage{url, hyperref}
\usepackage{tikz}
\usetikzlibrary{matrix,arrows.meta}

\newtheorem{theorem}{Theorem}
\newtheorem{lemma}{Lemma}

\newtheorem{prop}{Proposition}

\newtheorem{remark}{Remark}
\newtheorem{corollary}{Corollary}

\newcommand{\R}{\mathbb{R}}

\newcommand{\N}{\mathbb{N}}

\newcommand{\E}{\mathbb{E}}
\renewcommand{\P}{\mathbb{P}}

\makeatletter
\pgfmathdeclarefunction{erf}{1}{%
  \begingroup
    \pgfmathparse{#1 > 0 ? 1 : -1}%
    \edef\sign{\pgfmathresult}%
    \pgfmathparse{abs(#1)}%
    \edef\x{\pgfmathresult}%
    \pgfmathparse{1 / (1 + 0.3275911 * \x)}%
    \edef\t{\pgfmathresult}%
    \pgfmathparse{%
      1 - ( ( ( ( 1.061405429 * \t - 1.453152027 ) * \t + 1.421413741 ) * \t - 0.284496736 ) * \t + 0.254829592 ) * \t * exp( - \x * \x)}%
     \edef\y{\pgfmathresult}%
    \pgfmathparse{(\sign)*\y}%
    \pgfmath@smuggleone\pgfmathresult%
  \endgroup
}
\makeatother

\begin{document}

\title[Phase Transitions with Microscopic Heterogeneities]{Diffuse Interface Energies with Microscopic Heterogeneities II: Rare Events}
\author[P.S.\ Morfe and C.\ Wagner]{Peter S.\ Morfe and Christian Wagner}
\address{%
  Pennsylvania State University, 219B McAllister Building, University Park, State College, PA 16802
}
\email{pmorfe@psu.edu}
\address{%
	Institute for Science and Technology Austria (ISTA), Am Campus 1, 3400 {Kloster\-neu\-burg}, Austria
}
\email{christian.wagner@ist.ac.at}

\begin{abstract} We analyze Allen-Cahn functionals with stationary ergodic coefficients in the regime where the length scale $\delta$ of the heterogeneities is much smaller (microscopic) than the interface width $\epsilon$ (mesoscopic).  In a companion paper, we show that if the ratio $\epsilon^{-1} \delta$ vanishes fast enough as $\epsilon \to 0$, then the functionals converge to an effective surface energy where the energy density is determined by homogenization effects originating at microscopic scales.  Here we prove that if the ratio $\epsilon^{-1} \delta $ vanishes too slowly, the limit of the functional may actually be smaller than this homogenized energy. We refer to this as the rare events regime.  

In the case of the random checkerboard in dimension one, we use large deviations techniques to give a complete description of the rare events regime,  showing that the limiting energy depends in a nontrivial way on the limit of $\epsilon^{-1} \delta | \log \epsilon |$.  We further construct, in any dimension, examples of random media in which rare events become relevant at algebraic scales $\delta \approx \epsilon^{1 + \alpha}$ for an arbitrary $\alpha > 0$, as well as almost periodic examples in which atypical configurations play the same role as rare events.
\end{abstract}

\date{\today}

\maketitle

\section{Introduction} In this paper, which is the second part of \cite{part1}, we analyze diffuse interface functionals with heterogeneous coefficients in the regime where the diffuse interface length scale $\epsilon$ is much larger than the characteristic length scale $\delta$ of the heterogeneities.  Specifically, we are interested in energy functionals of the form
	\begin{equation} \label{E: our energy}
		\mathscr{F}_{\epsilon,\delta}(u; U) = \int_{U} \left( \frac{\epsilon}{2} a(\delta^{-1} x) \nabla u \cdot \nabla u + \frac{1}{\epsilon} \theta(\delta^{-1}x) W(u) \right) \, dx,
	\end{equation}
where the coefficients $a$ and $\theta$ are samples of stationary, ergodic fields (for instance, these could be random, periodic, or almost periodic) and we assume there are positive constants $\lambda \leq \Lambda$ and $\theta_{*} \leq \theta^{*}$ such that $\theta_{*} \leq \theta \leq \theta^{*}$ and $\lambda \text{Id} \leq a \leq \Lambda \text{Id}$ pointwise.  For concreteness, we assume that the double-well $W$ is a nonnegative continuous function with zeros precisely at $-1$ and $1$.  We consider the asymptotics when $(\delta,\epsilon ) \to 0$ in such a way that $\epsilon^{-1} \delta \to 0$.  Our interest is in the limiting behavior of minimizers of $\mathscr{F}_{\epsilon,\delta}$ (possibly with boundary conditions or external forces), hence we consider this problem through the lens of $\Gamma$-convergence.  Whereas the companion paper \cite{part1} considers the regime in which homogenization (or averaging) effects dominate, what we call the \emph{homogenization regime}, the aim of this paper is to investigate the regime in which atypical configurations of the medium $(a,\theta)$ contribute to the limiting behavior, or what we refer to as the \emph{rare events regime}.

Energies of the form \eqref{E: our energy} go back at least as far as the work of van der Waals \cite{van-der-waals,rowlinson_translation}, who considered them as a model to explain the formation of capillary surfaces in fluids, and also arise in phenomenological theories of interface formation in materials science (see, e.g., \cite[Chapter 8]{glicksman2010principles}).  If the coefficients $a$ and $\theta$ are constant, say, $a \equiv \bar{a}$ and $\theta \equiv \bar{\theta}$, as in a homogeneous medium, then, by a classical result of Modica and Mortola \cite{modica_mortola, modica}, the energy $\mathscr{F}_{\epsilon,\delta}$ $\Gamma$-converges as $\epsilon \downarrow 0$ to the surface energy $\bar{ \mathscr{E} }$ given by
	\begin{align}
		\bar{ \mathscr{ E } }(u ; U ) &= \left\{ \begin{array}{r l}
			\int_{ U \cap \partial \{ u = 1 \} } \bar{\sigma} ( n_{\partial \{ u = 1 \} } ) \, d \mathcal{H}^{d-1}, & \text{if} \, \, u \in BV_{\text{loc}}(U; \{-1,1\}),  \\
			 \infty, & \text{otherwise,}
		\end{array} \right. \label{E: homogenized energy} \\ 
		&\text{where} \, \, \bar{\sigma}(n)^{2} = \sigma_{W}^{2} \bar{\theta}  n \cdot \bar{a} n  \quad \text{and} \quad \sigma_{W} = \int_{-1}^{1} \sqrt{ W(u) } \, du. \nonumber
	\end{align}
(See, for instance, \cite{AlbertiLectureNotes} for an expository discussion of this theorem.)

In this article and the companion paper \cite{part1}, our interest is in the extent to which this qualitative picture changes when $a$ and $\theta$ are nonconstant, but instead vary in space in a stationary and ergodic manner, as in a heterogeneous material. More broadly, this problem fits within the larger mathematical investigation of homogenization for geometric PDE's and variational problems; see the survey \cite{caffarelli} for an overview.  As mentioned already above, we focus on the setting in which the length scale $\delta$ of the heterogeneities is much smaller than the typical width $\epsilon$ of the interfacial region.

Heuristically, if $\epsilon^{-1} \delta \to 0$ as $\epsilon \to 0$, it is natural to expect that homogenization at the smallest scale $\delta$ determines what happens at the interface length scale $\epsilon$, and, thus, the limit of $\mathscr{F}_{\epsilon,\delta}$ ought to be the same as what would be obtained upon first sending $\delta \to 0$ and then sending $\epsilon \to 0$.  By classical homogenization theory, this is precisely the energy $\bar{\mathscr{E}}$ above provided we let $\bar{\theta} = \mathbb{E} [ \theta(0) ]$ be the ensemble average of $\theta$ and $ \bar{a} $ be the homogenized matrix associated with the operator $\nabla \cdot ( a(y) \nabla )$.

Until recently, $\bar{\mathscr{E}}$ was known to be the limit of $\mathscr{F}_{\epsilon,\delta}$ when $\epsilon , \delta , \epsilon^{-1} \delta \to 0$ only in a few cases, almost exclusively in the setting when $a$ and $\theta$ are periodic.  First, it was shown by Ansini, Braides, and Chiad\`{o}-Piat in \cite{ansini_braides_chiado-piat} when $a$ is a periodic function, $\theta$ is constant, and under the stricter assumption that $\epsilon^{-3/2} \delta \to 0$.  More recently, Cristoferi, Ganedi, and Fonseca \cite{cristoferi_fonseca_ganedi_supercritical} proved it in case $a$ is constant and $\theta$ is periodic, and their work was revisited by Wojtowtysch \cite{wojtowytsch}, who also proved convergence to $ \bar{\mathscr{E}}$ when $\theta$ is the so-called random checkerboard and $ \epsilon^{ - ( 1 + \frac{ 2 }{ d } ) } \delta \to 0 $. 

In the companion paper \cite{part1}, we prove that, in the generality of stationary ergodic coefficients, $\mathscr{F}_{\epsilon,\delta}$ converges to $\bar{\mathscr{E}}$ provided $\epsilon^{-1} \delta \to 0$ sufficiently fast. In that result, the speed is quantified in terms of rates of homogenization, namely, the probability that $\theta(\delta^{-1} x)$ deviates from its average $\mathbb{E} [ \theta (0) ]$ and $( \nabla \cdot a(\delta^{-1} x) \nabla )^{-1}$ from the homogenized operator  $( \nabla \cdot \bar{a} \nabla )^{-1}$ on scales of order $ \epsilon $. We refer to this setting as the homogenization regime.

Our aim in this article is to complement the results of \cite{part1} by showing that some assumption on the rate of convergence of $\epsilon^{-1} \delta$ is necessary: There are coefficients $a$ and $\theta$  such that if $\epsilon^{-1} \delta \to 0$ too slowly, then $\mathscr{F}_{\epsilon,\delta}$ will not converge to $\bar{ \mathscr{E} }$.  We make this precise using the notion of surface tension.  

For an arbitrary unit vector $n \in S^{d-1}$ denote by $Q^{ n } $ a cube with two faces orthonormal $ n $. Given an arbitrary smooth function $ q : \mathbb{R} \to [-1,1]$ with $q(\infty) = 1$ and $q(-\infty) = -1$, we define the finite-volume surface tension in the direction $n$ in the box $ Q^{ n } ( x ) \coloneqq x + Q^{n} $ by
	\begin{align}\label{eqn:surface-tension}
	\begin{aligned}
		& \sigma_{\epsilon,\delta}( n ; Q^{n} ( x ) ) \\ &\qquad = \min \left\{ \mathscr{F}_{\epsilon,\delta}(u; Q^{n} ( x ) )\, \mid \, u(x') = q( \epsilon^{-1} ( x' - x) \cdot n ) \, \, \text{for} \, \, x' \in \partial Q^{n}\right\}. \\
	\end{aligned}
	\end{align}
Simply put, this measures the minimal energy required to form a planar interface with normal $n$ in the box $ Q^{n} ( x ) $.

It is well-known that the limit of the finite-volume surface tension determines that of $\mathscr{F}_{\epsilon,\delta}$ (see \cite{ansini_braides_chiado-piat,morfe}).  In \cite[Theorem 2]{part1}, we prove that, for any scaling $\epsilon \mapsto \delta(\epsilon)$ such that $\epsilon^{-1} \delta(\epsilon) \to 0$ as $\epsilon \downarrow 0$, 
	\begin{align} \label{E: what we want to prove}
		\limsup_{\epsilon \downarrow 0} \sigma_{\epsilon,\delta(\epsilon)} ( n ; Q^{n}(x) ) \leq \bar{\sigma}(n) | Q^{n} |^{ \frac{ d - 1 }{ d } } ,
	\end{align} 
where $ | Q^{n} |$ denotes the volume of $Q^{n}$.  While in the homogenization regime considered in \cite{part1} equality holds in \eqref{E: what we want to prove}, in the present paper, we give examples where the inequality is strict.  As the proofs given here show, this reflects a competition between averaging and energy minimization: If $\delta$ is only moderately small compared to $\epsilon$, minimizing interfaces can take advantage of atypical configurations of the medium $(a,\theta)$ to drag their energy below the mean.  In contrast to the homogenization regime, we refer to this as the rare events regime as, in the one-dimensional setting, there is a close connection between this phenomenon and large deviations.

For further references and discussion of previous results, we refer the reader to \cite{part1}.

\subsection{Analysis of the 1D Random Checkerboard} Our first result concerns the random checkerboard in dimension $d = 1$, where we sharply determine the threshold between the homogenization and rare events regimes.  By the random checkerboard, we mean the random fields $(a,\theta)$ determined by the formula
	\begin{equation} \label{E: one d random checkerboard}
		(a(x),\theta(x)) = \sum_{z \in \mathbb{Z}^{d}} (A_{z},\Theta_{z}) \boldsymbol{1}_{ z + (-1/2,1/2)^{d} }(x)
	\end{equation}  
where $\{ (A_{z},\Theta_{z} ) \}_{z \in \mathbb{Z}^{d}}$ is a sequence of i.i.d.\ random variables such that $\lambda \text{Id} \leq A_{z} \leq \Lambda \text{Id}$ as symmetric matrices and $\theta_{*} \leq \Theta_{z} \leq \theta^{*}$ for some positive $\lambda \leq \Lambda$ and $ \theta_{*} \leq \theta^{*}$.

The results of \cite[Theorem 1 \& Sections 4.6--4.7]{part1} establish that, for the random checkerboard, $\mathscr{F}_{\epsilon,\delta}$ converges to $\bar{\mathscr{E}}$ if 
	\begin{align}\label{E: assumption homogenization regime}
		\lim_{ \epsilon \downarrow 0 } \frac{ \delta ( \epsilon ) | \log \epsilon |^{1/d} }{ \epsilon } = 0.
	\end{align}
In particular, in that case, equality holds in \eqref{E: what we want to prove}. We expect that the logarithm can be removed in dimensions $d \geq 2$; that question will be revisited in future work. By contrast, the next result,  shows that the logarithm is actually sharp in dimension $d = 1$; see Figure \ref{fig01} for an illustration.  For technical reasons discussed below, we restrict to the case when $a$ is constant.

	\begin{theorem} \label{T: main deviations} Consider the random checkerboard $(a,\theta)$ in dimension $d = 1$ in the special case when $a \equiv 1$.  Assume that the random variables $\{\Theta_{z}\}_{z \in \mathbb{Z} }$ are nonconstant with essential infimum $\theta_{*}$, and let $\sigma_{*} = \sigma_{W} \sqrt{ \theta_{*} }$, where $\sigma_{W}$ is as in \eqref{E: homogenized energy}.  
    
    There is a nonincreasing continuous function $\check{\sigma} : [0,\infty] \to [\sigma_{*},\bar{\sigma}(1)]$ such that if $\epsilon \mapsto \delta(\epsilon)$ is any choice of scaling such that $\epsilon^{-1} \delta(\epsilon) \to 0$ as $\epsilon \to 0$ and for which there is a constant $\kappa \in [0,\infty]$ such that 
		\begin{align} \label{E: log condition}
			\lim_{\epsilon \downarrow 0} \frac{ \delta(\epsilon) | \log \epsilon | }{ \epsilon } = \kappa,
		\end{align}
	then, for any bounded interval $I$,
		\begin{align*}
			\lim_{\epsilon \downarrow 0} \sigma_{\epsilon,\delta(\epsilon)}(1; I ) = \lim_{ \epsilon \downarrow 0 } \sigma_{\epsilon,\delta(\epsilon)} (-1; I ) = \check{\sigma}(\kappa) \quad \text{in probability.}
		\end{align*}
    Further, $\check{\sigma}(0) = \bar{\sigma}(1)$, $\check{\sigma}(\kappa) < \bar{\sigma}(1)$ for any $\kappa > 0$, and $\check{\sigma}(\infty) = \sigma_{*}$. 
	\end{theorem}

As a consequence of the theorem, in the next corollary, we obtain a characterization of the possible $\Gamma$-limits of $\mathscr{F}_{\epsilon,\delta}$.  Here limits are interpreted in terms of convergence in probability, as made precise in \cite{part1}: Namely, for any bounded interval $I \subseteq \mathbb{R}$, there is a metric space $(\mathcal{M},d_{\Gamma})$ consisting of l.s.c.\ functionals on $L^{1}(I)$ such that convergence in $\mathcal{M}$ is equivalent to $\Gamma$-convergence and $\mathbb{P} \{ \mathscr{F}_{\epsilon,\delta}(\cdot \, ; I) \in \mathcal{M} \} = 1$ for any $\epsilon \in (0,1)$ and $\delta > 0$.  The statement that a sequence $\{\mathscr{G}_{n}\}$ of $\mathcal{M}$-valued random variables $\Gamma$-converges in probability to some functional $\mathscr{G}$ means that the random variable $d_{\Gamma}(\mathscr{G}_{n},\mathscr{G})$ converges to zero in probability as $n \to \infty$.   (In the statement, we denote by $S_{u}$ the jump set of a $BV$ function $u$.)

\begin{corollary} \label{C: gamma limits} Assume $a \equiv 1$ and $\theta$ is the 1D random checkerboard as in Theorem \ref{T: main deviations}.  Suppose that $\epsilon \mapsto \delta(\epsilon)$ is any scaling such that $\epsilon^{-1} \delta(\epsilon) \to 0$ as $\epsilon \to 0$.  If there is a $\kappa \in [0,\infty]$ such that
    \begin{align} \label{E: limit corollary}
        \lim_{ \epsilon \downarrow 0 } \frac{ \delta(\epsilon) | \log \epsilon | }{ \epsilon } = \kappa , 
    \end{align}
then for any bounded open interval $I \subseteq \mathbb{R}$, the functional $\mathscr{F}_{\epsilon,\delta(\epsilon)}(\cdot \, ; I)$ $\Gamma$-converges in probability to the energy $\check{\mathscr{E}}_{\kappa}(\cdot \, ;I)$ given by
    \begin{align*}
        \check{\mathscr{E}}_{\kappa}(u;I) = \left\{ \begin{array}{r l}
            \check{\sigma}(\kappa) \# ( S_{u} \cap I ), & \text{if} \, \, u \in BV(I; \{-1,1\}), \\ \infty, & \text{otherwise.}
        \end{array} \right.
    \end{align*}
\end{corollary}

Notice that the limit $\check{\mathscr{E}}_{\kappa}$ defined above depends on $\kappa$ only through the number $\check{\sigma}(\kappa)$.  This is important since $\check{\sigma}$ is not necessarily injective (see Remark \ref{R: eventually constant} below).  After passing to subsequences and accounting for this degeneracy, we obtain the following characterization of the possible accumulation points of $\mathscr{F}_{\epsilon,\delta(\epsilon)}$ for a fixed choice of the scale $\epsilon \mapsto \epsilon(\delta)$.

\begin{corollary} \label{C: characterize gamma limits} Assume $a \equiv 1$ and $\theta$ is the 1D random checkerboard as in Theorem \ref{T: main deviations}.  Suppose that $\epsilon \mapsto \delta(\epsilon)$ is any scaling such that $\epsilon^{-1} \delta(\epsilon) \to 0$ as $\epsilon \to 0$.  Let $I \subseteq \mathbb{R}$ be a bounded open interval and suppose that $\mathscr{E}(\cdot \, ; I) \in \mathcal{M}$. If $\{\epsilon_{j}\}_{j \in \mathbb{N}}$ is any sequence such that $\epsilon_{j} \to 0$ as $j \to \infty$, then
    \begin{align*}
        d_{\Gamma}( \mathscr{F}_{\epsilon_{j},\delta(\epsilon_{j})}(\cdot \, ; I ) , \mathscr{E} ( \cdot \, ; I ) ) \to 0 \quad \text{in probability as} \, \, j \to \infty
    \end{align*}
if and only if there is a $\kappa \in [0,\infty]$ such that 
    \begin{align*}
        \lim_{ j \to \infty } \check{\sigma} ( \epsilon_{j}^{-1} \delta(\epsilon_{j}) | \log \epsilon_{j} | ) = \check{\sigma}(\kappa) \quad \text{and} \quad \mathscr{E}( \cdot \, ; I ) = \check{\mathscr{E}}_{\kappa}(\cdot \, ; I ) . 
    \end{align*}
In particular, in the limit $\epsilon \to 0$, the accumulation points of $\mathscr{F}_{\epsilon,\delta(\epsilon)}$ are in one-to-one correspondence with the accumulation points of $\check{\sigma}(\epsilon^{-1} \delta(\epsilon) |\log \epsilon|)$. \end{corollary}

\begin{figure}

\centering

\begin{tikzpicture}

	\draw[->] (-4.5, -0.5) -- node[at end, right] {$ \kappa $} (4.0, -0.5);
	
	\draw[->] (-4.0, -1.0 ) -- node[at end, above] {$ \check\sigma(\kappa) $} ( -4.0, 2.7 );

	\draw ( -4.1, 0.0 ) -- node[at start, left] { $ \sigma_W \sqrt{ \theta_* } $ } ( -3.9, 0.0 );
	\draw [dashed]  ( -3.6, 0.0 ) -- ( 3.6, 0.0 );
	
	\draw ( -4.1, 1.91 ) -- node[at start, left] { $ \sigma_W \sqrt{ \bar\theta } $ } ( -3.9, 1.91 );

	\draw ( -3.6, -0.6 ) -- node[at start, below] { $ 0 $ } ( -3.6, -0.4 );

	\draw plot[domain = - 3.6 : 3.6, samples=200] (\x, { 1 - 2 * rad( atan( 2 * \x) / pi  ) });




\end{tikzpicture}
	
\caption{Theorem \ref{T: main deviations} shows that on the one hand, if $ \epsilon^{-1} \delta $ vanishes faster than $1/|\log \epsilon|$, then the relevant effective interfacial energy density is $\sigma_{W} \sqrt{\bar{\theta} }$.  This is the homogenization regime.  On the other hand if $ \epsilon^{-1} \delta $ vanishes slower than $1/|\log \epsilon|$, the minimum value $\theta_{*}$ of $\theta$ determines the macroscopic energy. Between these two behaviors there is a continuous transition w.r.t.~to the limit $\kappa$ of $\epsilon^{-1} \delta(\epsilon) |\log \epsilon |$, determined by the function $\check{\sigma}$.  Since $\check{\sigma}(\kappa) < \sigma_{W} \sqrt{ \bar{\theta} }$ for any $\kappa > 0$, the condition $\kappa > 0$ characterizes the rare events regime in this setting.}

\label{fig01}

\end{figure}
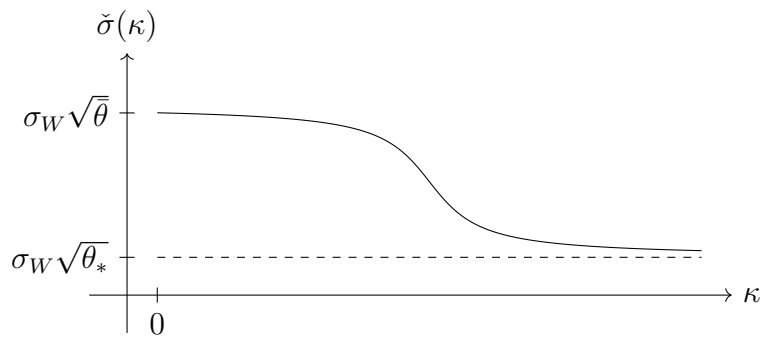

In equation \eqref{eqn:defn-sigma-check} below, we provide an explicit formula for $ \check\sigma $ in terms of a large deviations rate function that we introduce in Proposition \ref{P: large deviations}.

	\begin{remark} \label{R: eventually constant} It is possible to prove that if $\mathbb{P} \{ \theta(x) = \theta_{*} \} > 0$, then there is a $\kappa_{c} > 0$ such that $\check{\sigma}(\kappa) = \sigma_{*}$ for each $\kappa \geq \kappa_{c}$. \end{remark}
	
The proof of Theorem \ref{T: main deviations} proceeds via a novel application of large deviations techniques.  As explained in Section \ref{S: overview large deviations} below, the choice of length scales amounts to a selection of a large deviations rate: The number $\kappa$ is the rate at which atypical configurations of $\theta$ with surface tension smaller than $\check{\sigma}(\kappa)$ appear, and the left-hand side of \eqref{E: log condition} asserts those configurations occupy a nonnegligible volume fraction of space in the limit $\epsilon \downarrow 0$.    At a high level, this correspondence between length scales and large deviations rates is a quenched, variational analogue of the one identified by Vanden-Eijnden and Westdickenberg \cite{vanden-eijnden_westdickenberg} in the context of nucleation in the stochastic Allen-Cahn equation.

In Theorem \ref{T: main deviations}, we restrict attention to the case when $a$ is constant because large deviations principles are readily available for linear (in $\theta$) random variables of the form $\int_{-\infty}^{\infty} \theta(\delta^{-1} x) W(u(x)) \, dx$ for a given $u$.  We expect that the theorem above could be generalized to the case when $a$ is nonconstant.  However, since homogenization of the gradient term in \eqref{E: our energy} is determined by the inverse $( \nabla \cdot a(\delta^{-1} x) \nabla )^{-1}$, a nonlinear function of $a$, this would require significantly more effort.  In particular, this is closely related to the question of large deviations principles for the solution $u^{\delta}$ of the equation $-\nabla \cdot ( a ( \delta^{-1} x ) \nabla u^{\delta} ) = f(x)$ for a given $f$, which has yet to be treated in the literature to the best of our knowledge.

Nonetheless, the proof of Theorem \ref{T: main deviations} can be adapted to establish the following weaker version in case of the general 1D random checkerboard with $a$ nonconstant.  In particular, the case $\kappa =  \infty$ is easily treated even in this generality.

	\begin{prop}\label{P: deviations 1d checkerboard}
	Consider the random checkerboard $(a,\theta)$ in dimension $d = 1$, and let $\lambda > 0$ be the essential infimum of the random variables $\{A_{z}\}_{z \in \mathbb{Z}}$ and $\theta_{*} > 0$ be the essential infimum of $\{ \Theta_{z} \}_{z \in \mathbb{Z}}$.  If $\epsilon \mapsto \delta(\epsilon)$ is any choice of scaling for which we have
		\begin{align*}
			\lim_{\epsilon \downarrow 0} \frac{ \delta(\epsilon) | \log \epsilon | }{ \epsilon } = \infty \quad \text{and} \quad \lim_{\epsilon \downarrow 0} \frac{ \delta(\epsilon) }{ \epsilon } = 0,
		\end{align*}
	then, for any bounded interval $I$,
		\begin{align*}
			\lim_{ \epsilon \downarrow 0 } \sigma_{\epsilon,\delta(\epsilon)} ( 1 ; I ) = \lim_{ \epsilon \downarrow 0 }  \sigma_{\epsilon,\delta(\epsilon)} ( -1 ; I ) = \sigma_{W} \sqrt{ \lambda \theta_{*} } \quad \text{in probability.}
		\end{align*}
	\end{prop}
		
As is recalled at the beginning of Section \ref{S: heuristics} and Section \ref{S: deviations 1d checkerboard}, the strict inequalities $\bar{\theta} > \theta_{*}$ and $\bar{a} > \lambda$ hold as soon as $\{A_{z}\}$ and $\{\Theta_{z}\}$ are nonconstant, so that, in particular, $\bar{\sigma}(1) > \sigma_{W} \sqrt{ \lambda \theta_{*} }$.  Therefore, the previous result shows that, for the 1D random checkerboard, it is always possible to find a scaling $\epsilon \mapsto \delta(\epsilon)$ for which equality fails to hold in \eqref{E: what we want to prove}.
	
\subsection{Counterexamples in Higher Dimensions} The results of the previous section show that, in dimension one, even in the case of the random checkerboard, the energy $\mathscr{F}_{\epsilon,\delta}$ need not converge to the homogenized functional $\bar{\mathscr{E}}$ if the ratio $\epsilon^{-1} \delta$ does not vanish fast enough.  

The remainder of the paper is devoted to some further counterexamples, this time in arbitrary dimensions, which demonstrate further pathologies that can be encountered. For simplicity, we restrict attention to media with constant $a \equiv \text{Id}$ and with the function $\theta$ taking the form
	\begin{equation} \label{E: quasi one dimensional}
		\theta(y) = \theta^{\text{stripe}}(y_{1}) \tilde{\theta}(y), \quad \text{where} \, \, y = (y_{1},y_{2},\dots,y_{d}),
	\end{equation}
where $\theta^{\text{stripe}}$ and $\tilde{\theta}$ are suitable stationary ergodic fields.

In the previous works \cite{ansini_braides_chiado-piat,hagerty}, positive results were obtained in the periodic setting under the assumption that $\epsilon^{- \gamma } \delta \to 0$ for some $\gamma > 1$.  The next result shows that these cannot be extended to the random case.  The form \eqref{E: quasi one dimensional} of $\theta$ motivates consideration of the direction $e_{1} = (1,0,\dots,0)$.

	\begin{theorem} \label{T: algebraic correlation counterexample}Let $a \equiv \text{Id}$ and let $\tilde{\theta}$ be a $d$-dimensional random checkerboard with mean $\mathbb{E}[\tilde{\theta}] = 1$.  Fix a $\gamma > 1$.  It is possible to construct the joint law of $(\theta^{\text{stripe}},\tilde{\theta})$ such that,  in any dimension $d$, if $\theta$ is defined by \eqref{E: quasi one dimensional} and the scale $\epsilon \mapsto \delta(\epsilon)$ is chosen so that both $\epsilon^{-1} \delta(\epsilon) \to 0$ as $\epsilon \to 0$ and
		\begin{equation*}
			\liminf_{\epsilon \downarrow 0} \epsilon^{-\gamma} \delta(\epsilon) > 0,
		\end{equation*}
	then, for any $\varrho > 0$,
		\begin{equation} \label{E: algebraic correlations lim inf}
			\liminf_{ \epsilon \downarrow 0 } \varrho^{-(d-1)} \sigma_{\epsilon,\delta}(e_{1} ; [-\varrho/2,\varrho/2]^{d} )  \leq \sigma_{W} \sqrt{\theta^{\text{stripe}}_{*}}  < \bar{\sigma}(e_{1})
		\end{equation}
	in probability.  Above $\theta_{*}^{\text{stripe}}$ is the essential infimum of $\theta^{\text{stripe}}$.\end{theorem}

See Section \ref{S: long range correlations} for the proof.
		
So far, the examples discussed all involved random media.  The next example shows that atypical configurations also play a role in the case of almost periodic media.  

As in the companion paper \cite{part1}, we consider almost periodic functions in the sense of Besicovitch: These are functions in the Besicovitch spaces $\{ B^{p}(\mathbb{R}^{d}) \, \mid \, 1 \leq p \leq \infty\}$ obtained for any $p \geq 1$ by completing the algebra of trigonometric polynomials (with arbitrary, possibly noninteger frequencies) under the norm
	\begin{align*}
		\| f \|_{B^{p}(\mathbb{R}^{d})} = \left\{ \begin{array}{r l}
									\left( \lim_{ R \to \infty} \frac{1}{(2R)^{d}} \int_{[-R,R]^{d}} |f|^{p} \, dx \right)^{\frac{1}{p}} , & \text{if} \, \, p < \infty, \\
									\|f\|_{L^{\infty}(\mathbb{R}^{d})}, & \text{if} \, \, p = \infty .
								\end{array} \right.
	\end{align*}
In the next result, we define $\theta$ in \eqref{E: quasi one dimensional} by constructing a suitable function $\theta^{\text{stripe}} \in \cap_{ 1 \leq p < \infty } B^{p}(\mathbb{R})$ and letting $\theta$ be an arbitrary continuous, positive $\mathbb{Z}^{d}$-periodic function, so that H\"{o}lder's inequality implies that $\theta \in \cap_{1 \leq p < \infty} B^{p}(\mathbb{R}^{d})$.  The relevant definition of $\bar{\theta}$ in this case is the mean value
	\begin{align} \label{E: mean value}
		\bar{\theta} = \lim_{ R \uparrow \infty } \frac{ 1 }{ (2R)^{d} } \int_{ [ -R, R ]^{ d } } \theta \, dx 
	\end{align}
and we use this number to define $\bar{\sigma}(e) = \sigma_{W} \sqrt{ \bar{\theta} }$ as in random media. 

	\begin{theorem} \label{T: almost periodic counterexample}Let $a \equiv \text{Id}$ and let $\tilde{\theta}$ be an arbitrary positive, continuous, $\mathbb{Z}^{d}$-periodic function with mean $\int_{\mathbb{T}^{d}} \tilde{\theta} = 1$. There is a choice of almost periodic function $\theta^{\text{stripe}} : \mathbb{R} \to \{1,2\}$, which is in $B^{p}(\mathbb{R})$ for any $p < \infty$, and a scale $\epsilon \mapsto \delta(\epsilon)$ satisfying $\epsilon^{-1} \delta(\epsilon) \to 0$ as $\epsilon \to 0$ such that  if $\theta$ is given by \eqref{E: quasi one dimensional}, then, for any $\varrho > 0$,
		\begin{equation} \label{E: almost periodic example equation}
			\lim_{\epsilon \downarrow 0} \varrho^{-(d-1)} \sigma_{\epsilon,\delta(\epsilon)}(e_{1}; [-\varrho/2,\varrho/2]^{d} ) \leq \sigma_{W}  < \bar{\sigma}(e_{1}).
		\end{equation}		
	\end{theorem}

This last result shows that the rare events regime is not merely a probabilistic phenomenon, but is instead a general feature of stationary ergodic media.  The proof is given in Section \ref{S: quasiperiodic examples}.

It should be stressed that \cite[Corollary 2]{part1} implies that $\mathscr{F}_{\epsilon,\delta}$ always converges to $\bar{\mathscr{E}}$ if the coefficients $a$ and $\theta$ are periodic or uniformly almost periodic (i.e., functions in $B^{\infty}$), no matter how slowly the ratio $\epsilon^{-1} \delta$ vanishes.  The previous theorem thus shows that this is no longer true if the coefficients are nonuniformly almost periodic.

\addtocontents{toc}{\protect\setcounter{tocdepth}{0}}
\section*{Acknowledgements}  

We thank Felix Otto for organizing a stimulating research environment at the MPI in Leipzig, where parts of this research were conducted. The first author acknowledges the support of NSF Grant DMS-2202715.  We also thank the late Robert V.\ Kohn for bringing our attention to reference \cite{vanden-eijnden_westdickenberg}.

\addtocontents{toc}{\protect\setcounter{tocdepth}{2}}
\tableofcontents

\section{Rare Events in the 1D Random Checkerboard} \label{S: heuristics}

In this section, we analyze the functional $\mathscr{F}_{\epsilon,\delta}$ in 1D in the case when $a \equiv 1$ and $\theta$ is the random checkerboard.  In particular, we prove Theorem \ref{T: main deviations}, giving a full characterization of the possible limits.  We also prove Proposition \ref{P: deviations 1d checkerboard}, which provides a partial extension to the case when both $a$ and $\theta$ are 1D random checkerboards.

Since $S^{0} = \{-1,1\}$ and the random checkerboard is reflection invariant, $\bar{\sigma}$ is a constant function.  For notational convenience, we will abuse notation by letting $\bar{\sigma}$ also denote this constant. For brevity, we will use the notation $\bar{\mathscr{F}}$ and $\mathscr{F}_{c}$ for any constant $c > 0$ to denote the constant-coefficient functionals
    \begin{align*}
        \bar{\mathscr{F}}(u; \mathbb{R}) &= \int_{-\infty}^{\infty} \left( \frac{1}{2} |u'|^{2} + \bar{\theta} W(u) \right) \, dy, \quad \mathscr{F}_{c}(u;\mathbb{R}) &= \int_{-\infty}^{\infty} \left( \frac{1}{2} |u'|^{2} + c W(u) \right) \, dy,
    \end{align*}
where we recall that $\bar{\theta}$ is the mean of $\theta$.

We will use the fact that $ \bar\sigma $ can be characterized in terms of the variational problem
\begin{align}
		\bar\sigma &= \min \left\{ \bar{\mathscr{F}}(u;\mathbb{R}) \, \mid \, u(\pm \infty) = \pm 1 \right\}, \label{eqn:bar-sigma}
\end{align}
see \cite[(3.5)]{AlbertiLectureNotes} and also \eqref{E: homogenized energy}. By the same reasoning, $ \sigma_* $ as defined in Theorem \ref{T: main deviations} is given by
\begin{align} \label{E: lower energy variational formula}
		\sigma_*&= \min \left\{ \mathscr{F}_{\theta_{*}}(u;\mathbb{R}) \, \mid \, u(\pm \infty) = \pm 1 \right\} ,
\end{align}
where, as in the statement of the theorem, $\theta_{*}$ is the essential infimum of $\theta$.
Since $\theta$ is nonconstant in Theorem \ref{T: main deviations}, we know that $ \theta_* < \bar\theta $, and thus $ \sigma_* < \bar\sigma $.

\subsection{Deviations in the 1D Random Checkerboard} \label{S: overview large deviations} We now prove the main result, Theorem \ref{T: main deviations}.  We proceed by proving three intermediate results.
To make the main ideas of the proof more transparent, it is convenient to assume $q$ is constant outside of the interval $(-1,1)$:
	\begin{align} \label{A: triviality}
		q( y ) = 1 \quad \text{if} \, \, y \geq 1, \quad q(y) = -1 \quad \text{if} \, \, y \leq -1.
	\end{align}
Assumption \eqref{A: triviality} reduces some technical overhead in the proof and is justified by the following remark. 
		
		\begin{remark}\label{rmk: triviality} The asymptotic behavior of the finite-volume surface tension $\sigma_{\epsilon,\delta}$ is not affected by the choice of boundary condition $q$.  This can be seen, for instance, via arguments similar to those in \cite[Theorem 3.7]{ansini_braides_chiado-piat} or \cite[Proposition 10]{morfe}. 
		\end{remark}
        
First, we prove a large deviations principle for the finite-volume surface tension in a fixed mesoscopic interval.

	\begin{prop} \label{P: large deviations} Assume \eqref{A: triviality}.  For any $r \geq 1$, there is a nondecreasing left-continuous function $J_{r} : [0,\infty) \to [0,\infty]$ such that, for any $\zeta \geq 0$, 
		\begin{align}\label{eqn:ldp-principle-r}
				\lim_{ \gamma \downarrow 0 } \gamma \log \mathbb{P} \left\{ \sigma_{ 1, \gamma } ( 1 ; [-r,r] ) \leq \bar{\sigma} - \zeta \right\} = -J_{r}(\zeta).
		\end{align}
	Further, the functions $\{J_{r}\}_{r \geq 1}$ are nonincreasing in $r$, and the limit $J = \inf_{r \geq 1} J_{r}$ has the following properties:
		\begin{itemize}
			\item[(i)] $J(0) = 0$ and $J$ is continuous at zero.
			\item[(ii)] The interior of the set $\{ \zeta \geq 0 \, \mid \, J(\zeta) < \infty \}$ equals $(0,\bar{\sigma} - \sigma_{*})$.
			\item[(iii)] $J$ is strictly increasing and left-continuous in the interval $[0,\bar{\sigma} - \sigma_{*}]$. 		
		\end{itemize}  
	\end{prop}

Next, we exploit the large deviations principle to determine the asymptotics of $\sigma_{\epsilon,\delta}$ in the joint limit when $\epsilon^{-1} \delta \downarrow 0$ as $\epsilon \downarrow 0$.  The heuristic that follows explains how this is done and is a counterpart to the one discussed in \cite[Sections 3.5 \& 4.3]{part1}: Without loss of generality, we consider an interval $I = [-\varrho,\varrho]$ centered at zero.

Working at the mescoscopic scale 
\begin{align}\label{eqn:change-of-variables}
	R \coloneqq \epsilon^{ -1 } \varrho
	\quad \text{and} \quad
	\gamma \coloneqq \epsilon^{ - 1 } \delta(\epsilon) ,
\end{align}
and changing variables according to $y = \epsilon^{-1} x$, the formula \eqref{eqn:surface-tension} implies that 
	\begin{align*}
			\sigma_{\epsilon,\delta(\epsilon)}(1; [-\varrho,\varrho] ) = \sigma_{1,\gamma}(1;[-R,R]) .
	\end{align*}
In particular, this motivates the study of the variational problem
	\begin{align*}
		\sigma_{1,\gamma}(1;[-R,R]) &= \min \left\{ \int_{-R}^{R} \left( \frac{1}{2} |u'|^{2} + \theta(\gamma^{-1} y) W(u) \right) \, dy \, \mid \, u(\pm R) = q(\pm R) \right\} ,
	\end{align*}
in which, in the limit $\epsilon \searrow 0$, there is a competition between the fluctuations resulting from the increase in $R$ and the averaging effect induced by the decrease in $\gamma$.

Notice that \eqref{eqn:change-of-variables} allows us to write $\gamma = \gamma(R)$ and, in the new variables, the limit $\epsilon \searrow 0$ corresponds to sending $R \nearrow \infty$.
	
Since in the new variables the diffuse interface length scale is of order one, up to making a small error, we expect that
	\begin{align*}
		\mathbb{P} \left\{ \sigma_{1,\gamma}(1;[-R,R]) \leq \bar{\sigma} - \zeta \right\}
			&\approx \mathbb{P} \big( \bigcup_{k \in r \mathbb{Z} \cap [-R,R]} \left\{ \sigma_{1,\gamma}(1;[k-r,k+r]) \leq \bar{\sigma} - \zeta \right\} \big) ,
	\end{align*} 
where $ r \geq 1$ is arbitrary (but the error we make only becomes negligible as $ r \to \infty$). Since the random variables $\sigma_{1,\gamma}(1;[k-r,k+r])$ and $\sigma_{1,\gamma}(1;[j-r,j+r])$ are independent if $|k - j| > 2 ( r + \gamma )$, concentration arguments show that the righthand side above converges to $1$ if
	\begin{align*}
		\lim_{ R \uparrow \infty } \sum_{ k \in r \mathbb{Z} \cap [-R,R]} \mathbb{P} \left\{ \sigma_{1,\gamma(R)}(1;[k-r,k+r]) \leq \bar{\sigma} - \zeta \right\} = \infty .
	\end{align*}
Using Proposition \ref{P: large deviations}, this is essentially equivalent to asking that
	\begin{align*}
		\lim_{ R \uparrow \infty } R \exp ( - \frac{ J_{r}(\zeta) }{ \gamma(R) } ) = \infty .
	\end{align*}
This last condition motivates the hypothesis \eqref{E: log condition} in Theorem \ref{T: main deviations} because
	\begin{align}\label{eqn:limit-condition-ed-to-gr}
		\lim_{ \epsilon \downarrow 0 } \frac{ \delta ( \epsilon ) }{ \epsilon } | \log \epsilon | = J_{0}
		\quad \stackrel{(\ref{eqn:change-of-variables})}{\iff} \quad
		\lim_{ R \uparrow \infty } R \exp ( - \frac{ J }{ \gamma(R) } )  = \left\{ \begin{array}{r l}
            0, & \text{if} \, \, J > J_{0} , \\
            \infty , & \text{if} \, \, J < J_{0} .
        \end{array} \right.
	\end{align}

The next two propositions make this heuristic precise.

	\medskip

	\begin{prop} \label{P: lower deviations} Let $J = \inf J_{r}$ be the function from Proposition \ref{P: large deviations} and fix $\zeta \geq 0$.  If the scaling $\epsilon \mapsto \delta(\epsilon)$ is chosen in such a way that $\epsilon^{-1} \delta(\epsilon) \to 0$ as $\epsilon \searrow 0$ and 
		\begin{align}\label{E: assumption about delta}
			\liminf_{ \epsilon \downarrow 0} \frac{ \delta(\epsilon)  | \log \epsilon | }{ \epsilon }> J( \zeta ),
		\end{align}
	then, for any interval $ \varrho >  0 $,
		\begin{align*}
			\lim_{ \epsilon \downarrow 0 } \mathbb{P} \{ \sigma_{\epsilon,\delta(\epsilon)}(1; [ - \varrho , \varrho ] ) \leq \bar{\sigma} - \zeta \} = 1.
		\end{align*}
	\end{prop}

	\begin{prop} \label{P: upper deviations} Let $J = \inf J_{r}$ be the function from Proposition \ref{P: large deviations} and fix $\zeta \geq 0$.  If the scaling $\epsilon \mapsto \delta(\epsilon)$ is chosen in such a way that $\epsilon^{-1} \delta(\epsilon) \to 0$ as $\epsilon \searrow 0$ and
		\begin{align}\label{eqn:lb-helper-minus-two}
			\limsup_{ \epsilon \downarrow 0 } \frac{ \delta(\epsilon) | \log \epsilon | }{ \epsilon }  < J( \zeta ),
		\end{align}
	then, for any $\varrho > 0$,
		\begin{align*}
			\lim_{ \epsilon \downarrow 0 } \mathbb{P} \{ \sigma_{\epsilon,\delta(\epsilon)}(1; [ - \varrho , \varrho ] ) \geq \bar{\sigma} - \zeta \} = 1.
		\end{align*}
	\end{prop}

	The proofs of these propositions are facilitated by the fact that the functions $\{J_{r}\}$ are strictly increasing.

	\begin{prop}[Monotonicity of $ J_{ r } $ and $ J $] \label{cor:monotonicity-jr}
	For any $r \geq 1$, the function $ J_{ r } $ from Proposition \ref{P: large deviations} is strictly increasing; more precisely there exists a constant $ C > 0 $ such that for any $ r \geq 1 $
		\begin{align*}
			J_{ r } ( \zeta_{2} ) \geq J_{ r } ( \zeta_{1} ) + C ( \zeta_{2} - \zeta_{1} )^{2} \quad \text{for each} ~ \zeta_{2} > \zeta_{1} .
		\end{align*}
	In particular the infimum $ J = \inf J_{r} $ is also strictly increasing.
\end{prop}

\subsection{The Large Deviations Principle and Properties of $J$} \label{S: large deviations explained} The large deviations principle, Proposition \ref{P: large deviations}, follows from standard techniques from the theory of large deviations.  The main argument is sketched next.  First, let $\Lambda$ be the (centered) logarithmic moment generating function of $\theta(x)$ for any given $x$, that is,
	\begin{align} \label{E: moment generating function}
		\Lambda(\xi) = \log \mathbb{E} [ \exp ( \xi \{ \theta(0) - \bar{\theta} \} ) ],
	\end{align}
where we can fix $x = 0$ by stationarity. To motivate the definition of $J_{r}$, we recall the role of $\Lambda$ in the large deviations principle for spatial averages of $\theta$, called Cram\'{e}r's Theorem (see \cite{dembo_zeitouni,Rassoul-Agha,varadhan}): For any $\zeta > 0$ and any $r > 0$,
	\begin{align*}
		\lim_{ \gamma \downarrow 0 } \gamma \log \mathbb{P} \left\{ \fint_{-r}^{r} \theta( \gamma^{-1} y) \, dy \leq \bar{\theta} - \zeta \right\} = -\sup \left\{ \xi \zeta - \Lambda(\xi) \, \mid \, \xi \leq 0 \right\}.
	\end{align*}

Similarly, for any sufficiently nice $u : \mathbb{R} \to [-1,1]$, it is possible to prove that
	\begin{align*}
		\lim_{ \gamma \downarrow 0 } \gamma \log \mathbb{P} \left\{ \mathscr{F}_{1,\gamma}(u; \mathbb{R}) \leq \bar{\sigma} - \zeta \right\} = -J_{u}(\zeta),
	\end{align*}
where $J_{u} : [0,\infty) \to [0,\infty]$ is the function given by 
	\begin{align} \label{E: Ju function}
		J_{u}(\zeta) = \sup \left\{ - \xi ( \zeta + \bar{\mathscr{F}}(u;\mathbb{R}) - \bar{\sigma} ) - \int_{-\infty}^{\infty} \Lambda( \xi W(u) ) \, dy \, \mid \, \xi \leq 0 \right\}.
	\end{align}
On the other hand, if $r \geq 1$ and $u$ is chosen so that $u(x) = -1$ for $x \leq -r$ and $u(x) = 1$ for $x \geq r$, then the assumption \eqref{A: triviality} on the boundary condition $q$ implies that $u$ is a competitor for $\sigma_{1,\gamma}$, and, thus,
	\begin{align*}
		\liminf_{ \gamma \downarrow 0 } \gamma \log \mathbb{P} \left\{ \sigma_{1,\gamma}(1,[-r,r]) \leq \bar{\sigma} - \zeta \right \} \geq \lim_{ \gamma \downarrow 0 } \gamma \log \mathbb{P} \left\{ \mathscr{F}_{1,\gamma}(u;\mathbb{R}) \leq \bar{\sigma} - \zeta \right\} = - J_{u}(\zeta).
	\end{align*}
In particular, after optimizing over $u$, this implies
	\begin{gather}
		\liminf_{ \gamma \downarrow 0 } \gamma \log \mathbb{P} \left\{ \sigma_{1,\gamma}(1,[-r,r]) \leq \bar{\sigma} - \zeta \right \} \geq - J_{r}(\zeta), \nonumber 
	\end{gather}
	where	%
	\begin{gather}
		J_{r}(\zeta) = \min \{ J_{u}(\zeta) \, \mid \, -1 \leq u \leq 1, \, \, \mathscr{F}_{\theta_{*}}(u;\mathbb{R}) < \infty, \, \, u(x) = \text{sgn}(x) \, \, \text{if} \, \, |x| \geq r \}. \label{E: definition of Jr}
	\end{gather}
If the set of functions $u$ we were minimizing over were finite, it would not be hard to see that the corresponding upper bound on the $\limsup$ also holds.  The proof of Proposition \ref{P: large deviations} uses compactness  to show that this is indeed the case.

The arguments that follow will require some properties of the moment generating function $\Lambda$ and the rate functions $\{ J_{r} \}$.  First, we use the fact that the boundedness of $\theta$ implies that $\Lambda$ has a bounded second derivative.

\begin{lemma}\label{lem:properties-lambda}
	If $ \Lambda $ is given by \eqref{E: moment generating function}, then
	\begin{align*}
		\xi \mapsto \Lambda ( \xi ) ~ \text{is strictly convex},
		\quad \text{and} \quad
		\xi \mapsto \Lambda ( \xi ) - \Big( \frac{ \theta^* - \theta_* }{ 2 } \Big)^2 \frac{ \xi^2 }{ 2 } ~ \text{is concave}.
	\end{align*}
	Moreover, we have $ \Lambda ( 0 ) = \Lambda ' ( 0 ) = 0 $ and thus
	\begin{align*}
		\Lambda ( \xi ) \leq \Big( \frac{ \theta^* - \theta_* }{ 2 } \Big)^2 \frac{ \xi^2 }{ 2 } .
	\end{align*}
\end{lemma}

\begin{proof}  This is all standard; we give some details for the reader's convenience.  Since $\theta(0)$ is bounded, we can differentiate to find
	\begin{align}
		\Lambda ' (\xi) &= \frac{ \mathbb{E} [ ( \theta(0) - \bar{\theta}  ) \exp ( \xi \{ \theta(0) - \bar{\theta} \} ) ] }{ \mathbb{E} [ \exp ( \xi \{ \theta(0) - \bar{\theta} \} ) ] } , \label{E: derivative formula} \\
		\Lambda '' (\xi) &= \frac{ \mathbb{E} [ ( \theta(0) - \bar{\theta}  )^2 \exp ( \xi \{ \theta(0) - \bar{\theta} \} ) ] }{ \mathbb{E} [ \exp ( \xi \{ \theta(0) - \bar{\theta} \} ) ] } - \Big( \frac{ \mathbb{E} [ ( \theta(0) - \bar{\theta}  ) \exp ( \xi \{ \theta(0) - \bar{\theta} \} ) ]  }{ \mathbb{E} [ \exp ( \xi \{ \theta(0) - \bar{\theta} \} ) ] } \Big)^2 . \nonumber
	\end{align}
	Notice that $ \Lambda '' ( \xi ) $ is the variance of $ \theta(0) - \bar\theta $ w.r.t.~to the tilted probability measure with density proportional to $ \exp ( \xi \{ \theta(0) - \bar{\theta} \} ) $. In view of the fact that $\theta$ is nonconstant, the bound $ \theta_* \leq \theta( 0 ) \leq \theta^* $, and the variational characterization of the variance $ { \rm Var } [ X ] = \inf_{ c } \E [ ( X - c )^2 ] $, we have	
	\begin{align*}
		0 < \Lambda '' ( \xi ) \leq \Big( \frac{ \theta^* - \theta_* }{ 2 } \Big)^2 .
	\end{align*}
The upper bound, sometimes called Hoeffding's Lemma, is well-known in the literature; see \cite[Lemma 2.2]{boucheron_lugosi_massart} or \cite[Section 2.4.1]{dembo_zeitouni}.

	In particular, we learn from the lower bound that $ \Lambda $ is strictly convex, and from the upper bound, that $ \Lambda ( \xi ) -  \big( \frac{ \theta^* - \theta_* }{ 2 } \big)^2 \frac{ \xi^2 }{ 2 } $ is concave. Moreover, the formulas \eqref{E: moment generating function} and \eqref{E: derivative formula} imply that $ \Lambda ( 0 ) = \Lambda ' ( 0 ) = 0  $ and thus by concavity
	\begin{align*}
		\Lambda ( \xi ) - \Big( \frac{ \theta^* - \theta_* }{ 2 } \Big)^2 \frac{ \xi^2 }{ 2 } \leq \Lambda ( 0 ) + \Lambda ' ( 0 ) \xi = 0 ,
	\end{align*}
	as claimed.
\end{proof}

As a consequence of the Laplace principle (in the terminology of statistical mechanics), we learn that the measure with density $ \exp ( \xi \{ \theta(0) - \bar{\theta} \} ) $ concentrates near the extremes $ \theta^* - \bar{\theta} $ and $ \theta_{*} - \bar{\theta} $ so that, by \eqref{E: derivative formula},
	\begin{align} \label{E: asymptotics derivative lambda}
		\Lambda ' ( \xi ) \rightarrow \theta^* - \bar{\theta} \quad \text{as} \quad \xi \nearrow \infty
		\quad \text{and} \quad
		\Lambda ' ( \xi ) \rightarrow \theta_{*} - \bar{\theta} \quad \text{as} \quad \xi \searrow -\infty .
	\end{align}
In particular, by monotonicity of $\Lambda'$ (recall that $ \Lambda $ is convex),
	\begin{align*}
		\theta_{*} - \bar{\theta} \leq \Lambda'(\xi) \leq \theta^{*} - \bar{\theta}.
	\end{align*}

Finally, the next lemma gives a qualitative description of the set on which $J_{u}$ is finite for a given $u$.
	
\begin{lemma}[Properties of $ J_{ u } $] \label{L: finiteness trick} Fix a $u : \mathbb{R} \to [-1,1]$ such that $\mathscr{F}_{\theta_{*}}(u;\mathbb{R}) < \infty$ and for which the function $x \mapsto u(x) \text{sgn}(x) - 1$ has compact support.  Let $ J_{ u } $ be the function defined by \eqref{E: definition of Ju}.  Then for any $\zeta > 0$, $J_{u}$ is finite in the set $[0,\zeta)$ if and only if
		\begin{align*}
			\mathscr{F}_{\theta_{*}}(u;\mathbb{R}) \leq \bar{\sigma} - \zeta.
		\end{align*}
	\end{lemma}

\begin{proof} Recall that, for any $\zeta \geq 0$, $J_{u}(\zeta) = \sup_{\xi \leq 0} g_{u}(\xi)$, where $g_{u}$ is given by
    \begin{equation*}
        g_{u}(\xi) = - \xi ( \zeta + \bar{\mathscr{F}}(u;\mathbb{R}) - \bar{\sigma} ) - \int_{-\infty}^{\infty} \Lambda(\xi W(u) ) \, dy.
    \end{equation*}
Since $\Lambda$ is convex and $W$ is nonnegative, $g_{u}$ is concave.
Notice that 
    \begin{align} \label{E: derivative}
       \frac{d}{d\xi} g_{u}(\xi) &= \bar{\sigma} - \zeta - \bar{\mathscr{F}}(u;\mathbb{R}) - \int_{-\infty}^{\infty} \Lambda'(\xi W(u)) W(u) \, dy \\
        &= \bar{\sigma} - \zeta - \tilde{\sigma}_{u}(\xi), \nonumber
    \end{align}
provided $\tilde{\sigma}_{u}$ is defined by
    \begin{align*}
        \tilde{\sigma}_{u}(\xi) = \int_{-\infty}^{\infty} \left( \frac{1}{2} |u'|^{2} + \{ \bar{\theta} + \Lambda'(\xi W(u) ) \} W(u) \right) \, dy.
    \end{align*}   
By the monotonicity of $\Lambda'$, $\tilde{\sigma}_{u}$ is nondecreasing.  Further, \eqref{E: asymptotics derivative lambda} implies
			\begin{align*}
				 \tilde{\sigma}(0) = \bar{\mathscr{F}}(u;\mathbb{R}) \geq \bar{\sigma}, \quad \lim_{\xi \to -\infty} \tilde{\sigma}(\xi) = \mathscr{F}_{\theta_{*}}(u;\mathbb{R}).
			\end{align*}
        Thus, if $\mathscr{F}_{\theta_{*}}(u;\mathbb{R}) > \bar{\sigma} - \zeta$, then \eqref{E: derivative} implies that $g_{u}(\xi) \to \infty$ as $\xi \to -\infty$, so $J_{u}(\zeta) = \infty$ in that case.  Conversely, suppose that $\mathscr{F}_{\theta_{*}}(u;\mathbb{R}) \leq \bar{\sigma} - \zeta$.  Let $t \in [0,\zeta)$.  Since $\mathscr{F}_{\theta_{*}}(u;\mathbb{R}) < \bar{\sigma} - t$, by continuity of $\tilde{\sigma}$, there is a $\xi_{*} \in (-\infty,0]$ such that $\tilde{\sigma}_{u}(\xi_{*}) = \bar{\sigma} - t$.  In view of \eqref{E: derivative} and concavity, this implies $\xi_{*}$ maximizes $g_{u}$.  In particular, $J_{u}(t) = g_{u}(\xi_{*}) < \infty$.
\end{proof}

\subsection{Proof of the Large Deviations Principle.} Before going into the proof, the next lemma gives the precise statement of the large deviations principle for $\mathscr{F}_{1,\gamma}(u;\mathbb{R})$ for a fixed $u$.  Here and henceforth, we use $\mathcal{U}_{r}$ to denote the relevant set of functions:
	\begin{align*}
		\mathcal{U}_{r} = \{ u : \mathbb{R} \to [-1,1] \, \mid \, \mathscr{F}_{\theta_{*}}(u;\mathbb{R}) <\infty, \quad u(x) = \text{sgn}(x) \quad \text{if} \, \, |x| \geq r \}
	\end{align*}

\begin{lemma} \label{L: basic LDP} For any $r > 0$, any $u \in \mathcal{U}_{r}$, and any $\zeta \geq 0$, 
		\begin{align} \label{E: LDP most basic form}
			\lim_{\gamma \downarrow 0} \gamma \log \mathbb{P} \{ \mathscr{F}_{1,\gamma}(u; \mathbb{R}) \leq \bar{\sigma} - \zeta \} = -J_{u}(\zeta),
		\end{align}
	where $J_{u} : [0,\infty) \to [0,\infty]$ is the nondecreasing function given by 
		\begin{align}\label{E: definition of Ju}
			J_{u}(\zeta) = \sup \left\{ - \xi ( \zeta + \bar{\mathscr{F}}(u; \mathbb{R}) - \bar{\sigma}) - \int_{-\infty}^{\infty} \Lambda ( \xi W(u) ) \, dy \, \mid \, \xi \leq 0 \right\}.
		\end{align}
	In particular, for any finite subset $\mathcal{U}_{ \rm f } $ of $ \cup_{ r > 0 } \, \mathcal{U}_{r}$,
		\begin{align*}
			\lim_{ \gamma \downarrow 0 } \gamma \log \mathbb{P} \{ \mathscr{F}_{1,\gamma} ( u ; \mathbb{R} ) \leq \bar{\sigma} - \zeta \, \, \text{for some} \, \, u \in \mathcal{U}_{ \rm f } \} = - \min \{ J_{u} \, \mid \, u \in \mathcal{U}_{ \rm f } \}.
		\end{align*}
	\end{lemma} 

Note that the lemma implies that $J_{u}(\zeta) \geq 0$, which can also be deduced from the definition \eqref{E: definition of Ju} upon recalling that $\Lambda(0) = 0$.  

		\begin{proof} As sketched in Section \ref{S: large deviations explained} above, the large deviations principle for $\mathscr{F}_{1,\gamma}(u;\mathbb{R})$ is an extension of Cram\'{e}r's Theorem, which can be proved by mimicking the proof of that theorem. (Alternatively, it can be deduced by applying the G\"{a}rtner-Ellis Theorem, see, e.g., \cite[Section 12.2]{Rassoul-Agha}.) The key observation is that, by independence, for any $\xi \in \mathbb{R}$,
        \begin{align*}
            & \log \mathbb{E} [ \exp ( \xi \gamma^{-1} ( \mathscr{F}_{1,\gamma}(u;\mathbb{R})  - \bar{\mathscr{F}}(u;\mathbb{R}))) ] \\
            & \qquad = \sum_{y' \in \gamma \mathbb{Z}} \log \mathbb{E} [ \exp ( \xi ( \theta(0) - \bar{\theta} ) \fint_{y'}^{y' + \gamma} W(u) \, dy ) ].
        \end{align*}
Upon multiplying by $\gamma$, sending $\gamma \downarrow 0$, and recalling the definition of $\Lambda$, this becomes
        \begin{align*}
        \lim_{ \gamma \downarrow 0 } \gamma  \log \mathbb{E} [ \exp ( \xi \gamma^{-1} ( \mathscr{F}_{1,\gamma}(u;\mathbb{R})  - \bar{\mathscr{F}}(u;\mathbb{R})) ) ] = \int_{-\infty}^{\infty} \Lambda ( \xi W(u) ) \, dy.
        \end{align*}
       Since $ \E [ \mathscr{F}_{ 1 , \gamma } ( u ; \mathbb{R} ) ] = \bar{\mathscr{F}}(u;\mathbb{R}) $, we can now conclude, for instance, by mimicking the standard proof of Cram\'{e}r's Theorem (see \cite[Section 5.11]{grimmett_stirzaker} or \cite[Section 2.4]{varadhan}). This completes the proof of \eqref{E: LDP most basic form}.
        
        Finally, here is the standard argument extending to a finite subset $\mathcal{U}_{ \rm f } \subseteq \cup_{r > 0} \mathcal{U}_{r}$: On the one hand, for any given $u_{*} \in \mathcal{U}_{ \rm f }$,
            \begin{align*}
                &\liminf_{\gamma \downarrow 0} \gamma \log \mathbb{P} \{ \mathscr{F}_{1,\gamma} ( u ; \mathbb{R} ) \leq \bar{\sigma} - \zeta \, \, \text{for some} \, \, u \in \mathcal{U}_{ \rm f } \} \\
                &\qquad\geq \liminf_{\gamma \downarrow 0} \gamma \log \mathbb{P} \{ \mathscr{F}_{1,\gamma} ( u_{*} ; \mathbb{R} ) \leq \bar{\sigma} - \zeta \} = -J_{u_{*}}(\zeta).
            \end{align*}  
        Thus, the liminf is at least as large as $-\min\{J_{u}(\zeta) \, \mid \, u \in \mathcal{U}_{ \rm f }\}$.   On the other hand, by the union bound,
            \begin{align*}
                &\limsup_{\gamma \downarrow 0} \gamma \log \mathbb{P} \{ \mathscr{F}_{1,\gamma} ( u ; \mathbb{R} ) \leq \bar{\sigma} - \zeta \, \, \text{for some} \, \, u \in \mathcal{U}_{ \rm f }  \}\\
                &\qquad \leq \lim_{\gamma \downarrow 0} \gamma \log \sum_{u \in \mathcal{U}_{ \rm f }} \mathbb{P} \{ \mathscr{F}_{1,\gamma}(u;\mathbb{R}) \leq \bar{\sigma} - \zeta \} = -\min\{J_{u}(\zeta) \, \mid \, u \in \mathcal{U}_{ \rm f }\} ,
            \end{align*}
            \end{proof} 
        
In view of the definition of $\mathcal{U}_{r}$, we can rewrite the definition \eqref{E: definition of Jr} of $J_{r}$ as follows
\begin{equation*}
    J_{r}(\zeta) = \inf \{ J_{u}(\zeta) \, \mid \, u \in \mathcal{U}_{r} \} .
\end{equation*}
In the proof of Proposition \ref{P: large deviations}, it will be useful to know that $J_{r}$ is left-continuous, as proved next.
 
	\begin{prop} \label{P: left-continuity}
		The function $J_{r}$ defined by \eqref{E: definition of Jr} is left-continuous in $[0,\infty)$.
	 \end{prop}

		\begin{proof} Since $J_{r}$ is increasing, it is enough to prove that $ J_{ r } $ is lower semicontinuous.
		
		From now on, consider $\mathcal{U}_{r}$ as a subset of $H^{1}((-r,r))$ with the weak topology.  First, notice that $ u \mapsto \bar{\mathscr{F}} ( u ; \mathbb{R} ) $ is lower semicontinuous,~and $ u \mapsto \int_{ -\infty }^{ \infty} \Lambda ( \xi W ( u ) ) \, dy $ is continuous. Therefore for any $ \xi \leq 0 $ the functions
		\begin{align*}
			( u , \zeta ) \mapsto - \xi \zeta - \xi (\bar{\mathscr{F}} ( u ; \mathbb{R} ) - \bar \sigma ) -  \int_{ -\infty }^{ \infty} \Lambda ( \xi W ( u ) ) \, dy
			\quad \text{are~lower semicontinuous}.
		\end{align*}
		By Lemma 2, if $u \in \mathcal{U}_{r}$ satisfies $J_{u}(\zeta) < \infty$ for some $\zeta > 0$, then $\mathscr{F}_{\theta_{*}} \leq \bar{\sigma}$.  Therefore, the set of such functions is weakly compact in $H^{1}((-r,r))$, so we can apply the direct method of the calculus of variations to see that
		\begin{align*}
			\zeta \mapsto J_{r} ( \zeta ) \stackrel{ \eqref{E: definition of Jr} }{ = } \inf \{ J_{ u } ( \zeta ) \, \mid \, u \in \mathcal{U}_{r} \}
			\quad \text{is lower semicontinuous}.
		\end{align*}
    \end{proof}
   
The large deviations principle for $\sigma_{1,\gamma}(1;[-r,r])$ now follows by using Lemma \ref{L: basic LDP} in conjunction with an approximation argument and left-continuity.

	\begin{proof}[Proof of Proposition \ref{P: large deviations}] First, we establish the large deviations principle: The discussion in Section \ref{S: large deviations explained} shows that 
    \begin{equation*}
        \liminf_{\gamma \downarrow 0 } \gamma \log \mathbb{P} \{ \sigma_{1,\gamma}(1;[-r,r]) \leq \bar{\sigma} - \zeta \} \geq - J_{r}(\zeta).
    \end{equation*}
    It remains to establish that
    \begin{equation*}
        \limsup_{\gamma \downarrow 0} \gamma \log \mathbb{P} \{ \sigma_{1,\gamma}(1;[-r,r]) \leq \bar{\sigma} - \zeta \} \leq - J_{r}(\zeta).
    \end{equation*}
   We claim this follows from the second half of Lemma \ref{L: basic LDP} and a deterministic approximation argument.
    
    Indeed, on the one hand, if $\mathcal{M} \subseteq \mathcal{U}_{r}$ is the set of all $u \in \mathcal{U}_{r}$ such that $\mathscr{F}_{1,\gamma}(u;\mathbb{R}) = \sigma_{1,\gamma}(1;[-r,r])$ for some realization of $\theta$ and some $\gamma > 0$, then $\mathcal{M}$ is precompact in $H^{1}((-r,r))$ with the norm topology.  (This follows from the Euler-Lagrange equation satisfied by minimizers, which implies a deterministic $L^{\infty}$ bound on the second derivative.)  On the other hand, due to the continuity of $W$ and the estimate 
        \begin{align*}
            | \mathscr{F}_{1,\gamma}(u;(-r,r)) - \mathscr{F}_{1,\gamma}(v;(-r,r)) | &\leq 4 ( \| u' \|_{L^{2}((-r,r))} + \|v'\|_{L^{2}((-r,r))} ) \| u' - v' \|_{L^{2}((-r,r))} \\
            &\qquad + \theta^{*} \int_{-r}^{r} | W(u) - W(v) | \, dy,
        \end{align*}
        $\mathscr{F}_{1,\gamma}$ is uniformly continuous on $\mathcal{M}$ with a deterministic modulus of continuity.  Therefore, for any $\nu > 0$, it is possible to find an $N \in \mathbb{N}$ and $\gamma$-independent (deterministic) functions $u_{1},\dots,u_{N} \in \mathcal{U}_{r}$ such that 
			\begin{align*}
				\sigma_{1,\gamma}(1; [-r,r]) \geq \min \left\{ \mathscr{F}_{1,\gamma}(u_{1};\mathbb{R}), \dots, \mathscr{F}_{1,\gamma}(u_{N};\mathbb{R}) \right\} - \nu.
			\end{align*}
        Thus, by Lemma \ref{L: basic LDP},
			\begin{align*}
				\limsup_{ \gamma \downarrow 0 } \gamma \log \mathbb{P} \{ \sigma_{\epsilon,\delta}(1;[-r,r]) \leq \bar{\sigma} - \zeta \} &\leq - \min \{ J_{u_{1}}(\zeta - \nu),\dots,J_{u_{N}}(\zeta - \nu) \} \\
                &\leq - J_{r}(\zeta - \nu).
			\end{align*}
        Since $J_{r}$ is left-continuous, see Proposition \ref{P: left-continuity}, we conclude upon sending $\nu \downarrow 0$.

        Next, we gather properties of $ J_r $ and the limit $ J $: By definition, if $s \geq r \geq 1$, then $\mathcal{U}_{r} \subseteq \mathcal{U}_{s}$, so $J_{r}(\zeta) \geq J_{s}(\zeta)$ holds pointwise.  As in the statement, let $J : [0,\infty) \to [0,\infty]$ be the infimum
        \begin{equation}\label{eqn:defn-j-alternative}
            J(\zeta) = \inf \left\{ J_{r}(\zeta) \, \mid \, r \geq 1 \right\} = \lim_{ r \to \infty } J_{r}(\zeta).
        \end{equation}
       	Proposition \ref{cor:monotonicity-jr} immediately implies property (iii), i.e.~that the limit $ J $ is strictly increasing in $[0,\bar{\sigma}-\sigma_{*}]$.
	
        Property (ii) of the limit $ J $ follows from the first half of Lemma \ref{L: finiteness trick}. If $\zeta \in (0,\bar{\sigma} - \sigma_{*})$, then $\sigma_{*} <\bar{\sigma} - \zeta$ so the formula \eqref{E: derivative} characterizing $\sigma_{*}$ implies there is an $r > 0$ and a $u \in \mathcal{U}_{r}$ such that $\mathscr{F}_{\theta_{*}}(u;\mathbb{R}) \leq \bar{\sigma} - \zeta$.  By Lemma \ref{L: finiteness trick}, for any $\zeta' \in [0,\zeta)$, we have $J(\zeta') \leq J_{u}(\zeta') < \infty$.  This proves $J$ is finite in $[0,\bar{\sigma} - \sigma_{*})$.  
        Conversely, for any $r > 0$ and $u \in \mathcal{U}_{r}$, we know that $\mathscr{F}_{\theta_{*}}(u;\mathbb{R}) \geq \sigma_{*}$ and, thus, $J_{u}(\zeta) = \infty$ for any $\zeta > \bar{\sigma} - \sigma_{*}$ by Lemma \ref{L: finiteness trick}.  This proves $J = \infty$ in $(\bar{\sigma}-\sigma_{*},\infty)$.

        Finally, we prove property (i) of $J$, that is, $J(0) = 0$ and $J$ is continuous at zero.  To this end, since $J$ is nonnegative and nondecreasing, it suffices to prove that if $\{\zeta_{n}\}_{n \in \mathbb{N}}$ is any sequence of positive numbers such that $\zeta_{n} \to 0$ as $n \to \infty$, then $J(\zeta_{n}) \to 0$. By the formula \eqref{eqn:bar-sigma} determining $\bar{\sigma}$ and a standard compactness argument, we can fix a function $\bar{q} : \mathbb{R} \to [-1,1]$ such that $\bar{q}(\pm \infty) = \pm 1$ and $\bar{\mathscr{F}}(\bar{q};\mathbb{R}) = \bar{\sigma}$.  Let $\{u_{n}\}_{n \in \mathbb{N}} \subseteq \bigcup_{r > 0} \mathcal{U}_{r}$ be a sequence such that $u_{n} \to \bar{q}$ in $L^{\infty}(\mathbb{R})$, $u'_{n} \to \bar{q}'$ in $L^{2}(\mathbb{R})$, and $W(u_{n}) \to W(\bar{q})$ in $L^{1}(\mathbb{R})$.  Since $\mathscr{F}_{\theta_{*}}(u_{n};\mathbb{R}) \to \mathscr{F}_{\theta_{*}}(\bar{q};\mathbb{R})$ as $n \to \infty$ and $\mathscr{F}_{\theta_{*}}(\bar{q};\mathbb{R}) < \bar{\mathscr{F}}(\bar{q};\mathbb{R}) = \bar{\sigma}$, there is an $N \in \mathbb{N}$ such that 
            \begin{align*}
                \mathscr{F}_{\theta_{*}}(u_{n};\mathbb{R}) \leq \bar{\sigma} - 2\zeta_{n} \quad \text{for each} \, \, n \geq N. 
            \end{align*}
        Therefore, by Lemma \ref{L: finiteness trick} and its proof, for each $n \geq N$, $J_{u_{n}}(\zeta_{n}) < \infty$ and there is a $\xi_{n} \in (-\infty,0]$ such that 
        	\begin{align}
		J_{u_{n}}(\zeta_{n}) = - \xi_{n} ( \zeta_{n} + \bar{\mathscr{F}}(u_{n}; \mathbb{R}) - \bar{\sigma} ) - \int_{-\infty}^{\infty} \Lambda ( \xi_{n} W(u_{n}) ) \, dy  \label{E: analyze J at zero} \\
		\text{and} \quad \zeta_{n} + \bar{\mathscr{F}}(u_{n};\mathbb{R}) - \bar{\sigma} = - \int_{-\infty}^{\infty} \Lambda'( \xi_{n} W(u_{n}) ) W(u_{n}) \, dy. \nonumber
	\end{align}
Up to extracting a subsequence, there is no loss of generality in assuming that $\xi_{n} \to \xi_{*}$ for some $\xi_{*} \in [-\infty,0]$, which implies, by the dominated convergence theorem (and an application of \eqref{E: asymptotics derivative lambda} if $ \xi_* = - \infty $),
	\begin{align*}
		0 = \bar{\mathscr{F}}(\bar{q}; \mathbb{R}) - \bar{\sigma} &= \lim_{n \to \infty} \left( \zeta_{n} + \bar{\mathscr{F}}(u_{n};\mathbb{R}) - \bar{\sigma} \right) \\
		&= - \lim_{ n \to \infty } \int_{-\infty}^{\infty} \Lambda'(\xi_{n} W(u_{n}) ) W(u_{n}) \, dy = - \int_{-\infty}^{\infty} \Lambda'(\xi_{*} W(\bar{q}) ) W(\bar{q}) \, dy.
	\end{align*}
Since $\xi_{*} \leq 0$, $\Lambda'(0) = 0$, and $\Lambda'$ is strictly increasing by Lemma \ref{lem:properties-lambda}, this implies $\xi_{*} = 0$.  Upon invoking \eqref{E: analyze J at zero} and recalling that $\Lambda$ is nonnegative, we conclude
	\begin{align*}
		\limsup_{n \to \infty} J(\zeta_{n}) \leq \limsup_{n \to \infty} J_{u_{n}}(\zeta_{n}) \leq \limsup_{n \to \infty} \left( - \xi_{n} ( \zeta_{n} + \bar{\mathscr{F}}(u_{n};\mathbb{R}) - \bar{\sigma} ) \right) = 0.
	\end{align*}
       \end{proof}  

\begin{remark} The argument used in the previous proof can also be used to prove the following formula for $J$:
	\begin{align*}
		J(\zeta) = \inf \left\{ J_{u}(\zeta) \, \mid \, u : \mathbb{R} \to [-1,1], \, \, \bar{\mathscr{F}}(u;\mathbb{R}) < \infty, \, \, u(\pm \infty) = \pm 1 \right\}.
	\end{align*}\end{remark}
        
\subsection{Proof of Proposition \ref{P: lower deviations}} We now show that the limiting surface tension is no greater than $\bar{\sigma} - \zeta$ provided the scaling $\epsilon \mapsto \delta(\epsilon)$ is chosen in such a way that \eqref{E: assumption about delta} holds. 
	
\begin{proof}[Proof of Proposition \ref{P: lower deviations}]
	Since \eqref{E: assumption about delta} is stable under changing $ \delta( \epsilon ) $ and $ \epsilon $ by a multiplicative constant, we can assume w.l.o.g.~that $ \varrho = 1 $.
	Thus under the change of variables \eqref{eqn:change-of-variables} we have 
	\begin{align*}
		\sigma_{ \epsilon, \delta ( \epsilon ) } ( 1; [- 1, 1 ] )
		&= \sigma_{ 1 , \gamma ( R ) } ( 1  ; [ - R , R ] ) .
	\end{align*}
	Moreover, due to Remark \ref{rmk: triviality} we can assume that the boundary datum satisfies \eqref{A: triviality}, i.e.~$ q ( y ) = \pm 1 $ if $ \pm y \geq 1 $. Thus we have the monotonicity property
	\begin{gather}
		\sigma_{ 1 , \gamma } ( 1; [ - R , R ] ) \leq \sigma_{ 1 , \gamma } ( 1; I ) \nonumber \\
		 \text{for any interval } I \subseteq [ - R , R ]  ~ \text{of width at least } 2 . \nonumber
	\end{gather}
	Therefore it is enough to establish the following: There exists an $ r_0 = r_0 ( \zeta ) \geq 1 $ such that for any $ r \geq r_0 $ we have
	\begin{align}\label{varm02}
		\lim_{ R \uparrow \infty } \mathbb{P} \big\{  \sigma_{ 1 , \gamma(R) } ( 1;  [ y - r , y + r ] ) \leq \bar\sigma - \zeta ~ \text{for some} ~ y \in [ - R + r , R - r  ] \big\} = 1
	\end{align}
	provided $ \gamma(R), R $ are defined by \eqref{eqn:change-of-variables} with $ \epsilon $, $ \delta $ satisfying \eqref{E: assumption about delta}.
	
	To this end, it is convenient to consider the Bernoulli random variables
	\begin{align*}
		X_{ k } = \boldsymbol{1}_{ \{ \sigma_{ 1 , \gamma(R) } ( 1 ; [ k - r , k + r ] ) \leq \bar\sigma - \zeta \} }
		\quad \text{and their sum} \quad
		S = \sum_{ k \in r \mathbb{Z} \cap [ - R + r , R - r  ] } X_k .
	\end{align*}
	For ease of notation we do not indicate their dependence on $ R $ explicitly. The idea is that, up to direct neighbours, the sum is taken over independent variables, and therefore to leading order behaves like its expectation. More specifically we establish \eqref{varm02} by using assumption \eqref{E: assumption about delta} to verify that
	\begin{align}\label{varm01}
		\mathbb{E} [ S ] \rightarrow \infty ~ \text{as} ~ R \nearrow \infty
		\quad \text{alongside} \quad
		{ \rm Var } [ S ] \lesssim \mathbb{E} [ S ].
	\end{align}
	This suffices to establish \eqref{varm01} due to the elementary Chebyshev-type bound
	\begin{align*}
		\mathbb{P} \big\{ S \leq 1 \big\} 
		\leq \mathbb{P} \big\{  | S - \mathbb{E} [ S ] | \geq \frac{ 1 }{ 2 } \mathbb{E} [ S ] \big\} 
		\lesssim \frac{ { \rm Var } [ S ] }{ ( \mathbb{E} [ S ] )^2 }
		\stackrel{ \eqref{varm01} }{ \rightarrow } 0 
		~ ( \text{as} ~ R \nearrow \infty ) ,
	\end{align*}
which, in view of the definition of $ S $, is equivalent to \eqref{varm02}.
        
To establish the first item in \eqref{varm01} notice that the random variables $ X_k $ have the same law and thus
	\begin{align*}
		\mathbb{E} [ S ]
		\sim \frac{ R }{ r } \, \mathbb{P} \big\{ \sigma_{ 1 , \gamma } ( 1 ; [ - r , r ] ) \leq \bar\sigma - \zeta \big\}   .
	\end{align*}
	To lower bound the probability on the right-hand side, we notice that by \eqref{eqn:defn-j-alternative} and due to assumption \eqref{E: assumption about delta}, we can find $ r_0 = r_0 ( \zeta ) \geq 1 $ and $ \nu = \nu ( \zeta ) \ll 1 $ such that
	\begin{align*}
		J_r ( \zeta ) 
		< ( 1 + \nu )  J_r ( \zeta ) < \liminf_{ \epsilon \downarrow 0} \frac{ \delta(\epsilon) }{ \epsilon } | \log \epsilon |
		\quad \text{for all} \quad
		r \geq r_0 .
	\end{align*}
	Due to the first inequality we may use the large deviation principle in Proposition \ref{P: large deviations} to obtain the lower bound
	\begin{align*}
		\mathbb{E} [ S ] \gtrsim \frac{ R }{ r } \exp \Big( - ( 1 + \nu )  J_{ r } ( \zeta ) \frac{ 1 }{ \gamma(R) } \Big) 
		\quad \text{provided} \quad
		R \gg_{ \nu, r } 1 .
	\end{align*}
	Thus by the choice of $ \nu $, \eqref{E: assumption about delta} and \eqref{eqn:limit-condition-ed-to-gr} imply the claimed \eqref{varm01}.
	
	Let us now focus on the second item in \eqref{smm02}. We start from the observation that the random variable $ X_k $ depends only on the rescaled field $ \theta ( \gamma^{-1} \cdot )$ restricted to the interval $  [ k - r , k + r ]  $. Therefore, since $ \theta$ has range of dependence equal to $1$, we learn that
	\begin{align*}
		X_{ j } ~ \text{and} ~ X_k ~ \text{are independent if~} | k - j | > 2 ( r + \gamma ) .
	\end{align*}
	Using Cauchy-Schwarz and Young, and also $ \gamma(R) \leq r $ (for $R$ large enough), we thus learn
	\begin{align*}
		{ \rm Var } [ S ]
		= \sum_{ \substack{ k , j  \in r \mathbb{Z} \cap [ - R + r , R -  r  ] \\ | k - j | \leq 2 ( r + \gamma(R) ) } } { \rm Cov } [ X_k , X_j ]
		\lesssim  \sum_{ k \in r \mathbb{Z} \cap [ - R + r , R - r  ] } { \rm Var } [ X_k ]
	\end{align*}
	Since for any Bernoulli variable we have $ { \rm Var } [ X_k ] \leq \mathbb{E} [ X_k ] $, we have shown the second item in \eqref{smm02}.
\end{proof}

\subsection{Proof of Proposition \ref{P: deviations 1d checkerboard}}\label{S: deviations 1d checkerboard}

As announced in the introduction it is easy to adapt the proof of Theorem \ref{T: main deviations} to the situation of Proposition \ref{P: deviations 1d checkerboard}. More specifically, we follow the strategy of Proposition \ref{P: lower deviations} presented in the previous subsection.

Beforehand, recall that for one-dimensional media we have
\begin{align*}
    \bar a = ( \E[ a^{ - 1 } ( 0 ) ] )^{ - 1 } .
\end{align*}
In particular, by Jensen's inequality, if $ a $ is nonconstant, then $ \bar a > \lambda $.

\begin{proof}[Proof of Proposition \ref{P: deviations 1d checkerboard}]
	As in the proof of Proposition \ref{P: lower deviations} we assume w.l.o.g.~that $ \varrho = 1 $ and $ q ( y ) = \pm 1 $ if $ \pm y \geq 1 $. Since we can always lower bound
	\begin{align*}
		\min \Big\{ \int_{-R}^{R} \Big( \frac{1}{2}\lambda |u'|^{2} + \theta_{*} W(u) \Big) \, dy \, \mid \, u(\pm R) = q(\pm R) \Big\}
		\leq \sigma_{ \epsilon, \delta ( \epsilon ) } ( 1; [ - 1, 1 ] ) ,
	\end{align*}
	and the lower bound converges to $  \sigma_{W} \sqrt{ \theta_{*} \lambda } $ as $ R \rightarrow \infty $, cf.~\eqref{E: lower energy variational formula}, the claim of Proposition \ref{P: deviations 1d checkerboard} reduces to proving
	\begin{align*}
		\lim_{ \epsilon \downarrow 0 } \mathbb{P} \big\{ \sigma_{ \epsilon, \delta ( \epsilon ) } ( 1 ; [ - 1, 1 ] ) \leq ( 1 + \nu ) \sigma_{W} \sqrt{ \theta_{*} \lambda } \big\} = 1
		\quad \text{for any } \nu > 0 .
	\end{align*}
	Now the difference compared to the proof of Proposition \ref{P: lower deviations} is simply to replace the sufficient condition \eqref{varm02} by the statement
	\begin{align}
		\lim_{ R \uparrow \infty } \mathbb{P} \big( \bigcup_{ y \in [ -R + r, R - r ] } \Omega_y^{ \gamma(R) }  \big) = 1 ,
	\end{align}
    where $ \Omega_y^{ \gamma } $ denotes the event
    \begin{align*}
        \Omega_{ y }^{ \gamma } = \{  a ( \gamma^{-1} \cdot ) |_ { [ y - r , y + r ] } \leq ( 1 + \nu ) \lambda ~ \text{and} ~ \theta ( \gamma^{-1} \cdot ) |_ { [ y - r , y + r ] } \leq ( 1 + \nu )  \theta_* \}
    \end{align*}
	This is simpler to prove than \eqref{varm02} since it does not rely on a large deviations principle but merely the elementary
	\begin{align}
		\mathbb{P} (  \Omega_0^{ \gamma } )
		\sim p^{ \frac{ r }{ \gamma } }
		\quad \text{for~some} ~ p = p(\nu) > 0 ~ \text{and} ~ r \gg 1,
	\end{align}
	which follows from the fact that $\lambda$ and $\theta_{*}$ are the essential infima of $a(y)$ and $\theta(y)$, respectively,
	together with the observation that
	\begin{align}
		\lim_{ R \uparrow \infty } R p^{ \frac{ r }{ \gamma(R) } } = \infty 
		\quad \text{for~any } p > 0,
	\end{align}
	by our assumption and the change-of-variables \eqref{eqn:change-of-variables}.
\end{proof}

\subsection{Proof of Proposition \ref{P: upper deviations}} Next, we obtain a lower bound on the limiting surface tension under assumption \eqref{eqn:lb-helper-minus-two}. In the proof, we use the following lemma about the shape of minimizers.
	
	\begin{lemma} \label{L: topology lemma} Fix $ E_{0} > 0$. For any $\nu \in (0,1)$, there is an $ r = r (\nu, E_{0}) > 0$ with the following property: If $u : [-R,R] \to [-1,1]$ has $u(R) = 1$, $u(-R) = -1$, and $\mathscr{F}_{\theta^{*}}(u; [-R,R]) \leq E_{0}$, then there is an interval $[a,b] \subseteq [ -R, R ]$ such that 
		\begin{align*}
			u(a) = -1 + \nu, \quad u(b) = 1 - \nu, \quad \text{and} \, \, \quad |b - a| \leq r.
		\end{align*}
	\end{lemma} 
	
		\begin{proof} By continuity, since $u(-R) = -1$ and $u(R) = 1$, there are points $a < b$ such that 
			\begin{gather*}
				-1 + \nu < u(x) < 1 - \nu \quad \text{for each} \, \, x \in (a,b), \\
				u(a) = -1 + \nu, \quad u(b) = 1 - \nu.
			\end{gather*}
		This implies
			\begin{align*}
				|b - a| \min \left\{ W(u) \, \mid \, -1 + \nu \leq u \leq 1 - \nu \right\} \leq \mathscr{F}_{\theta^{*}}(u;[-R,R]) \leq E_{0}.
			\end{align*}
		So if $ r =  E_{0}  / \min \{ W(u) \, \mid \, -1 + \nu \leq u \leq 1 - \nu \} $, we learn that $ |b - a| \leq r $ as claimed.
	\end{proof}
		
\begin{proof}[Proof of Proposition \ref{P: upper deviations}]
		
			As in the proof of Proposition \ref{P: lower deviations} we assume w.l.o.g.~that $ \varrho = 1 $ and that the boundary datum satisfies $ q ( y ) = \pm 1 $ if $ \pm y \geq 1 $, see Remark \ref{rmk: triviality}.
			
			Our first goal is to prove a converse to the monotonicity property of $  \sigma_{ 1 , \gamma } ( 1;  \cdot )  $ used in the proof of Proposition \ref{P: lower deviations}, namely, for every $ \nu_{0} > 0 $ there exists an $ r_{*} = r_{*} ( \nu_{0} )  \geq 2 $ such that for any $ R \geq 1 $
			\begin{gather}
				\sigma_{ 1 , \gamma } ( 1; I ) \leq \sigma_{ 1 , \gamma } ( 1; [ - R , R ] ) + \nu_{0} \label{eqn:lb-helper} \\
				\text{for some interval } I \subseteq [ -R, R ] ~ \text{with width between } 2 \text{ and }  r_{*} . \nonumber
			\end{gather}
			Our second goal is to establish that under our assumption \eqref{eqn:lb-helper-minus-two} there exists an $ r_0 = r_0 ( \zeta ) \geq 1 $ and a $\nu_{0} = \nu_{0} ( \zeta ) > 0$ such that for any $ r \geq r_0 $ we have
			\begin{align}\label{smm01}
				\lim_{ R \uparrow \infty } \mathbb{P} \big\{  \sigma_{ 1 , \gamma(R) } ( 1;  [ y - r , y + r ] ) \geq \bar{\sigma} - (\zeta - \nu_{0}) ~ \text{for all} \, \, y \in r \mathbb{Z} \cap  [- R  , R ] \big\} = 1 .
			\end{align}
			provided $ R, \gamma(R) $ are defined by \eqref{eqn:change-of-variables} with $ \epsilon $, $ \delta $ satisfying \eqref{eqn:lb-helper-minus-two}. 
			
			We complete the proof using \eqref{eqn:lb-helper} and \eqref{smm01} as follows: If $R \geq r/2$ and if $I$ is as in \eqref{eqn:lb-helper}, then we can fix a $y \in r \mathbb{Z} \cap [-R,R]$ such that $ I \subseteq [ y - r , y + r ] $, so that, by the monotonicity argument from Proposition \ref{P: lower deviations},
				\begin{align*}
					\sigma_{1,\gamma(R)}(1;[y-r,y+r]) \leq \sigma_{1,\gamma(R)}(1;I) \leq \sigma_{1,\gamma(R)}(1;[-R,R]) + \nu_{0}
				\end{align*}
		Therefore, \eqref{smm01} implies that $\sigma_{1,\gamma(R)}(1;[-R,R]) \geq \bar{\sigma} - \zeta$ with probability tending to one as $R \nearrow \infty$.
			
			\textit{Argument for \eqref{eqn:lb-helper}}: To establish \eqref{eqn:lb-helper} let us first pick some small $ \nu > 0 $. Due to the boundary datum of $  \sigma_{ 1 , \gamma } ( 1;  [ - R , R ] ) $ we know from Lemma \ref{L: topology lemma} that there exists an $ r_{*} > 0 $ such that for any minimizer $ u $ of $  \sigma_{ 1 , \gamma } ( 1; [ - R , R ] ) $ there exists an interval $ I = [ a, b ] $ of width at most $r_{*}$ such that $ u ( a ) = - 1 + \nu $ and $ u ( b )= 1 - \nu $. We use $ u $ to build a competitor $ w $ by setting
			\begin{align*}
				w ( x ) &= \left\{ \begin{aligned}
					- 1 \quad & \text{if} ~ x \leq \max \{ -R , a - 1 \} , \\
					u ( x ) \quad & \text{if} ~ a \leq x \leq b , \\
					+1 \quad & \text{if} ~ \min \{ R, b + 1 \}  \leq x .
				\end{aligned}\right.
			\end{align*}
			In either of the remaining two intervals $  ( a - 1 , a ) \cap ( - R , R )  $ and $ ( b , b + 1 ) \cap ( - R , R ) $ we set $ w = u $ if the interval has width less than $1$, and otherwise extend $ w $ by linear interpolation. Thus, from estimating the linear interpolation part, we learn that
			\begin{align*}
				\mathscr{F}_{ \gamma } ( w ; ( - R , R ) )
				- \mathscr{F}_{ \gamma } ( u ; \{ u = w \} )
				\lesssim \nu^2 +  \sup_{ 1 - \nu \leq | u | \leq 1 } W ( u ) .
			\end{align*}
			Using the boundary data of  $ w $ to lower bound $ \mathscr{F}_{ \gamma } ( w ; ( - R , R ) ) $ and the estimate $  \mathscr{F}_{ \gamma } ( u ; \{ u = w \} ) \leq \sigma_{ 1 , \gamma } ( 1 ; [ - R , R ] ) $ (by the choice of $ u $), we find
			\begin{align*}
				 \sigma_{ 1 , \gamma } ( 1; ( a - 1 , b + 1 ) ) \cap [ -R , R ] ) 
				 -  \sigma_{ 1 , \gamma } ( 1; [ - R , R ] )
				 \lesssim \nu^2 +  \sup_{ 1 - \nu \leq | u | \leq 1 } W ( u ) ,
			\end{align*}
			which establishes \eqref{eqn:lb-helper} upon choosing $ \nu $ such that the error is at most $ \nu_{0} $.
			
			\textit{Argument for \eqref{smm01}}: To establish \eqref{smm01}, we first choose $\nu_{0}$ and then proceed as in the argument for \eqref{varm02}.  Since $J$ is left-continuous by Proposition \ref{P: left-continuity}, the present assumption on $J(\zeta)$ implies there is a $\nu_{0} \in (0, \zeta)$ such that 
				\begin{align} \label{eqn:lb-helper-minus-two-again}
					\limsup_{ \epsilon \downarrow 0 } \frac{ \delta(\epsilon) |\log \epsilon| }{ \epsilon } < J ( \zeta - \nu_{0} ).
				\end{align}
			Next, we consider the complementary event, and as before, the Bernoulli random variables
			\begin{align*}
				X_{ k } = \boldsymbol{1}_{ \{  \sigma_{ 1 , \gamma(R) } ( 1; [ k - r , k + r ] ) \leq \bar{\sigma} - ( \zeta - \nu_{0} ) \} } 
				\quad \text{and their sum} \quad
				S = \sum_{ k \in r \mathbb{Z} \cap [ - R , R  ] } X_k .
			\end{align*}
			We establish \eqref{smm01} by showing that
			\begin{align}\label{smm02}
				 \mathbb{E} [ S ] \rightarrow 0 ~ ( \text{as} ~ R \nearrow \infty )
				 \quad \text{and} \quad
				 { \rm Var } [ S ] \lesssim \mathbb{E} [ S ] .
			\end{align}
			Indeed, \eqref{smm02} shows that $ S \rightarrow 0 $ in probability as $ R \nearrow 0 $, which is equivalent to \eqref{smm01}.
			
			We conclude similarly to the argument of \eqref{varm02} in the proof of Proposition \ref{P: lower deviations}: As above the first item in \eqref{smm02} follows from the observation that all $ X_k $ have the same law, so that by the large deviation principle in Proposition \ref{P: large deviations} we know that, for any $\nu \in (0,1)$,
			\begin{align*}
				\mathbb{E} [ S ] 
				\sim \frac{ R }{ r } \, \mathbb{P} \big\{  \sigma_{ 1 , \gamma } ( 1;  [ - r , r ] ) \leq \bar{\sigma} - ( \zeta - \nu_{0} ) \big\} 
				\leq  \frac{ R }{ r } \exp \Big( - ( 1 - \nu ) J_{ r } ( \zeta - \nu_{0} ) \frac{ 1 }{ \gamma } \Big) 
			\end{align*}
			provided $R \gg_{r,\nu} 1$.  Since by \eqref{eqn:lb-helper-minus-two-again} we can choose $\nu = \nu(\zeta - \nu_{0}) \in (0,1)$ and $r_{0} = r_{0}(\zeta-\nu_{0}) > 0$ such that 
				\begin{align*}
					\limsup_{ \epsilon \downarrow 0 } \frac{ \delta(\epsilon) |\log \epsilon| }{ \epsilon } < ( 1 - \nu ) J_{r} ( \zeta - \nu_{0} ) \quad \text{for each} \, \, r \geq r_{0},
				\end{align*}
			we can again use \eqref{eqn:limit-condition-ed-to-gr} to deduce the first item of \eqref{smm02}. The second item of \eqref{smm02} follows by the same reasoning as the second item of \eqref{varm01} earlier on.
		\end{proof}

\subsection{Strict Monotonicity of $J_{r}$} The proofs of Proposition \ref{P: lower deviations} and \ref{P: upper deviations} used the fact that $J_{r}$ is strictly increasing.  This is proved next using the strict convexity of the functions $\{ J_{u} \}$.

	\begin{lemma} For any $r > 0$ and any $u \in \mathcal{U}_{r}$ with $\mathscr{F}_{\theta_{*}}(u;\mathbb{R}) \leq \bar{\sigma}$, the function
		\begin{gather} \label{eqn:convexity-ju}
			\zeta \mapsto J_{ u } ( \zeta ) - \frac{ 1 }{ 2 C } \zeta^2
		~ \text{is~convex~in~}[0,\infty),
		\end{gather}
	where $C = \theta_{*}^{-1} \left( \frac{ \theta^{*} - \theta_{*} }{ 2 } \right)^{2} \|W\|_{L^{\infty}([-1,1])} \bar{\sigma}$. \end{lemma}
	
		\begin{proof} Here it is convenient to define $ f ( \xi ) =  \xi ( \bar{\mathscr{F}}(u; \mathbb{R}) - \bar{\sigma}) + \int_{-\infty}^{\infty} \Lambda ( \xi W(u) ) \, dy  $, so that by \eqref{E: definition of Ju} the function $ J_{ u } ( - \zeta ) $ is the Fenchel-Legendre transform of $ f $. In particular $ J_{ u } $ is convex, and a standard argument shows that the semiconcavity of $ f $ implies strict convexity of the transform: From Lemma \ref{lem:properties-lambda} we know that
		\begin{align*}
			f ( \xi ) - \frac{ C_{ u } }{ 2 } \xi^2
			\quad \text{is concave for} \quad
			C_{ u } = \Big( \frac{ \theta^* - \theta_* }{ 2 } \Big)^2 \int_{ - \infty }^{ \infty } W( u )^2  \, dy .
		\end{align*}
		With this constant $ C_{ u } $ we use definition \eqref{E: definition of Ju} to compute
		\begin{align*}
			J_{ u } ( \zeta ) - \frac{ 1 }{ 2  C_{ u } } \zeta^2 
			& \stackrel{ \eqref{E: definition of Ju} }{ = } \sup_{ \xi \leq 0 } \Big\{ - \xi \zeta - f ( \xi ) - \frac{ 1 }{ 2 C_{ u } } \zeta^2  \Big\} \\
			&=  \sup_{ \xi \leq 0 } \Big\{ - \frac{ 1 }{ 2 } \big( \frac{ \zeta }{ \sqrt{ C_{ u } } } + \sqrt{ C_{ u } } \xi \big)^2 - \big( f ( \xi ) - \frac{  C_{ u } }{ 2 } \xi^2 \big)  \Big\} .
		\end{align*}
		Therefore, since the supremum is taken over convex functions, we learn that
		\begin{align*}
			J_{ u } ( \zeta ) - \frac{ 1 }{ 2 C } \zeta^2 
			\quad \text{is convex for any} \quad 
			C \geq C_{ u } = \Big( \frac{ \theta^* - \theta_* }{ 2 } \Big)^2 \int_{ - \infty }^{ \infty } W( u )^2  \, dy .
		\end{align*}
		Since $\int_{-\infty}^{\infty} W(u)^{2} \, dy \leq \theta_{*}^{-1} \|W\|_{L^{\infty}([-1,1])} \mathscr{F}_{\theta_{*}}(u;\mathbb{R})$, we can take $C$ to be the constant $C = \theta_{*}^{-1} \left( \frac{ \theta^{*} - \theta_{*} }{ 2 } \right)^{2} \|W\|_{L^{\infty}([-1,1])} \bar{\sigma}$, as claimed.
	\end{proof}

\begin{proof}[Proof of Proposition \ref{cor:monotonicity-jr}]
	First, recall from the proof of Lemma \ref{L: finiteness trick} that if $\mathscr{F}_{\theta_{*}}(u;\mathbb{R}) > \bar{\sigma}$, then $J_{u} \equiv \infty$.  Therefore, we can add the energy constraint to $J_{r}$ as follows:
		\begin{equation*}
			J_{r}(\zeta) = \inf \{ J_{u}(\zeta) \, \mid \, u \in \mathcal{U}_{r}, \, \, \mathscr{F}_{\theta_{*}}(u;\mathbb{R}) \leq \bar{\sigma} \} \quad \text{for any} \, \, \zeta \geq 0.
		\end{equation*}
	Next, notice that, for such a $u$, the convexity of $ J_{ u } ( \zeta ) - \frac{ 1 }{ 2 C } \zeta^2 $ is equivalent to
	\begin{align*}
		J_{ u } ( ( 1 - t ) \zeta_1 + t \zeta_2 ) + \frac{ 1 }{ 2 C } t ( 1 - t ) ( \zeta_2 - \zeta_1 )^2  
		\leq (1 - t ) J_{ u } ( \zeta_1 ) + t J_{ u } ( \zeta_2 ) 
		~ \text{for} ~ 0 < t < 1 .
	\end{align*}
	Moreover, since the term $ - \xi \zeta $ in definition \eqref{E: definition of Ju} is nondecreasing in $ \zeta $ the function $ J_{ u } $ is nondecreasing, so that $ J_{ u } ( \min \{ \zeta_1, \zeta_2 \} ) \leq J_{ u } ( ( 1 - t ) \zeta_1 + t \zeta_2 ) $. Together with the above inequality this implies
	\begin{align*}
		\frac{ 1 }{ 2 C } ( 1 - t ) ( \zeta_2 - \zeta_1 )^2 + J_{ u } ( \zeta_1 ) 
		\leq J_{ u } ( \zeta_2 ) 
		\quad \text{for} ~ 0 < t < 1  ~ \text{and} ~ \zeta_1 < \zeta_2 .
	\end{align*}
	Optimizing in $ u $ and $ t $ yields the first claim by definition of $ J_{ r } $ via \eqref{E: definition of Jr}. Taking the limit $ r \rightarrow \infty $ and recalling \eqref{eqn:defn-j-alternative} we see that $ J $ has the same property. 
\end{proof}

\subsection{Proof of Theorem \ref{T: main deviations}} Finally, we combine Propositions \ref{P: large deviations}, \ref{P: lower deviations}, and \ref{P: upper deviations} to prove our main theorem and its corollary.  To this end, define the function $\check{\sigma} : [0,\infty) \to [\sigma_{*},\bar{\sigma}]$ by 
	\begin{align}\label{eqn:defn-sigma-check}
		\check{\sigma}(\kappa) = \bar{\sigma} - \inf \left\{ \zeta \geq 0 \, \mid \, J(\zeta) \geq \kappa \right\}.
	\end{align}
Since $J$ is strictly increasing in the set $ [ 0 , \bar\sigma - \sigma_{ * } ] $ and infinite in $(\bar{\sigma} - \sigma_{*},\infty)$, see Proposition \ref{P: large deviations}, $\check{\sigma}$ is well-defined, continuous, and nonincreasing.
        
        \begin{proof}[Proof of Theorem \ref{T: main deviations}]
Let $I \subseteq \mathbb{R}$ be an open interval. By stationarity, we can assume that $I$ is centered at the origin.  Notice also that by reflection symmetry the limits of $ \sigma_{\epsilon,\delta(\epsilon)}(1; I ) $ and $ \sigma_{\epsilon,\delta(\epsilon)} (-1; I ) $ agree. Thus, we only compute the former. 

Let us first handle the nondegenerate case $ \kappa \in ( 0, \infty ) $: Combining Proposition \ref{P: lower deviations} and \ref{P: upper deviations} from above, we learn that due to condition \eqref{E: log condition}, for any $ \zeta_{ - } < \zeta_{ + } $,
\begin{align*}
	J ( \zeta_{ - } ) < \kappa < J ( \zeta_{ + } ) 	
	\quad \text{implies} \quad
	\lim_{ \epsilon \downarrow 0 } \mathbb{P} \{  \bar{\sigma} - \zeta_{ + } \leq \sigma_{\epsilon,\delta(\epsilon)}(1; I ) \leq \bar{\sigma} - \zeta_{ - } \} 
	= 1.
\end{align*}
Since $J$ is strictly increasing (Proposition \ref{P: large deviations}), this readily implies that
\begin{align*}
	 \lim_{ \epsilon \downarrow 0 } \sigma_{\epsilon,\delta(\epsilon)}(1; I )
	 = \bar \sigma - \inf \left\{ \zeta \geq 0 \, \mid \, J(\zeta) \geq \kappa \right\}
	 \stackrel{ (\ref{eqn:defn-sigma-check}) }{ = }  \check\sigma ( \kappa )
	\quad \text{in probability}.
\end{align*}

The two degenerate cases $ \kappa \in \{ 0 , \infty \} $ follow by similar reasoning.  First, if $\kappa = 0$, then $J(\zeta) > J(0) = 0$ for any $\zeta > 0$, hence Proposition \ref{P: upper deviations} gives
	 
	\begin{align*}
		\lim_{ \epsilon \downarrow 0 } \mathbb{P} \{  \bar{\sigma} - \zeta \leq \sigma_{\epsilon,\delta(\epsilon)}(1; I ) \} = 1 \quad \text{for each} \, \, \zeta > 0.
	\end{align*}
At the same time, by \eqref{E: what we want to prove} (see \cite[Theorem 2]{part1}), the $\limsup$ of $\sigma_{\epsilon,\delta(\epsilon)}(1;I)$ is almost surely no larger than $\bar{\sigma}$, so $\sigma_{\epsilon,\delta(\epsilon)}$ converges in probability to $\bar{\sigma}$ in this case.
Second, if $\kappa = \infty$, then since $J(\zeta) < \infty$ for each $\zeta \in (0,\bar{\sigma} - \sigma_{*})$ by Proposition \ref{P: lower deviations}, 
	\begin{align*}
		\lim_{ \epsilon \downarrow 0 } \mathbb{P} \{  \sigma_{\epsilon,\delta(\epsilon)}(1; I ) \leq \sigma_{*} + \nu \} = 1 \quad \text{for each} \, \, \nu \in (0,\infty).
	\end{align*}
Since $\theta \geq \theta_{*}$ pointwise in $\mathbb{R}$, the formula \eqref{E: lower energy variational formula} implies that the $\liminf$ of $\sigma_{\epsilon,\delta(\epsilon)}(1;I)$ is almost surely larger than $\sigma_{*}$.  Therefore, $\sigma_{\epsilon,\delta(\epsilon)}(1;I) \to \sigma_{*}$ in probability in this case.  To see that these limits are consistent with the definition of $\check{\sigma}$, notice that since $J$ is nonnegative, we have $\check{\sigma}(0) = \bar{\sigma}$, and item (ii) of Proposition \ref{P: large deviations} implies that $\check{\sigma}(\infty) = \sigma_{*}$. Finally, since $J(0) = 0$ and $J$ is continuous at zero by item (i) of Proposition \ref{P: large deviations}, it follows from the definition of $\check{\sigma}$ that $\check{\sigma}(\kappa) < \bar{\sigma}$ for each $\kappa > 0$.
\end{proof}

Next, we show how the theorem determines the possible $\Gamma$-limits of $\mathscr{F}_{\epsilon,\delta}$, that is, we prove Corollaries \ref{C: gamma limits} and \ref{C: characterize gamma limits}.

Let us recall the notion of $\Gamma$-convergence in probability from \cite{part1}.  Suppose $I \subseteq \mathbb{R}$ is a bounded open interval.  In \cite[Appendix B]{part1}, we prove that there is a metric space $(\mathcal{M},d_{\Gamma})$ such that $\mathcal{M}$ is a set of l.s.c.\ functionals on $L^{1}(I)$;  convergence in $(\mathcal{M},d_{\Gamma})$ is equivalent to $\Gamma$-convergence; and $\mathbb{P} \{ \mathscr{F}_{\epsilon,\delta}(\cdot \, ; I ) \in \mathcal{M} \, \, \text{for each} \, \, \epsilon \leq 1 \, \, \text{and} \, \, \delta > 0 \} = 1$.  When we say that a sequence $\{ \mathscr{G}_{k} \}$ of random elements of $\mathcal{M}$ $\Gamma$-converges in probability to some limit $\mathscr{G}$, we mean that $d_{\Gamma}(\mathscr{G}_{k},\mathscr{G}) \to 0$ in probability as $k \to \infty$. 

\begin{proof}[Proof of Corollary \ref{C: gamma limits}] If $\epsilon^{-1} \delta(\epsilon) \to 0$ and $\epsilon^{-1} \delta(\epsilon) |\log \epsilon| \to \kappa$ as $\epsilon \to 0$, then Theorem \ref{T: main deviations} shows that, for each $e \in \{-1,1\}$ and any bounded open interval $I \subseteq \mathbb{R}$,  $\sigma_{\epsilon,\delta(\epsilon)}(e;I) \to \check{\sigma}(\kappa)$ in probability as $\epsilon \to 0$. Therefore, by \cite[Theorem 4]{part1}, for any such interval $I$, the functional $\mathscr{F}_{\epsilon,\delta(\epsilon)}(\cdot \, ;I)$ $\Gamma$-converges in probability to $\check{\mathscr{E}}_{\kappa}(\cdot \, ;I)$, that is, $d_{\Gamma} ( \mathscr{F}_{\epsilon,\delta(\epsilon)}(\cdot \,;I), \check{\mathscr{E}}_{\kappa}(\cdot \,;I)) \to 0$ in probability as $\epsilon \to 0$.\end{proof}

\begin{proof}[Proof of Corollary \ref{C: characterize gamma limits}] Let $\epsilon \mapsto \delta(\epsilon)$ be any scaling such that $\epsilon^{-1}\delta(\epsilon) \to 0$ as $\epsilon \to 0$.  Let $I \subseteq \mathbb{R}$ be a bounded open interval and $\mathscr{E}(\cdot \, ; I )$ be a l.s.c.\ functional on $L^{1}(I)$.

Suppose there is a sequence $\{\epsilon_{j}\}_{j \in \mathbb{N}}$ such that 
    \begin{align*}
    \epsilon_{j} \to 0 \quad \text{and} \quad d_{\Gamma}( \mathscr{F}_{\epsilon_{j},\delta(\epsilon_{j})}( \cdot \, ; I ) , \mathscr{E} ( \cdot \, ; I ) ) \to 0 \quad \text{in probability} \quad \text{as} \, \, j \to \infty.
    \end{align*}
    We show that $ \check{\sigma} ( \epsilon_{j}^{-1} \delta(\epsilon_{j}) | \log \epsilon_{j} | ) $ has a limit, by showing that any subsequence has a further subsequence that converges to a limit that is independent of the original subsequence. To this end, consider a subsequence $\{\epsilon_{j_{k}}\}_{k \in \mathbb{N}}$.  After passing to yet another subsequence, there is no loss of generality in assuming that there is a $\kappa \in [0,\infty]$ such that 
    \begin{align*}
        \lim_{ k \to \infty } \epsilon_{j_{k}}^{-1} \delta(\epsilon_{j_{k}}) | \log \epsilon_{j_{k}} | = \kappa . 
    \end{align*}
By the previous corollary, $d_{\Gamma}( \mathscr{F}_{\epsilon_{j_{k}},\delta(\epsilon_{j_{k}})} ( \cdot \, ; I ) , \check{\mathscr{E}}_{\kappa}(\cdot \, ; I ) ) \to 0$ in probability as $k \to \infty$.  This implies $\mathscr{E}( \cdot \, ; I ) = \check{\mathscr{E}}_{\kappa}(\cdot \, ; I )$.  Further, by continuity, $\check{\sigma}(\epsilon_{j_{k}}^{-1} \delta(\epsilon_{j_{k}}) | \log \epsilon_{j_{k}} |) \to \check{\sigma}(\kappa)$. Since $ \check{\sigma}(\kappa) | I | = \mathscr{E} ( \boldsymbol{1}_{ ( - \infty, x ) } - \boldsymbol{1}_{ ( x, \infty ) }; I ) $ for any $x \in I$, the limit only depends on $ \mathscr{E} $ and is thus independent of the choice of subsequences. Therefore, it holds for the original sequence, that is, $\check{\sigma}(\epsilon_{j}^{-1} \delta(\epsilon_{j}) | \log \epsilon_{j} | )$ converges to $\check{\sigma}(\kappa)$ as $j \to \infty$.

Conversely, suppose that $\{\epsilon_{j}\}$ is a sequence such that $\epsilon_{j} \to 0$ as $j \to \infty$ and  $\check{\sigma}(\epsilon_{j}^{-1} \delta(\epsilon_{j}) | \log \epsilon_{j} | ) \to \check{\sigma}(\kappa)$ for some $\kappa \in [0,\infty]$.  We claim that this implies that $d_{\Gamma}( \mathscr{F}_{\epsilon_{j},\delta(\epsilon_{j})}( \cdot \, ; I ) , \check{\mathscr{E}}_{\kappa}(\cdot \, ; I ) ) \to 0$ in probability as $j \to \infty$.  Indeed, let $\{ \epsilon_{j_{k}} \}$ be a subsequence such that $\epsilon^{-1}_{j_{k}} \delta(\epsilon_{j_{k}}) | \log \epsilon_{j_{k}} | \to \kappa_{0}$ as $k \to \infty$ for some $\kappa_{0} \in [0,\infty]$.  By assumption, this constrains $\kappa_{0}$ via the equation $\check{\sigma}(\kappa_{0}) = \check{\sigma}(\kappa)$.  As in the previous paragraph, this implies $\check{\mathscr{E}}_{\kappa_{0}}(\cdot \, ; I ) = \check{\mathscr{E}}_{\kappa}( \cdot \, ; I )$.  Therefore, by the previous corollary, $d_{\Gamma}( \mathscr{F}_{\epsilon_{j_{k}},\delta(\epsilon_{j_{k}})} ( \cdot \, ; I ) , \check{\mathscr{E}}_{\kappa}( \cdot \, ; I ) ) \to 0$ in probability as $k \to \infty$.  Since the subsequence was chosen arbitrarily, we conclude that the full sequence $\{\mathscr{F}_{\epsilon_{j},\delta(\epsilon_{j})}(\cdot \, ; I ) \}$ $\Gamma$-converges in probability to $\check{\mathscr{E}}_{\kappa}(\cdot \, ; I )$ as $j \to \infty$.    \end{proof}

\section{Generalities on Rare Events}

In this section, to prepare for the proofs of Theorems \ref{T: algebraic correlation counterexample} and \ref{T: almost periodic counterexample}, we establish an intermediate result.  Roughly speaking, it shows that if $\theta$ has the form $\theta(x) = \theta^{\text{stripe}}(x_{1}) \tilde{\theta}(x)$ for some one-dimensional profile $\theta^{\text{stripe}}$ and a nice-enough function $\tilde{\theta}$, then a sufficient condition for the limit of $\sigma_{\epsilon,\delta(\epsilon)} ( e_1 ; \cdot ) $ to lie below the homogenized value is for $\theta^{\text{stripe}}$ to attain its minimum value on arbitrarily long intervals. 

	\begin{theorem} \label{T: general rare events} Fix $\theta_{*}, \theta^{*} \in (0,\infty)$ with $\theta_{*} < \theta^{*}$.  Suppose there is a scaling $\epsilon \mapsto \delta(\epsilon)$ and a  deterministic function $\theta^{\text{stripe}} : \mathbb{R} \to [\theta_{*},\theta^{*}]$ such that $\epsilon^{-1} \delta(\epsilon) \to 0$ as $\epsilon \to 0$ and the following property holds: For any $M \geq 1$ and any $\varrho' > 0$, for all sufficiently small $\epsilon > 0$, there is a $T_{\epsilon} \in \mathbb{R}$ such that 
		\begin{align} \label{E: stripe condition}
			\theta^{\text{stripe}}( \delta(\epsilon)^{-1} s ) = \theta_{*} \quad \text{for each} \, \, s \in [ T_{\epsilon} - M \epsilon, T_{\epsilon} + M \epsilon ] \subseteq [-\varrho', \varrho'].
		\end{align}
	Assume that $\tilde{\theta} : \mathbb{R}^{d} \to [0,\infty)$ satisfies either
		\begin{itemize}
			\item[(i)] $\tilde{\theta}$ is the random checkerboard of the form
				\begin{align*}
					\tilde{\theta}(x) = \sum_{ z \in \mathbb{Z}^{d} } \tilde{\Theta}_{z} \mathbf{1}_{z + (-1/2,1/2)^{d} }(x)
				\end{align*}
			for some i.i.d.\ positive, bounded random variables $\{ \tilde{\Theta}_{z} \}_{z \in \mathbb{Z}^{d}}$ with $\mathbb{E} [ \tilde{\Theta}_{z} ] = 1$, or 
			\item[(ii)] $\tilde{\theta}$ is a deterministic $\mathbb{Z}^{d}$-periodic function with $\int_{\mathbb{T}^{d}} \tilde{\theta} = 1$.
		\end{itemize}
	Then for any $\varrho > 0$,
		\begin{align*}
			\limsup_{ \epsilon \downarrow 0 } \sigma_{ \epsilon, \delta(\epsilon) } ( e_{1} ; [-\varrho,\varrho]^{d} ) \leq \sigma_{W} \sqrt{ \theta_{*} } (2\varrho)^{d-1},
		\end{align*}
	where the $\limsup$ is interpreted in terms of convergence in probability in case (i).
	\end{theorem}

	\begin{proof} There are three steps of the proof.  We first argue that the limit of $\sigma_{\epsilon,\delta}$ can be estimated using a certain natural class of boundary condition-unconstrained competitors.  In the second step, we show that the energy of the competitors converges to the desired upper bound $\sigma_{W} \sqrt{ \theta_{*} } ( 2 \varrho )^{d-1}$ conditional on some averaging results involving $\tilde{\theta}$.  In the final step, we prove these averaging results.

    \textit{Step 1: Reformulating the Limit.} Given a unit vector $n \in S^{d-1}$,  define $\bar{u}_{n} : \mathbb{R}^{d} \to \{-1,1\}$ by setting $\bar{u}_{n}(x) = 1$ if $x \cdot n \geq 0$ and $\bar{u}_{n}(x) = -1$ if $x \cdot n < 0$.  Fix a function $q_{*} : \mathbb{R} \to [-1,1]$ such that $q_{*}(\pm \infty) = \pm 1$ and
		\begin{align*}
			&\int_{-\infty}^{\infty} \left( \frac{1}{2} | q_{*}' |^{2} + \theta_{*} W ( q_{*} ) \right) \, dy \\
			&\quad= \min \left\{ \int_{-\infty}^{\infty} \left( \frac{1}{2} | v' |^{2} + \theta_{*} W ( v ) \right) \, dy \, \mid \, v : \mathbb{R} \to [-1,1], \, \, v(\pm \infty) = \pm1 \right\} = \sigma_{W} \sqrt{ \theta_{*} }.
		\end{align*}
    To obtain the desired bound on the $\limsup$ of $\sigma_{\epsilon,\delta(\epsilon)}(e_{1};[-\varrho,\varrho]^{d})$, it suffices to prove that there is a deterministic function $M \mapsto \eta(M)$ such that $\eta(M) \to 0$ as $M \to \infty$ and, for each $M \geq 1$, 
        \begin{align} \label{E: limsup bound generalities}
            \limsup_{ \epsilon \downarrow 0 } \sigma_{ \epsilon, \delta(\epsilon) } ( e_{1} ; [-\varrho,\varrho]^{d} ) \leq (2 \varrho)^{d-1} \int_{-M}^{M} \left( \frac{1}{2} |q_{*}'(t)|^{2} + \theta_{*} W(q_{*}(t)) \right) \, dt + \eta(M),
        \end{align}
    or, more precisely, in case (i), for each $M \geq 1$ and $\nu > \eta(M)$,        
        \begin{align*}
            \lim_{ \epsilon \downarrow 0 } \mathbb{P} \left\{ \sigma_{ \epsilon, \delta(\epsilon) } ( e_{1} ; [-\varrho,\varrho]^{d} ) - (2 \varrho)^{d-1} \int_{-M}^{M} \left( \frac{1}{2} |q_{*}'(t)|^{2} + \theta_{*} W(q_{*}(t)) \right) \, dt \geq \nu \right\} = 0.
        \end{align*}
        
    Let $M \geq 1$ and $\varrho' \in (0,\varrho)$.  By assumption, for sufficiently small $\epsilon$, there is a $T_{\epsilon} \in \mathbb{R}$ such that \eqref{E: stripe condition} holds.  Define $u_{\epsilon} : \mathbb{R}^{d} \to [-1,1]$ by
		\begin{align*}
			u_{\epsilon}(x) = q_{*} ( \epsilon^{-1} ( x \cdot e_{1} - T_{\epsilon} ) ).
		\end{align*}
	Since $T_{\epsilon} \in [-\varrho',\varrho']$ by hypothesis,
		\begin{align} \label{E: limiting useful estimate}
			\limsup_{ \epsilon \downarrow 0 } \| u_{\epsilon} - \bar{u}_{e_{1}} \|_{L^{1}([-\varrho,\varrho]^{d})} \leq (2 \varrho') (2\varrho)^{d-1} .
		\end{align}
	We will show below that there is a deterministic function $M \mapsto \eta(M)$ such that $\eta(M) \to 0$ as $M \to \infty$ and
		\begin{align} \label{E: energy striped construction}
			\limsup_{ \epsilon \downarrow 0 } \mathscr{F}_{\epsilon,\delta(\epsilon)} ( u_{\epsilon} ; [-\varrho,\varrho]^{d} ) \leq ( 2 \varrho )^{d-1} \int_{-M}^{M} \left( \frac{1}{2} | q_{*}'(t) |^{2} + W(q_{*}(t)) \right) \, dt + \eta(M),
		\end{align}
	where the limit above holds in the sense of convergence in probability in case (i).  This only becomes useful, however, once we modify $u_{\epsilon}$ so that it satisfies the needed boundary conditions in the definition of $\sigma_{\epsilon,\delta(\epsilon)}$.

    This can be done via an application of the fundamental estimate of $\Gamma$-convergence, specifically \cite[Theorem 6]{morfe}.  Applying that result in conjunction with \eqref{E: limiting useful estimate}, we deduce that  there is a deterministic constant $C_{0} > 0$, which depends on $\theta_{*}$, $\theta^{*}$, and $W$ but not $\rho$ or $\rho'$, such that, for any $\epsilon > 0$ small enough, there is a function $\tilde{u}_{\epsilon} \in H^{1}([-\varrho,\varrho]^{d};[-1,1])$ satisfying $\tilde{u}_{\epsilon}(x) = q(\epsilon^{-1} x \cdot e_{1} )$ for $x \in \partial [-\varrho,\varrho]^{d}$ and such that
        \begin{align*}
            \limsup_{ \epsilon \downarrow 0 } \mathscr{F}_{\epsilon,\delta(\epsilon)} (\tilde{u}_{\epsilon}; [-\varrho,\varrho]^{d}) \leq \limsup_{ \epsilon \downarrow 0 } \mathscr{F}_{\epsilon,\delta(\epsilon)}(u_{\epsilon}; [-\varrho,\varrho]^{d}) + C_{0} ( 2 \rho' ) ( 2 \rho )^{d-1} .  
        \end{align*}
    Due to the choice of boundary conditions, this implies
        \begin{align*}
            \limsup_{ \epsilon \downarrow 0 } \sigma_{\epsilon,\delta(\epsilon)}(e_{1};[-\varrho,\varrho]^{d}) \leq \limsup_{ \epsilon \downarrow 0 } \mathscr{F}_{\epsilon,\delta(\epsilon)}(u_{\epsilon}; [-\varrho,\varrho]^{d}) + C_{0} ( 2 \rho' ) ( 2 \rho )^{d-1},
        \end{align*}
     where the inequality holds with probability one in case (i).  Combined with \eqref{E: energy striped construction}, since $\rho'$ can be taken arbitrarily small, this implies \eqref{E: limsup bound generalities}.
	
	\textit{Step 2: Main Argument.} We now determine the limit of $\mathscr{F}_{\epsilon,\delta(\epsilon)}(u_{\epsilon} ; [-\varrho,\varrho]^{d})$.  To that end, let us decompose the energy as follows:
		\begin{align*}
			\mathscr{F}_{\epsilon,\delta(\epsilon)} &(u_{\epsilon}; [-\varrho,\varrho]^{d}) = I_{\epsilon} + II_{\epsilon}, 
		\end{align*}
	where
		\begin{align*}
			I_{\epsilon} &= \epsilon^{-1} \int_{U_{\epsilon}} \left( \frac{1}{2} q_* '(\epsilon^{-1}(x_{1} - T_{\epsilon}))^{2} + \theta_{*} \tilde{\theta}(\delta(\epsilon)^{-1} x) W(q_* (\epsilon^{-1}(x_{1} - T_{\epsilon}))) \right) \, dx, \\
			U_{\epsilon} &= [T_{\epsilon} - M \epsilon, T_{\epsilon} + M \epsilon] \times [- \varrho, \varrho]^{d - 1}, \\
			0 \leq II_{\epsilon} &\leq \left( 1 + \Lambda \|\tilde{\theta}\|_{L^{\infty}(\mathbb{R}^{d})} \right) \frac{ (2\varrho)^{d-1} }{ \epsilon } \int_{\mathbb{R} \setminus [- M \epsilon, M \epsilon]} \left( \frac{1}{2} q_{*} '( \epsilon^{-1} s )^{2} +W(q_{*} ( \epsilon^{-1} s )) \right) \,ds.
		\end{align*}
	If we define $\eta(M)$ by 
		\begin{equation*}
			\eta(M)  = \left( 1 + \Lambda \|\tilde{\theta}\|_{L^{\infty}(\mathbb{R}^{d})} \right) (2 \varrho)^{d-1} \int_{\mathbb{R} \setminus [-M,M]} \left(\frac{1}{2} q_{*} '(t)^{2} + W(q_{*} (t)) \right) \, dt,
		\end{equation*}
	then, upon changing variables according to $t = \epsilon^{-1} s$, we observe that 
		\begin{align*}
			II_{\epsilon} \leq \eta(M),
		\end{align*}
	where the inequality holds with probability one in case (i), and $\eta(M) \to 0$ as $M \to \infty$.  Thus, it only remains to study the asymptotics of $I_{\epsilon}$.\
	
	\textit{Step 3: Asymptotics of $I_{\epsilon}$.} Concerning $I_{\epsilon}$, after changing variables in the potential term and recalling the notation $e_{1} = (1,0,0,\dots,0)$, we observe that 
		\begin{align*}
			 &\epsilon^{-1} \int_{U_{\epsilon}} \tilde{\theta}(\delta(\epsilon)^{-1} x) W(q_* (\epsilon^{-1}(x_{1} - T_{\epsilon}))) \, dx \\
			 &\quad = \epsilon^{d-1} \int_{\epsilon^{-1} (U_{\epsilon} - T_{\epsilon} e_{1})} \tilde{\theta}(\delta(\epsilon)^{-1} (\epsilon y + T_{\epsilon} e_{1})) W(q_*(y_{1})) \, dy.
		\end{align*}
	The rest of the analysis involves using the assumptions on $\tilde{\theta}$ to prove that
		\begin{align} \label{E: stripe condition need to prove}
			\lim_{ \epsilon \downarrow 0 } \epsilon^{d-1} \int_{\epsilon^{-1} (U_{\epsilon} - T_{\epsilon} e_{1})} \tilde{\theta}(\delta(\epsilon)^{-1} (\epsilon y + T_{\epsilon} e_{1})) W(q_*(y_{1})) \, dy = (2\varrho)^{d-1} \int_{-M}^{M} W(q_{*}(t)) \, dt,
		\end{align}
	where the limit is interpreted in terms of convergence in probability in case (i).  Note that the gradient term similarly converges
		\begin{align*}
			\lim_{ \epsilon \downarrow 0 }\epsilon^{-1} \int_{U_{\epsilon}} \frac{1}{2} q_{*} '(\epsilon^{-1}(x_{1} - T_{\epsilon}))^{2} \, dx = (2\varrho)^{d-1} \int_{-M}^{M} \frac{1}{2} q_{*}'(t)^{2} \, dt.
		\end{align*}
	Therefore, once \eqref{E: stripe condition need to prove} is established, we will have proved \eqref{E: energy striped construction} as claimed.

	\textit{Step 3(a): Random Checkerboard.} In this case, since $\tilde{\theta}$ is stationary and $\mathbb{E}[\tilde{\theta}(0)] = 1$, 
		\begin{align*}
			&\mathbb{P} \left\{ \left| \epsilon^{d-1} \int_{\epsilon^{-1} (U_{\epsilon} - T_{\epsilon}e_{1})} \{\tilde{\theta}(\delta(\epsilon)^{-1} (\epsilon y + T_{\epsilon}e_{1})) - 1 \} W(q_*(y_{1})) \, dy \right| > \mu \right\} \\
			&\quad = \mathbb{P} \left\{ \left| \epsilon^{d-1} \int_{ [-M,M] \times [-\epsilon^{-1} \varrho, \epsilon^{-1} \varrho]^{d-1} } \{ \tilde{\theta}(\delta(\epsilon)^{-1}\epsilon y) - \mathbb{E} [ \tilde{\theta}(0) ] \} W(q_*(y_{1})) \, dy \right| > \mu \right\}.
		\end{align*}
	Thus, a large deviations argument analogous to the one in the proof of Lemma \ref{L: basic LDP} (the difference being that --- in the language of that proof --- we see on the order of $ R^{d-1} / \gamma^d  = \epsilon / \delta^d $ instead of $ 1 / \gamma = \epsilon / \delta $ random variables) implies that  there is a constant $C(\mu,M,\varrho) > 0$ such that, for sufficiently small $\epsilon$,
		\begin{align*}
			&\mathbb{P} \left\{ \left| \epsilon^{d-1} \int_{ [-M,M] \times [-\epsilon^{-1} \varrho, \epsilon^{-1} \varrho]^{d-1} } \{ \tilde{\theta}(\delta(\epsilon)^{-1}\epsilon y) - \mathbb{E}[\tilde{\theta}(0)] \} W(q_*(y_{1})) \, dy \right| > \mu \right\} \\
			&\qquad \qquad \leq \exp \left( - C(\mu,M,\varrho) \frac{ \epsilon }{ \delta(\epsilon)^d } \right) .
		\end{align*}
		Notice that the right-hand side vanishes as $ \epsilon \downarrow 0 $ since $ \epsilon \delta(\epsilon)^{-d} = \epsilon^{ 1 - d } ( \epsilon^{-1} \delta(\epsilon) )^{-d} \rightarrow \infty $.  In particular, for any $\mu > 0$, the probability
		\begin{align*}
			\mathbb{P} \left\{ \left| \epsilon^{d-1} \int_{\epsilon^{-1}(U_{\epsilon} - T_{\epsilon} e_1 )} \tilde{\theta}(\delta(\epsilon)^{-1} ( \epsilon y + T_{\epsilon} e_{1} ) )  W(q_*(y_{1})) \, dy - (2\varrho)^{d-1} \int_{-M}^{M} W(q_{*}) \, dt \right| > \mu \right\}
		\end{align*}
	converges to zero as $\epsilon \downarrow 0$.
	
	\textit{Step 3(b): Periodic Medium.} Since in this case $\tilde{\theta}$ is periodic and $\int_{\mathbb{T}^{d}} \tilde{\theta}(y) \, dy = 1$, one readily deduces by trigonometric approximation that
	\begin{align*}
		&\lim_{\epsilon \downarrow 0} \epsilon^{d-1} \int_{[-M,M] \times [-\epsilon^{-1}\varrho,\epsilon^{-1}\varrho]^{d-1}} \left[ \tilde{\theta}(\delta(\epsilon)^{-1}(\epsilon y + T_{\epsilon} e_{1} )) - 1 \right] W(q_{*}(y_{1})) \, dy = 0.
	\end{align*}
Indeed, if we replace $\tilde{\theta} - 1$ by the plane wave $\exp ( \iota 2 \pi k \cdot y )$ for some nonzero $k = (k_{1},k_{2},\dots,k_{d}) \in \mathbb{Z}^{d} \setminus \{0\}$, then we can invoke the Riemann-Lebesgue Lemma to find
	\begin{align*}
		&\lim_{\epsilon \downarrow 0} \epsilon^{d-1} \int_{[-\epsilon^{-1}\varrho,\epsilon^{-1}\varrho]^{d-1}} e^{\iota 2 \pi \delta(\epsilon)^{-1} \epsilon (k_{2},\dots,k_{d}) \cdot y'} \, dy'  \int_{-M}^{M} e^{ \iota 2 \pi k_{1} \delta(\epsilon)^{-1} ( \epsilon y_{1} + T_{\epsilon} ) } W(q_{*}(y_{1})) \, dy \\
		&\quad = \mathbf{1}_{\{0\}} (k_{2},\dots,k_{d})  \lim_{ \epsilon \downarrow 0 } e^{ \iota 2 \pi k_{1} \delta(\epsilon)^{-1} \epsilon T_{\epsilon} } \int_{-M}^{M} e^{ \iota 2 \pi k_{1} \delta(\epsilon)^{-1} \epsilon y_{1}  } W(q_{*}(y_{1})) \, dy_{1} = 0.
	\end{align*}
Due to the normalizing factor $\epsilon^{d-1}$ appearing in front of the integral, this remains true for any mean-zero $\mathbb{Z}^{d}$-periodic function, hence, in particular, to $\tilde{\theta} - 1$.\end{proof}
	
\section{Random Media with Long-Range Correlations} \label{S: long range correlations}

This section describes examples of random media in which the rare events regime occurs above an algebraic scale $\delta \approx \epsilon^{\gamma}$, that is, we prove Theorem \ref{T: algebraic correlation counterexample}.  For simplicity, we take $a \equiv \text{Id}$ and focus on $\theta$.  As in the introduction, $\theta$ will have the form
	\begin{equation*}
		\theta(y) = \theta^{\text{stripe}}(y_{1}) \tilde{\theta}(y),
	\end{equation*}
where $\tilde{\theta}$ is a $d$-dimensional random checkerboard with mean $\mathbb{E}[\tilde{\theta}] = 1$ and $\theta^{\text{stripe}}$, which is taken to be statistically independent of $\tilde{\theta}$, has the properties stated next:

	\begin{prop} \label{P: random stripes} For any $\gamma > 1$, there is a probability space $(\Omega,\mathcal{F},\mathbb{P})$ supporting a nonconstant stationary random field $\theta^{\text{stripe}} : \mathbb{R} \to \{1,2\}$ such that if $\epsilon \mapsto \delta(\epsilon)$ is any deterministic scaling such that
		\begin{equation} \label{E: liminf condition}
			\liminf_{\epsilon \downarrow 0} \delta(\epsilon)\epsilon^{-\gamma} > 0,
		\end{equation}
	then, for any (deterministic) $M \geq 1$ and $\varrho > 0$, with probability one, for all sufficiently small $\epsilon > 0$, there is a $T_{\epsilon} \in [- \varrho, \varrho]$ such that 
		\begin{equation} \label{E: long correlations need this}
			\theta^{\text{stripe}}(\delta(\epsilon)^{-1}s) = 1 \quad \text{for each} \quad s \in [T_{\epsilon} - M \epsilon, T_{\epsilon} + M\epsilon] \subseteq [- \varrho,\varrho].
		\end{equation}
	\end{prop}

Next, we use the previous proposition together with the main result of the previous section (Theorem \ref{T: general rare events}) to prove Theorem \ref{T: algebraic correlation counterexample}.

\begin{proof}[Proof of Theorem \ref{T: algebraic correlation counterexample}] Let $\gamma > 1$ and choose a scaling $\epsilon \mapsto \delta(\epsilon)$ such that $\epsilon^{-1} \delta(\epsilon) \to 0$ as $\epsilon \downarrow 0$ and \eqref{E: liminf condition} holds.  Let $a \equiv \text{Id}$ and let $\theta(y) = \theta^{\text{stripe}}(y_{1}) \tilde{\theta}(y)$ for $y = (y_{1},y_{2},\dots,y_{d}) \in \mathbb{R}^{d}$, where $\theta^{\text{stripe}}$ and $\tilde{\theta}$ are independent random fields; $\theta^{\text{stripe}} : \mathbb{R} \to \{1,2\}$ is the stationary random field described in Proposition \ref{P: random stripes}; and, as in case (i) of  Theorem \ref{T: general rare events}, $\tilde{\theta}$ is a positive random checkerboard with mean $\mathbb{E} [ \tilde{\theta}(x) ] = 1$.

In what follows, denote by $\mathbb{P}(A \, \mid \, \theta^{\text{stripe}} )$ the conditional probability of an event $A$ given $\theta^{\text{stripe}}$.

Fix $\nu > 0$.  Let $\Omega_{0}$ denote the event that, for any $M \geq 1$ and any $\rho > 0$, there is an $\epsilon_{*}(M,\rho) > 0$ such that, for each $\epsilon < \epsilon_{*}(M,\rho)$, there is a $T_{\epsilon}$ such that \eqref{E: long correlations need this} holds.
Since conditioning on $\theta^{\text{stripe}}$ does not change the law of $\tilde{\theta}$ (due to independence), we can invoke Theorem \ref{T: general rare events} to find
\begin{align*}
\lim_{ \epsilon \downarrow 0 } \mathbb{P} \{ \sigma_{\epsilon,\delta(\epsilon)}(e_{1}; [-\varrho,\varrho]^{d} )  \leq \sigma_{W} ( 2 \rho )^{d-1} + \nu \, \mid \, \theta^{\text{stripe}} \} = 1 \quad \text{almost surely on} \, \, \Omega_{0}.
\end{align*}
Since $\mathbb{P}(\Omega_{0}) = 1$ and $\mathbb{E}[ \mathbb{P}(A \, \mid \, \theta^{\text{stripe}} ) ] = \mathbb{P}(A)$ for any event $A$, we can invoke Lebesgue's dominated convergence theorem to conclude that
\begin{align*}
\lim_{ \epsilon \downarrow 0 } \mathbb{P} \{ \sigma_{\epsilon,\delta(\epsilon)}(e_{1}; [-\varrho,\varrho]^{d} ) \leq \sigma_{W} ( 2 \rho )^{d-1} + \nu \} = 1 .
\end{align*}
This completes the proof of \eqref{E: algebraic correlations lim inf} since, by construction, $\theta^{\text{stripe}}_{*} = 1$ and 
	\begin{align*}
		\theta^{\text{stripe}}_{*} < \mathbb{E} [ \theta^{\text{stripe}}(0) ] = \mathbb{E} [ \theta(0) ] = \bar{\theta}
	\end{align*}
by independence, $\mathbb{E}[\tilde{\theta}(0)] = 1$, and the fact that $\theta^{\text{stripe}}$ is nonconstant.
\end{proof}

The remainder of this section is devoted to the proof of Proposition \ref{P: random stripes}.

\subsection{Idea of the Construction} The construction of the field $\theta^{\text{stripe}}$ of Proposition \ref{P: random stripes} is based on what amounts to a long-range site percolation model.  Suppose a lamp sits at each point of $\mathbb{Z}$.  The brightness of the lamp is determined by an integer-valued random variable $X_{m}$.  For each $k \in \mathbb{Z}$, if $|m - k| \leq X_{m}$, then the lamp at $m$ illuminates $k$.  (If $X_{m} < 0$, then the lamp at $m$ is off.)  In this way, the set of all illuminated sites is a random subset of $\mathbb{Z}$ determined by $\{X_{m}\}_{m \in \mathbb{Z}}$.  

We define $\theta^{\text{stripe}}$ as a random step function
	\begin{equation} \label{E: random step correlations}
		\theta^{\text{stripe}}(y) = \sum_{k \in \mathbb{Z}} \Theta^{\text{stripe}}_{k} \mathbf{1}_{[k,k+1)}(y),
	\end{equation}
where $\Theta^{\text{stripe}}_{k} = 1$ if $k$ is illuminated and $\Theta^{\text{stripe}}_{k} = 2$, otherwise.  More precisely,
	\begin{equation} \label{E: markers}
		\Theta^{\text{stripe}}_{k} = 1 \quad \text{if} \quad \sum_{m \in \mathbb{Z}} \mathbf{1}_{\{|m - k| \leq X_{m}\}} \geq 1 , \quad \Theta^{\text{stripe}}_{k} = 2 \quad \text{otherwise.}	
	\end{equation}

In the remainder of the construction, the law of $\{X_{m}\}_{m \in \mathbb{Z}}$ is chosen so that $\theta^{\text{stripe}}$ has the desired properties laid out in Proposition \ref{P: random stripes}.
	
\subsection{Construction of $\theta^{\text{stripe}}$} \label{S: algebraic correlation construction} To specify the law of the random variables $\{X_{m}\}_{m \in \mathbb{Z}}$, fix an $\alpha > 0$ and define a probability measure $\mu_{\alpha}$ on $\mathbb{Z}$ by setting
	\begin{equation*}
		\mu_{\alpha}(k) = Z_{\alpha}^{-1} |k|^{-(2 + \alpha)} \quad \text{for each} \quad k \in \mathbb{Z} \setminus \{0\}, \quad \mu_{\alpha}(0) = 0.
	\end{equation*}
The constant $Z_{\alpha} > 0$ is fixed so that $\sum_{k \in \mathbb{Z}} \mu_{\alpha}(k) = 1$.

Let $(\Omega,\mathcal{F},\mathbb{P}_{\alpha})$ be a probability space supporting i.i.d.\ random variables $\{X_{m}\}_{m \in \mathbb{Z}}$, each with law $\mu_{\alpha}$.  Define $\{\Theta^{\text{stripe}}_{k}\}_{k \in \mathbb{Z}}$ by \eqref{E: markers} as before.  

	The choice of the decay $\mu_{ \alpha } (k) \sim k^{-(2 + \alpha)}$ with exponent $2 + \alpha > 2$ is explained by the next result.

	\begin{prop} $\{\Theta^{\text{stripe}}_{k}\}_{k \in \mathbb{Z}}$ is a stationary sequence with a nontrivial law.  In particular, for any choice of $\alpha > 0$, we have $0 < \mathbb{P}_{\alpha} \{\Theta^{\text{stripe}}_{0} = 2\} < 1$. \end{prop}
	
	The proof below also shows that if the exponent $2 + \alpha$ is replaced by $1 + \alpha$ with $\alpha \leq 1$, then $\mathbb{P}_{\alpha}\{\Theta^{\text{stripe}}_{0} = 1\} = 1$ and then $\Theta^{\text{stripe}}$ is simply a constant sequence.  Thus, we restrict to exponents above two.
	
		\begin{proof} It is straightforward to check that $\mathbb{P}_{\alpha}\{\Theta_{0}^{\text{stripe}} = 2\} \leq \mathbb{P}_{\alpha}\{X_{0} < 0\} = \frac{1}{2}$.  Thus, it only remains to check that the probability is positive. 
		
		Toward that end, observe that 
			\begin{align*}
				\{\Theta^{\text{stripe}}_{0} = 2\} = \bigcap_{m \in \mathbb{Z}} \{X_{m} < |m|\}.
			\end{align*}
		Thus, since $\{X_{m}\}_{m \in \mathbb{Z}}$ is i.i.d.,
			\begin{align*}
				\mathbb{P}_{\alpha}\{\Theta^{\text{stripe}}_{0} = 2\}^{\frac{1}{2}}
				&= \mathbb{P}_{\alpha}\{X_{0} < 0\}^{\frac{1}{2}} \prod_{m = 1}^{\infty} \mathbb{P}_{\alpha}\{X_{m} < m\} \\
				&= \frac{1}{\sqrt{2}} \prod_{m = 1}^{\infty} \left(1 - Z_{\alpha}^{-1} \sum_{j = m}^{\infty} j^{-(2 + \alpha)} \right)
			\end{align*}
		It only remains to prove that 
			\begin{align*}
				\prod_{m = 1}^{\infty} \left(1 - Z_{\alpha}^{-1} \sum_{j = m}^{\infty} j^{-(2 + \alpha)} \right) > 0.
			\end{align*}
		To this end, recall that there is a constant $c_{\alpha} > 1$ such that
			\begin{align} \label{E: useful integral comparison algebraic}
				c_{\alpha}^{-1} m^{-(1 + \alpha)} \leq \sum_{j = m}^{\infty} j^{-(2 + \alpha)} \leq c_{\alpha} m^{-(1 + \alpha)} \quad \text{for each} \quad m \geq 1.
			\end{align}
		Therefore, upon fixing $m_{*} \in \mathbb{N}$ with $Z_{\alpha}^{-1} c_{\alpha} m_{*}^{-(1 + \alpha)} < 1$, we find
			\begin{align*}
			\prod_{m = m_{*}}^{\infty} \left(1 - Z_{\alpha}^{-1} \sum_{j = m}^{\infty} j^{-(2 + \alpha)} \right) \geq \prod_{m = m_{*}}^{\infty} \left( 1 - Z_{\alpha}^{-1} c_{\alpha} m^{-(1 + \alpha)} \right).
			\end{align*}
		In view of the fact that $\sum_{m = 1}^{\infty} m^{-(1 + \alpha)} < \infty$, we know that 
			\begin{align*}
				\sum_{m = m_{*}}^{\infty} \log \left( 1 - Z_{\alpha}^{-1} c_{\alpha} m^{-(1 + \alpha)} \right) > - \infty ,
			\end{align*}
		and, thus, $\prod_{m = m_{*}}^{\infty} \left( 1 - Z_{\alpha}^{-1} c_{\alpha} m^{-(1 + \alpha)} \right) > 0$.
		\end{proof}

Because of the long tails of the random variables $\{X_{m}\}_{m \in \mathbb{Z}}$, there are long, random intervals on which the field $\theta^{\text{stripe}}$ is identically equal to one.  The next proposition quantifies this.

	\begin{prop} \label{P: algebraic scale selection} Fix $0 < \alpha < \beta$, $c > 0$, and $\ell \in \mathbb{N}$. For any $N \in \mathbb{N}$, let $\tilde{E}_{N}$ be the event
		\begin{equation*}
			\tilde{E}_{N} = \bigcup_{j = \lfloor -cN^{1 + \beta} \rfloor}^{\lceil cN^{1 + \beta} \rceil} \{\theta^{\text{stripe}} \equiv 1 \, \, \text{in} \, \, [j - \ell N, j + \ell N]\}
		\end{equation*}
	and let $\tilde{E} = \liminf_{N \to \infty} \tilde{E}_{N}$.  There is a constant $C = C(\alpha,\beta,c) \geq 1$ such that
		\begin{equation} \label{E: what we want long range}
			\mathbb{P}_{\alpha}(\tilde{E}_{N}) \geq 1 - C \exp \left( - C^{-1} N^{\beta - \alpha} \right).
		\end{equation}
	In particular, $\mathbb{P}_{\alpha}(\tilde{E}) = 1$. \end{prop}
	
		\begin{proof} By independence and the same reasoning that led to \eqref{E: useful integral comparison algebraic}, there is a constant $C_{\alpha} > 0$ such that 
			\begin{align*}
				\mathbb{P}_{\alpha}(E_{N}) &\geq \mathbb{P}_{\alpha} \left( \bigcup_{ j = \lfloor - c N^{ 1 + \beta } \rfloor }^{ \lceil c N^{ 1 + \beta } \rceil } \{ X_{j} \geq \ell N \} \right) \\
					&= 1 - \prod_{ j = \lfloor - c N^{ 1 + \beta } \rfloor }^{ \lceil c N^{ 1 + \beta } \rceil } \mathbb{P}_{\alpha} \{ X_{j} < \ell N \} \geq 1 - \left( 1 -  C_{\alpha} \ell^{- ( 1 + \alpha ) } N^{- ( 1 + \alpha ) } \right)^{1 + 2 \lceil c N^{ 1 + \beta } \rceil }.
			\end{align*}
		Upon rearranging, this becomes
			\begin{align*}
				\mathbb{P}_{\alpha}( \Omega_{1} \setminus \tilde{E}_{N}) \leq \exp \left( (2 \lceil c N^{1 + \beta} \rceil + 1) \log \left(1 - C_{\alpha} \ell^{- ( 1 + \alpha ) }  N^{- (1 + \alpha) } \right) \right),
			\end{align*}
		which implies \eqref{E: what we want long range}.  Further, by the Borel-Cantelli Lemma, 
			\begin{align*}
		1 - \mathbb{P}_{\alpha} ( \tilde{E} ) = \mathbb{P}_{\alpha} \left(\limsup_{N \to \infty} \Omega_{1} \setminus \tilde{E}_{N} \right) = 0.
			\end{align*}
\end{proof}
	
\subsection{Proof of Proposition \ref{P: random stripes}} It only remains to prove that $\theta^{\text{stripe}}$ has the desired properties.  This follows more-or-less directly from Proposition \ref{P: algebraic scale selection} after rescaling.

To see this, not unlike the discussion in Section \ref{S: overview large deviations}, we rescale the problem, but this time in such a way that the microscale has units of order one.  We then find that the macroscale has units of order $T = \delta^{-1}$ and the mesoscale has units of order $\mu = \epsilon \delta^{-1}$.  In order for the energy functional to see the long intervals on which $\theta^{\text{stripe}}$ is constant, these intervals should have length of order $\mu$.  At the same time, to take full advantage of Proposition \ref{P: algebraic scale selection}, the corresponding macroscale should be at least of order $\mu^{1 + \beta}$.  In terms of $\epsilon$ and $\delta$, the requirement $T \gtrsim \mu^{1 + \beta}$ corresponds to $\delta \gtrsim \epsilon^{\frac{1 + \beta}{\beta}}$.

The function $\beta \mapsto \frac{1 + \beta}{\beta}$ maps the interval $(0,\infty)$ onto $(1,\infty)$.  Therefore, these examples cover any rate of algebraic decay faster than $\epsilon$.

	\begin{proof}[Proof of Proposition \ref{P: random stripes}] Fix $\gamma > 1$, $\varrho > 0$, and $M \geq 1$, and fix a scaling $\epsilon \mapsto \delta(\epsilon)$ such that 
		\begin{equation} \label{E: algebraic scaling}
			K\coloneqq \liminf_{\epsilon \downarrow 0} \epsilon^{-\gamma} \delta(\epsilon) > 0.
		\end{equation}
	Fix an arbitrary $0 < \alpha < (\gamma - 1)^{-1}$ and let $(\Omega,\mathcal{F},\mathbb{P}) \coloneqq (\Omega,\mathcal{F},\mathbb{P}_{\alpha})$ be the probability space built in Section \ref{S: algebraic correlation construction} with this choice of $\alpha$.  Define $\beta = (\gamma - 1)^{-1}$.
	
	For any scale $N$, define the random point $J_{N}$ by
		\begin{align*}
			J_{N} &= \inf \left\{ j \in \mathbb{Z} \, \mid \, - \varrho K^{\frac{\beta}{1 + \beta}} N^{1 + \beta}/4 \leq j \leq \varrho K^{\frac{\beta}{1 + \beta}} N^{1 + \beta}/4 \, \, \text{and} \, \,  \right. \\
				&\qquad \qquad \qquad \qquad \left. \theta^{\text{stripe}}(t) = 1 \, \, \text{for each} \, \, t \in [j - \lceil M \rceil N, j + \lceil M \rceil N] \right\},
		\end{align*}
	where we set $J_{N} = -\infty$ if the set inside the infimum is empty.  Note that the random variables $(J_{N})_{N \in \mathbb{N}}$ are $\mathcal{F}_{1}$-measurable.  By Proposition \ref{P: algebraic scale selection}, with probability one, there is a random $N_{*} < \infty$ such that $J_{N} > -\infty$ for each $N \geq N_{*}$.
	
	It only remains to rescale.  Toward that end, in view of the previous discussion, it is convenient to let $N(\epsilon) = \lceil \frac{\epsilon}{\delta(\epsilon)} \rceil$.  By construction, if we define a random variable $T_{\epsilon}$ by $T_{\epsilon} = \delta(\epsilon) J_{N(\epsilon)}$, then
		\begin{equation*}
			\theta^{\text{stripe}}(\delta(\epsilon)^{-1} s) = 1 \quad \text{for each} \quad s \in [T_{\epsilon} - M \epsilon, T_{\epsilon} + M \epsilon].
		\end{equation*}
	provided $\epsilon$ is so small that $N(\epsilon) \geq N_{*}$.  Note, in addition, that if $\epsilon$ is small enough, then (keeping also in mind that $ \beta = ( \gamma - 1 )^{-1} $)
		\begin{equation*}
			\varrho K^{\frac{\beta}{1 + \beta}} N(\epsilon)^{1 + \beta}/4 + \lceil M \rceil N(\epsilon) \leq \varrho K^{\frac{\beta}{1 + \beta}} N(\epsilon)^{1 + \beta} / 2 \leq \varrho \delta^{-1}.
		\end{equation*}
	In particular, for such $\epsilon$, the inclusion $[T_{\epsilon} - M \epsilon, T_{\epsilon} + M\epsilon] \subseteq [-\varrho,\varrho]$ also holds.
		\end{proof}


\section{Quasi-Periodic Media with Long Excursions from the Mean} \label{S: quasiperiodic examples}

In this section, we present an example demonstrating that rare events (or rather atypical configurations) are also relevant in almost periodic media.  The construction is based on Proposition \ref{P: quasiperiodic example done} below, which asserts the existence of almost periodic functions with long excursions from the mean.  More precisely, using Liouville numbers, we give a geometric proof of the existence of an almost periodic function $\theta^{\text{stripe}} : \mathbb{R} \to \{1,2\}$ for which the functional $\mathscr{F}_{\epsilon,\delta}$ has a nontrivial rare events regime, as in Theorem \ref{T: almost periodic counterexample}.

The main step in the construction is the proof of the following result.

	\begin{prop} \label{P: quasiperiodic construction} There is a function $f : \mathbb{R} \to \{1,2\}$, which is in $B^{p}(\mathbb{R})$ for any $1 \leq p < \infty$, such that, on the one hand, the mean value is greater than one
		\begin{align*}
			\lim_{ R \uparrow \infty } \frac{1}{2R} \int_{-R}^{R} f(t) \, dt > 1,
		\end{align*}
	and yet, on the other hand, $f$ attains its minimum on arbitrarily long intervals: More precisely, for any $M \in \mathbb{R}$, there is a $y \in \mathbb{R}$ such that 
		\begin{align} \label{E: weaker condition}
			f(t) = 1 \quad \text{for each} \, \, t \in [y,y+M].
		\end{align}
	\end{prop}

In fact, we will construct an $f$ as above together with a microscopic scaling $\epsilon \mapsto \delta(\epsilon)$ such that $\epsilon^{-1} \delta(\epsilon) \to 0$ as $\epsilon \searrow 0$ and the apparently stronger (but, in this setting, almost equivalent) condition holds: For each $M \in \mathbb{N}$ and $\varrho > 0$, there is a $T_{\epsilon} \in [-\varrho,\varrho]$ such that, for all $\epsilon$ small enough,
	\begin{align} \label{E: condition we need later}
		f(\delta(\epsilon)^{-1}s) = 1 \quad \text{for each} \quad s \in [T_{\epsilon} - M \epsilon,T_{\epsilon} + M \epsilon] \subseteq [-\varrho,\varrho].
	\end{align}
	
	\begin{proof}[Proof of Theorem \ref{T: almost periodic counterexample}] First, we invoke Proposition \ref{P: quasiperiodic example done done} below to obtain an almost periodic function $\theta^{\text{stripe}} = f$ taking values in $\{1,2\}$, with mean value greater than $1$, and satisfying \eqref{E: condition we need later}.  Then we let $\tilde{\theta}$ be any positive $\mathbb{Z}^{d}$-periodic function with $\int_{\mathbb{T}^{d}} \tilde{\theta} = 1$, let $\theta(y) = \theta^{\text{stripe}}(y_{1}) \tilde{\theta}(y)$, and conclude using Theorem \ref{T: general rare events} that, for any $\varrho > 0$,
		\begin{align*}
			\limsup_{ \epsilon \downarrow 0 } \varrho^{-(d-1)} \sigma_{\epsilon,\delta(\epsilon)} ( e_{1} ; [-\varrho/2,\varrho/2]^{d}) \leq \sigma_{W} .
		\end{align*}
	Since $\bar{\sigma}(e_{1}) = \sigma_{W} \sqrt{\bar{\theta}}$ with $\bar{\theta}$ the mean value of $\theta$ (see \eqref{E: mean value} above), to conclude that \eqref{E: almost periodic example equation} holds, it only remains to prove that $\bar{\theta} > 1$.  Here we argue by contradiction: If $\bar{\theta} = 1$, then, by the definition of mean value and the assumption that $\int_{\mathbb{T}^{d}} \tilde{\theta} = 1$,
		\begin{align*}
			0 = \bar{\theta} - 1 = \lim_{ R \uparrow \infty} \frac{ 1 }{ ( 2 R )^{ d } } \int_{ [ -R, R ]^{d} } ( \theta^{\text{stripe}}(y_{1}) - 1 ) \tilde{\theta}(y) \, dy.
		\end{align*}
	On the other hand, since $\theta^{\text{stripe}}$ takes values in $\{1,2\}$ by construction and $\tilde{\theta}$ is positive, this implies 
		\begin{align*}
			\lim_{ R \uparrow \infty } \frac{ 1 }{ 2 R } | \{ y_{1} \in [-R,R] \, \mid \, \theta^{\text{stripe}} (y_{1}) = 2 \} | = 0,
		\end{align*}
	and, thus, 
		\begin{align*}
			\lim_{ R \uparrow \infty } \frac{ 1 }{ 2 R } \int_{ -R }^{ R } \theta^{\text{stripe}} \, dy_{1} = 1.
		\end{align*}
	This contradicts the fact that the mean value of $\theta^{\text{stripe}}$ is greater than $1$. \end{proof}

\subsection{Idea of the Construction} \label{S: idea of construction} The idea of the construction of $f$  is based on the geometry of lines projected down to the torus $\mathbb{T}^{2}$. In what follows, denote by $[y] \in \mathbb{T}^{2}$ the equivalence class of $y \in \mathbb{R}^{2}$ modulo the action of $\mathbb{Z}^{2}$. Given an arbitrary unit vector $\eta \in S^{1}$, if, for a given $y \in \mathbb{R}^{2}$, we project the line $t \mapsto y + t \eta$ into $\mathbb{T}^{2}$, there are two possibilities:
	\begin{itemize}
		\item[(1)] The image $\{[y + t \eta] \, \mid \, t \in \mathbb{R}\}$ is a closed curve in $\mathbb{T}^{2}$. In particular, it is a compact set and determines a one-dimensional smooth submanifold of $\mathbb{T}^{2}$, see Figure \ref{fig:rational-line}.
		\item[(2)] The image $\{[y + t \eta] \, \mid \,t \in \mathbb{R}\}$ is a dense subset of $\mathbb{T}^{2}$, see Figure \ref{F: E_N}. It is an immersed submanifold of $\mathbb{T}^{2}$, but not embedded, and it is homeomorphic to $\mathbb{R}$.
	\end{itemize} 
It is not hard to show that situation (1) occurs if and only if there is an integer $k \in \mathbb{Z}^{2}$ such that $\eta = \frac{k}{|k|}$. In that case, we say that $\eta$ is a \emph{rational direction}; otherwise, we say that $\eta$ is an \emph{irrational direction}. 

Evidently, the set of rational directions is dense in $S^{1}$. Thus, given an irrational direction $\eta$, it is possible to find a sequence of rational directions $(\eta_{N})_{N \in \mathbb{N}}$ such that $\eta_{N} \to \eta$ as $N \to \infty$. If $L_{N}$ denotes the length of the curve $\{[y + t \eta_{N}] \, \mid \, t \in \mathbb{R}\}$ (which is independent of $y$), then $L_{N} \to \infty$ as $N \to \infty$. At the same time, if $N$ is large, the segment $\{[y + t \eta] \, \mid \, 0 \leq t \leq T_{N}\}$ should remain close to its counterpart $\{[y + t \eta_{N}] \, \mid \, 0 \leq t \leq T_{N}\}$ for a large time $T_{N}$.

\begin{figure}

\centering

\resizebox{\textwidth}{!}{
\begin{tikzpicture}
	\draw (1.5,0) rectangle (1.5 + 5,5);
	\draw [dotted] (1.5, 0) -- (1.5 , -0.5);
	\draw [dotted] (1.5 + 5,0) -- (1.5 + 5, -0.5);
	\draw [|-|] ( 1.5, -0.5 ) -- node[below] { $ 1 $ } ( 1.5 + 5, -0.5);

	\foreach \a / \b / \c / \d in {    0.0020 /     4.5017 /     0.5980 /     4.9983,     0.6020 /     0.0017 /     4.9980 /     3.6650,     0.0020 /     3.6683 /     1.5980 /     4.9983,     1.6020 /     0.0017 /     4.9980 /     2.8317,     0.0020 /     2.8350 /     2.5980 /     4.9983,     2.6020 /     0.0017 /     4.9980 /     1.9983,     0.0020 /     2.0017 /     3.5980 /     4.9983,     3.6020 /     0.0017 /     4.9980 /     1.1650,     0.0020 /     1.1683 /     4.5980 /     4.9983,     4.6020 /     0.0017 /     4.9980 /     0.3317,     0.0020 /     0.3350 /     4.9980 /     4.4983}
		\draw ( 1.5 + \a, \b ) -- ( 1.5 + \c, \d );

	\draw[->] (-9,-1) -- (-9,5.5);		
	\draw[->] (-9-1,0) -- (-9+5.5,0);
	\foreach \r in {1, ..., 5 } {
		\draw (-9 + \r , -0.1) -- (-9 + \r , 0.1);	
		\draw (-9 - 0.1, \r) -- (-9 + 0.1, \r);
	}
	\draw [dotted] (-9+1, 0) -- (-9+1, -0.5);
	\draw [dotted] (-9+2, 0) -- (-9+2, -0.5);
	\draw [|-|] (-9+1,-0.5) -- node[below] { $ 1 $ } (-9+2,-0.5);

	\draw (1/4 - 9, - 7/24) -- (6 + 1/4 -9, 113/24 );
	\draw[dotted] (6 + 1/4 -9, 113/24 ) -- (7 - 9, 16/3);
	\draw[dotted] ( -1/2 - 9, - 11/12 ) -- ( 1/4 - 9, - 7 / 24 );
	
	\draw[->] (-2,2) -- node[above] { $ (x,y) \mapsto [ (x,y) ] $ } ( 0, 2 );
	
\end{tikzpicture}
}

\caption{Projection of the rational line $ \{ ( x, y ) ~|~ y = \frac{5}{6} x - \frac{1}{2} \} \subseteq \R^2 $ on the torus $ \mathbb{T}^2 $}

\label{fig:rational-line}

\end{figure}
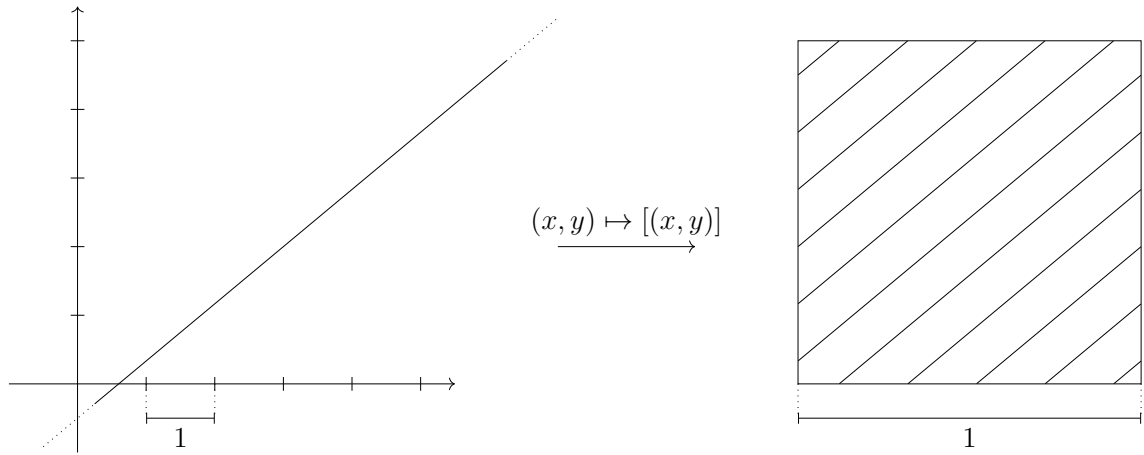

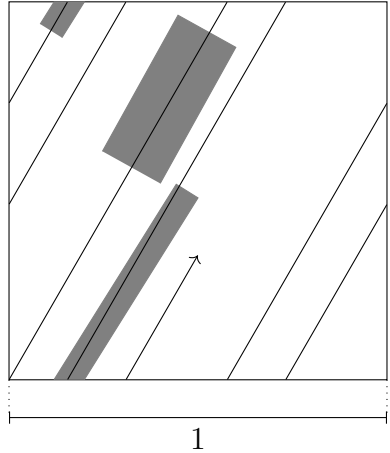
\begin{figure}

\centering

\begin{tikzpicture}

	\draw (0,0) rectangle (5,5);
	\draw [dotted] (0, 0) -- (0, -0.5);
	\draw [dotted] (5,0) -- (5, -0.5);
	\draw [|-|] (0, -0.5 ) -- node[below] { $ 1 $ } (5, -0.5);

	\filldraw[gray]  (    2.0000,     2.5981) --  (    3.0000,     4.3981) --  (    2.2358,     4.8226) --  (    1.2358,     3.0226);
	
	\filldraw[gray] (    0.7000,     4.5263) --  (    0.9961,     4.9980) --  (    0.5961,     4.9980) --  (    0.4124,     4.7061);
	\filldraw[gray] (    0.9961,     0.0020) --  (    2.5000,     2.4063) --  (    2.2124,     2.5861) --  (    0.5961,     0.0020);

	\foreach \a / \b / \c / \d in{    0.0020 /     0.0035 /     2.8848 /     4.9965,     2.8888 /     0.0035 /     4.9980 /     3.6568,     0.0020 /     3.6637 /     0.7715 /     4.9965,     0.7755 /     0.0035 /     3.6583 /     4.9965,     3.6623 /     0.0035 /     4.9980 /     2.3170,     0.0020 /     2.3240 /     1.5450 /     4.9965}
		\draw ( \a, \b ) -- ( \c, \d );
		
	\foreach \a / \b / \c / \d in {1.5490 /  0.0035 / 2.4980 / 1.6472 }
		\draw[->] (\a, \b) -- (\c, \d);
		
\end{tikzpicture}

\caption{Projection of (parts of) the irrational line $ \{ ( x, y ) ~|~ y = \sqrt{3} x \} \subseteq \R^2 $ on $ \mathbb{T}^2 $ and a possible choice of boxes $ \{ E_n \}_{n \in \N } $}
\label{F: E_N}

\end{figure}

More precisely, if we choose a base point $y_{N} \in \R^{2} $ and thicken the line $\{[y_{N} + t \eta_{N}] \, \mid \, t \in \mathbb{R}\}$ to a strip $E_{N}$ as in Figure \ref{F: E_N}, then the segment $\{[y_{N} + t \eta] \, \mid \, 0 \leq t \leq T_{N}\}$ should remain inside $E_{N}$ for a time $T_{N}$ that grows to infinity as $N \to \infty$. If it is possible to choose the width of the strip $E_{N}$ to be small enough, then we could hope for the areas to be summable, that is,
	\begin{equation*}
		\sum_{n = 1}^{\infty} |E_{n}| < \infty,
	\end{equation*}
where $|E_{n}|$ is the Lebesgue measure of $E_{n}$ considered as a subset of $\mathbb{T}^{2}$. Now if $E = \bigcup_{n = 1}^{\infty} E_{n}$ and if $f : \mathbb{R} \to [1,2]$ is the function
	\begin{equation*}
		f(t) = 2 - \chi_{E}(t \eta),
	\end{equation*}
then, morally speaking, since the path $t \mapsto t \eta $ winds densely around the torus, we expect that $f = 1$ in an interval $J_{N} \subseteq \mathbb{R}$ of length approximately equal to $T_{N}$.

It only remains to quantify the argument, showing, in particular, that it is possible to choose the strips $\{E_{n}\}_{n \in \mathbb{N}}$ so that we simultaneously have that $T_{N} \to \infty$ and that $\sum_{n = 1}^{\infty} |E_{n}| < \infty$. 

\subsection{Liouville Numbers} To make the discussion above rigorous, we use the notion of a \emph{Liouville number}. Specifically, in what follows, fix a $\lambda \in (0,1) \setminus \mathbb{Q} $ for which there is a sequence $(p_{n})_{n \in \mathbb{N}}, (q_{n})_{n \in \mathbb{N}} \subseteq \mathbb{N}$ such that
	\begin{equation} \label{E: liouville}
		q_{n} > 1, \quad \text{gcd}(p_{n},q_{n}) = 1, \quad 0 < \left| \lambda - \frac{p_{n}}{q_{n}} \right| < q_{n}^{-n} \quad \text{for each} \quad n \in \mathbb{N}.
	\end{equation}
Such real numbers are referred to as Liouville numbers, and they form a dense subset of $(0,1)$ (see, e.g., \cite[Chapter 2]{oxtoby}).  In what follows, it will be useful to define $\lambda_{N} = \frac{p_{N}}{q_{N}}$.

In the end, we will define the unit vector $\eta$ by $\eta = (1 + \lambda^{2})^{-\frac{1}{2}} (1,\lambda)$ and the sequence of unit vectors $(\eta_{N})_{N \in \mathbb{N}}$ by $\eta_{N} = (1 + \lambda_{N}^{2})^{-\frac{1}{2}} (1,\lambda_{N})$. Notice that $(\eta_{N})_{N \in \mathbb{N}}$ is a sequence of rational directions approximating $\eta$, and since Liouville numbers are irrational, $\eta$ is an irrational direction. 
	
\begin{remark} \label{rmk:liouville-q-divergence} It follows from \eqref{E: liouville} that $q_{n} \to \infty$ as $n \to \infty$. \end{remark}
	
\subsection{Rational Cycles} In what follows, for a given pair of relatively prime integers $p,q$, let $e \in S^{1}$ be the unit vector perpendicular to $(p,q)$ given by 
	\begin{align*}
		e = \frac{ ( - p, q ) }{ \sqrt{ p^{2} + q^{2} } }
	\end{align*}
and define the smooth closed curve $C_{r}(p/q) \subseteq \mathbb{T}^{2}$ for $r \in \mathbb{R}$ by
	\begin{align*}
		C_{r}(p/q) = \{ [ r e + t(1,p/q) ] \, \mid \, t \in \mathbb{R} \}.
	\end{align*}
Since any point $x \in \mathbb{R}^{2}$ can be written in the form $x = r e + t(1,p/q)$ for some $t$, these sets cover $\mathbb{T}^{2}$:
	\begin{align*}
		\mathbb{T}^{2} = \bigcup_{r \in \mathbb{R}} C_{r}(p/q).
	\end{align*}
In order to make the sketch in Section \ref{S: idea of construction} above rigorous, we will be interested both in the length of $C_{r}(p/q)$ and in the period of the set function $r \mapsto C_{r}(p/q)$.

We begin by computing the length $L(p/q)$ of $C_{r}(p/q)$. 

	\begin{lemma} \label{L: length cycle} $L(p/q) = \sqrt{ p^{2} + q^{2 }}$. \end{lemma}
	
		\begin{proof} Observe that $L(p/q)$ is given by
			\begin{equation*}
				L(p/q) = P(p/q) \sqrt{ 1 + ( p/ q)^{2} } , \quad \text{where} \quad P(p/q) =  \inf \left\{ T > 0 \, \mid \, (T,Tp/q) \in \mathbb{Z}^{2} \right\}.
			\end{equation*}
In particular, $P(p/q) \in \mathbb{N}$ and $P(p/q) \leq q$. We now argue that in fact $ P(p/q) = q$. Indeed, since
	\begin{equation*}
		\frac{P(p/q) p}{q} \in \mathbb{N}
	\end{equation*}
and $p$ and $q$ are relatively prime, it follows that $q$ divides $P(p/q)$. Thus, $ P(p/q) \in \{1,q\}$. Since $ \frac{ p }{ q } \notin \mathbb{N} $, $ P(p/q) \neq 1 $ by definition, so this proves $ P(p/q) = q $.  \end{proof}

Next, let $R(p/q)$ be the period of $r \mapsto C_{r}(p/q)$, that is,
	\begin{align*}
		R(p/q) = \min \left\{ R > 0 \, \mid \, C_{r + R}(p/q) = C_{r}(p/q) \, \, \text{for each} \, \, r \in \mathbb{R} \right\}.
	\end{align*}
To see that this is well-defined, observe that, for any $k \in \mathbb{Z}^{2}$ and any $r, t \in \mathbb{R}$,
	\begin{align*}
		( re + t(1,p/q) + k ) \cdot e = (r + k \cdot e ) \cdot e,
	\end{align*}
and, thus, since we work in $\mathbb{T}^{2}$, $C_{r + k \cdot e }(p/q) = C_{r}(p/q)$.  The next result shows that $R(p/q) = L(p/q)^{-1}$.

	\begin{lemma} \label{L: period cycle} $R(p/q) = \frac{1}{ L(p/q) } = \frac{ 1 }{ \sqrt{ p^{2} + q^{2} } }$. \end{lemma}
	
		\begin{proof} Consider the set $A \subseteq \mathbb{R}^{2}$ given by 
			\begin{align*}
				A = \{ r e + t (1,p/q) \, \mid \, r \in [0,R(p/q)), \, \, t \in [0,L(p/q)) \}.
			\end{align*}
		By the definition of $R(p/q)$ and $L(p/q)$, the projection $y \mapsto [y]$ sends $A$ bijectively onto $\mathbb{T}^{2}$.  Therefore, since the Jacobian equals one,
			\begin{align*}
				1 = | \mathbb{T}^{2} | = |A| = L(p/q) R(p/q).
			\end{align*}
		\end{proof}

\subsection{Construction of $f$} \label{S: construction quasiperiodic} We now prove Proposition \ref{P: quasiperiodic construction}.   Let $\{p_{n}\}_{n \in \mathbb{N}}$ and $\{q_{n}\}_{n \in \mathbb{N}}$ be integers satisfying \eqref{E: liouville} and define $\lambda_{N} = \frac{p_{N}}{q_{N}}$.   Let $L_{N} = L(\lambda_{N})$ and $R_{N} = R(\lambda_{N})$, so that, by Propositions \ref{L: length cycle} and \ref{L: period cycle},
	\begin{equation} \label{E: lengths}
		L_{N} = \sqrt{ p_{N}^{2} + q_{N}^{2} }, \quad R_{N} = \frac{ 1 }{ L_{N} }.
	\end{equation}
Since $\lambda \in (0,1)$ and $\lambda_{N} \to \lambda$ as $N \to \infty$, there is a constant $C > 0$ such that $p_{N} \leq C q_{N}$ for all $N$, hence
	\begin{equation*}
		(1 + C^{2})^{-\frac{1}{2}} q_{N}^{-1} \leq R_{N} \leq q_{N}^{-1}.
	\end{equation*}

For $N \in \mathbb{N}$, define $e_{N} \in S^{1}$ by 
	\begin{equation*}
		e_{N} = \frac{(-p_{N},q_{N})}{\sqrt{p_{N}^{2} + q_{N}^{2}}}.
	\end{equation*}
Observe that, for any $N$, we have
	\begin{equation*}
		(1,\lambda) \cdot e_{N} = R_{N} (-p_{N} + \lambda q_{N}) = R_{N} q_{N} (\lambda - \lambda_{N}).
	\end{equation*}
In particular,
	\begin{equation*}
		(1 + C^{2})^{-\frac{1}{2}} |\lambda - \lambda_{N}| \leq |(1,\lambda) \cdot e_{N}| \leq |\lambda - \lambda_{N}|.
	\end{equation*}

Finally, we construct a set $E \subseteq \mathbb{T}^{2}$ in the following way. To begin with, let $E_{1}$ be the set
	\begin{equation*}
		E_{1} = \left\{[r e_{1} + C_{0}(\lambda_{1})] \, \mid \, 0 \leq r \leq \frac{R_{1}}{3} \right\} = \bigcup_{r \in [0,\frac{R_{1}}{3}]} C_{r}(\lambda_{1}).
	\end{equation*}
Fix an arbitrary sequence $(m_{n})_{n \in \mathbb{N}} \subseteq \mathbb{N}$ such that $m_{N} \in \{0,\dots,3^{N} - 1\}$ for each $N \in \mathbb{N}$, and define $\{E_{n}\}_{n \in \mathbb{N}}$ via the rule
	\begin{equation*}
		E_{N} = \left\{[r e_{N} + C_{0}(\lambda_{N})] \, \mid \, \frac{m_{N} R_{N}}{3^{N}} \leq r \leq \frac{(m_{N} + 1) R_{N}}{3^{N}} \right\} = \bigcup_{3^{N} R_{N}^{-1} r \in [m_{N}, m_{N} + 1]} C_{r}(\lambda_{N}).
	\end{equation*}
Notice that, by definition and \eqref{E: lengths}, 	
	\begin{equation*}
		|E_{N}| = \frac{L_{N} \cdot R_{N}}{3^{N}} = 3^{-N}.
	\end{equation*}
Therefore,
	\begin{equation*}
		|E_{1} \cup \dots \cup E_{N}| \leq \sum_{n = 1}^{N} 3^{-n} = \frac{1}{2}.		\end{equation*}
In particular, if we set $E = \bigcup_{n = 1}^{\infty} E_{n}$, then $|E| \leq \frac{1}{2}$. 

In what follows, define two orthogonal directions $e, \eta \in S^{1}$ by 
	\begin{equation*}
		e = \frac{(\lambda, -1)}{\sqrt{1 + \lambda^{2}}}, \quad \eta = \frac{(1,\lambda)}{\sqrt{1 + \lambda^{2}}}
	\end{equation*}

	\begin{prop} \label{P: time}If $[y] \in E_{N}$ and if $T(y) \geq 0$ is defined by
		\begin{equation*}
			T(y) = \sup \left\{ t \geq 0 \, \mid \, [x + t \eta] \in E_{N} \right\},
		\end{equation*}
	then $T(y) \leq T_{N}$, where the sequence $\{T_{n}\}_{n \in \mathbb{N}}$ is given by
		\begin{equation*}
			T_{N} = (1 + \lambda^{2})^{\frac{1}{2}} 3^{-N} q_{N}^{-1} [\lambda - \lambda_{N}]^{-1}.
		\end{equation*}
	
	Moreover, there is a set $\tilde{E}_{N} \subseteq E_{N}$ such that $|\tilde{E}_{N}| = \frac{1}{2} |E_{N}|$ and, for each $[y] \in \tilde{E}_{N}$, we have
		\begin{equation*}
			T(y) \geq \frac{1}{2} T_{N}.
		\end{equation*}
	\end{prop} 

By the estimate \eqref{E: liouville} and Remark \ref{rmk:liouville-q-divergence}, we know that, for any $M > 1$, 
	\begin{equation} \label{E: superexponential liouville}
		\lim_{N \to \infty} \frac{ T_{N} }{ M^{N} } = \infty.
	\end{equation}
	
		\begin{proof} Define $\tilde{E}_{N} \subseteq E_{N}$ by
			\begin{align*}
				\tilde{E}_{N} &= \left\{ [r e_{N} + C_{0}(\lambda_{N})] \, \mid \, \frac{m_{N} R_{N}}{3^{N}} \leq r \leq \left( m_{N} + \frac{1}{2} \right) \frac{R_{N}}{3^{N}} \right\} \\
				&= \bigcup_{3^{-N} R_{N}^{-1} r \in [m_{N}, m_{N} + 1/2]} C_{r}(\lambda_{N}).
			\end{align*}
		Suppose that $ y \in \R^2 $ is such that $ [ y ] \in \tilde E_N $. For any $t \in [0,\infty)$, we have
			\begin{equation*}
				(y + t \eta) \cdot e_{N} = y \cdot e_{N} + t (1 + \lambda^{2})^{-\frac{1}{2}} R_{N} q_{N} (\lambda - \lambda_{N}).
			\end{equation*}
		Thus, the inequality $\frac{m_{N} R_{N}}{3^{N}} \leq (y + t \eta) \cdot e_{N} \leq \frac{(m_{N} + 1) R_{N}}{3^{N}}$ holds provided
			\begin{equation*}
				0 \leq t (1 + \lambda^{2})^{-\frac{1}{2}} q_{N} (\lambda - \lambda_{N}) \leq \frac{1}{2} 3^{-N},
			\end{equation*}
		or, in other words,
			\begin{equation*}
				0 \leq t \leq \frac{1}{2} (1 + \lambda^{2})^{\frac{1}{2}} 3^{-N} q_{N}^{-1} (\lambda - \lambda_{N})^{-1} = \frac{1}{2} T_{N}.
			\end{equation*}
	This implies that if $[y] \in \tilde{E}_{N}$, then
			\begin{equation*}
				\left\{ [y + t \eta] \, \mid \, 0 \leq t \leq \frac{1}{2} T_{N} \right\} \subseteq E_{N}.
			\end{equation*}
	
	At the same time, notice that $|\tilde{E}_{N}|$ can be computed explicitly
			\begin{equation*}
				|\tilde{E}_{N}| = \frac{1}{2} |E_{N}| \leq \frac{1}{2} 3^{-N}.
			\end{equation*}
		
		Finally, in general, if $x \in E_{N}$, then the same computations show $T(x) \leq T_{N}$. 
		\end{proof}
		
	We conclude as follows: Define $F : \mathbb{T}^{2} \to \{1,2\}$ by 
		\begin{equation*}
			F([y]) = 2 - \chi_{E}([y]).
		\end{equation*}
	Note that $\int_{\mathbb{T}^{2}} F > 1$ since $|E| \leq \frac{1}{2}$.  Given $[y] \in \mathbb{T}^{2}$, let $f_{y} : \mathbb{R} \to \{1,2\}$ be the function
		\begin{equation*}
			f_{y}(s) = F([y + s\eta]).
		\end{equation*}
	
	We will show that, for almost every $[y] \in \mathbb{T}^{2}$, the function $f_{y}$ has the properties described at the start of this section.  The argument involves basic concepts from ergodic theory.  Define a group action $(\tau_{t})_{t \in \mathbb{R}}$ of $\mathbb{R}$ on $\mathbb{T}^{2}$ via the formula $\tau_{t}[y] = [y + t \eta]$.  Note that this action preserves Lebesgue measure.  Furthermore, since $\eta \in \mathbb{R}^{2} \setminus \mathbb{R} \mathbb{Z}^{2}$ (i.e., $\eta$ is an irrational direction), it is well-known that this action is ergodic; see, for instance, \cite[Section 2.4]{morfe_thesis} or \cite[Appendix B.1]{morfe_variational}.

	Since $F \in L^{\infty}(\mathbb{T}^{d})$, there is a sequence of trigonometric polynomials $(F_{n})_{n \in \mathbb{N}} \subseteq C^{\infty}(\mathbb{T}^{d})$ such that, for any $p \in [1,\infty)$,
		\begin{equation*}
			\lim_{n \to \infty} \int_{\mathbb{T}^{2}} |F - F_{n}|^{p} = 0.
		\end{equation*}
	By the ergodic theorem applied to $(\tau_{s})_{s \in \mathbb{R}}$, we have, for almost every $[y] \in \mathbb{T}^{2}$,
		\begin{equation} \label{E: ergodic thing}
			\lim_{R \uparrow \infty} \frac{1}{2R} \int_{-R}^{R} |f_{y}(t) - F_{n}([x + t \eta])|^{p} \, dt = \int_{\mathbb{T}^{2}} |F - F_{n}|^{p}.
		\end{equation}
	Thus, by definition, $f_{y} \in \bigcap_{1 \leq p < \infty} B^{p}(\mathbb{R})$ for almost every $y$.
	
	Applying Proposition \ref{P: time} in conjunction with the ergodic theorem, we prove that, for a typical $[y] \in \mathbb{T}^{2}$, the function $f_{y}$ undergoes arbitrarily long excursions from the mean $\int_{\mathbb{T}^{2}} F$. 
	
		\begin{prop} \label{P: quasiperiodic example done} For Lebesgue almost every $[y] \in \mathbb{T}^{2}$, the function $f_{y}$ has the following property: Given any $M \geq 1$, there is a $\tau \in \mathbb{R}$ such that 
			\begin{equation*}
				f_{y}(t) = 1 \quad \text{for each} \quad t \in [\tau,\tau + M].
			\end{equation*}
		Moreover, for almost every $[y]$,
			\begin{equation} \label{E: mean property}
				\lim_{R \uparrow \infty} \frac{1}{2R} \int_{t' -R}^{t' + R} f_{y}(t) \, dt = 2 - |\tilde{E}| \geq \frac{3}{2} \quad \text{for each} \quad t' \in \mathbb{R}.
			\end{equation}
		\end{prop}
		
Notice that Proposition \ref{P: quasiperiodic construction} follows.
		
			\begin{proof} Let $\tilde{E}^{\text{return}}_{n}$ be the set
				\begin{equation*}
					\tilde{E}^{\text{return}}_{n} = \{[y] \in \mathbb{T}^{2} \, \mid \, [y + s \eta] \in \tilde{E}_{n} \, \, \text{for some} \, \, s \in \mathbb{R}\},
				\end{equation*}
			where $\{\tilde{E}_{n}\}_{n \in \mathbb{N}}$ are the sets constructed in Proposition \ref{P: time}.  Since $|\tilde{E}_{n}| > 0$ and the group action $(\tau_{s})_{s \in \mathbb{R}}$ is ergodic, the ergodic theorem implies that $|\tilde{E}^{\text{return}}_{n}| = 1$. Thus, if $\tilde{E}^{\text{return}}$ is given by
				\begin{equation*}
					\tilde{E}^{\text{return}} = \bigcap_{n = 1}^{\infty} \tilde{E}^{\text{return}}_{n},
				\end{equation*}
			then $|\tilde{E}^{\text{return}}| = 1$. 
			
			Suppose that $y \in \mathbb{R}^{2}$ and $[y] \in \tilde{E}^{\text{return}}$. Fix an $N \in \mathbb{N}$. Since $[y] \in \tilde{E}^{\text{return}}_{N}$, there is an $r \in \mathbb{R}$ such that $[y + r \eta] \in \tilde{E}_{n}$. In particular, by Proposition \ref{P: time},
				\begin{equation*}
					[y + (r + t) \eta] \in E \quad \text{for each} \quad t \in [0,2^{-1}T_{N}].
				\end{equation*}
			In terms of $f_{y}$, this implies
				\begin{equation*}
					f_{y}(s) = 1 \quad \text{for each} \quad s \in [r, r + 2^{-1}T_{N}].
				\end{equation*}
			Since $T_{N} \to \infty$ as $N \to \infty$ by \eqref{E: superexponential liouville}, this concludes the proof. 
			
		Finally, \eqref{E: mean property} follows directly from \eqref{E: ergodic thing}. \end{proof}
			
		\begin{remark} \label{R: no uniformity} Notice that the, due to the long intervals where $f_{y}$ equals one, the limit in \eqref{E: ergodic thing} is \emph{not} uniform in the base point $t'$. Since $f_{y} \in \cap_{1 \leq p < \infty} B^{p}(\mathbb{R})$ for almost every $y$, this demonstrates what is already well-known in the literature, namely, that functions in $B^{p}(\mathbb{R}^{d})$ with $p < \infty$ do not necessarily possess the uniform averaging property of functions in $B^{\infty}(\mathbb{R}^{d})$, see e.g.~\cite[(13)]{part1}. \end{remark}
		
The previous result can be strengthened to establish the existence of the scale $\epsilon \mapsto \delta(\epsilon)$ invoked previously in the proof of Theorem \ref{T: almost periodic counterexample} above (see \eqref{E: condition we need later}). 

\begin{prop} \label{P: quasiperiodic example done done} There is a choice of scaling $\epsilon \mapsto \delta(\epsilon)$ such that $\epsilon^{-1} \delta(\epsilon) \to 0$ as $\epsilon \to 0$ and the following property holds for Lebesgue almost every $[y] \in \mathbb{T}^{2}$: For any $M \in \mathbb{N}$ and any $\varrho > 0$, there is an $T_{\epsilon} \in [-\varrho,\varrho]$ such that, for $\epsilon$ sufficiently small,
	\begin{equation*}
		f_{y}(\delta(\epsilon)^{-1} s) = 1 \quad \text{for each} \quad s \in [T - M \epsilon, T + M \epsilon] \subseteq [-\varrho, \varrho].
	\end{equation*}
\end{prop}

\begin{proof} This follows from Proposition \ref{P: one dimensional weird thing helper} in the appendix and Proposition \ref{P: quasiperiodic example done}. To see that Proposition \ref{P: one dimensional weird thing helper} applies, define the probability space $(\Omega,\mathcal{F},\mathbb{P})$ by letting $\Omega = \mathbb{T}^{2}$, $\mathcal{F}$ be the Borel $\sigma$-algebra, and $\mathbb{P}$ be the Lebesgue measure on $\mathbb{T}^{2}$.  The group of transformations $(\tau_{s})_{s \in \mathbb{R}}$ is exactly as above.  We then apply Proposition \ref{P: one dimensional weird thing helper} to the function 
		\begin{equation*}
			\theta(s,[y]) = F(\tau_{s}[y]),
		\end{equation*}
	which satisfies condition \eqref{E: minimization thing redux} by Proposition \ref{P: quasiperiodic example done}. \end{proof}

\section*{Appendix: Qualitative Approach to Rare Events in 1D} \label{A: one dimensional weird thing}

In this appendix, we assume that $\theta : \mathbb{R} \to [\theta_{*},\theta^{*}]$ is a stationary ergodic field such that for any $M \geq 1$, 
	\begin{equation} \label{E: minimization thing redux}
		\mathbb{P}\{\exists x \in \mathbb{R} \, \, \text{such that} \, \, \theta(y) = \theta_{*} \, \, \text{for a.e.} \, \, y \in [x, x + M)\} = 1.
	\end{equation}
The goal of this appendix is to prove that this property can be upgraded, as indicated in Proposition \ref{P: one dimensional weird thing helper} below, so that $\theta$ satisfies the assumption of Theorem \ref{T: general rare events} with probability one.

To make things precise, and so that the results can be applied in the almost periodic context of Section \ref{S: construction quasiperiodic}, we specifically let $(\Omega,\mathcal{F},\mathbb{P})$ be a probability space and assume there is a jointly measurable map $\tau : \mathbb{R} \times \Omega \to \Omega$, $(s,\omega) \mapsto \tau_{s} \omega$ satisfying $\tau_{0} \omega = \omega$, $\tau_{s + t} = \tau_{s} \circ \tau_{t}$, and that preserves $\mathbb{P}$ in the sense that
	\begin{align*}
		\mathbb{P} ( \tau_{s} A ) = \mathbb{P} ( A ) \quad \text{for each} \, \, s \in \mathbb{R}, \, \, A \in \mathcal{F},
	\end{align*}
and which is ergodic, meaning that if $A \in \mathcal{F}$ and $\tau_{s} A = A$ for each $s \in \mathbb{R}$, then $\mathbb{P}(A) \in \{0,1\}$.  We fix a random variable $\Theta : \Omega \to [\theta_{*},\theta^{*}]$ and define the random field $\theta$ by 
	\begin{align*}
		\theta(s) = \Theta ( \tau_{s} \omega ).
	\end{align*}
	
\begin{prop} \label{P: one dimensional weird thing helper} If $\theta$ satisfies \eqref{E: minimization thing redux}, then there is a function $\gamma \mapsto \mathcal{R}(\gamma)$ such that $\lim_{\gamma \downarrow 0} \gamma \mathcal{R}(\gamma) = \infty$ and with the following property: If $\epsilon \mapsto \delta(\epsilon)$ is any scaling such that
		\begin{equation*}
			\lim_{\epsilon \downarrow 0} \epsilon^{-1} \delta(\epsilon) = \lim_{\epsilon \downarrow 0} \delta(\epsilon) \mathcal{R}(\epsilon^{-1} \delta(\epsilon)) = 0,
		\end{equation*}
	then, for any (deterministic) $M \in \mathbb{N}$ and $\varrho > 0$, there is a random variable $\mathcal{E}(M,\varrho) \geq 0$ such that $\mathbb{P}\{\mathcal{E}(M,\varrho) > 0\} = 1$ with the following property: For any $\epsilon < \mathcal{E}(M,\varrho)$, there is a point $T_{\epsilon} \in [-\varrho,\varrho]$ such that
		\begin{align*}
			\theta(\delta(\epsilon)^{-1} s) = \theta_{*} \quad \text{for each} \quad s \in [T_{\epsilon} - M \epsilon , T_{\epsilon} + M \epsilon] \subseteq [-\varrho, \varrho].
		\end{align*}\end{prop}

The proof will use the next observation, which is completely elementary. Before stating the result, let us define the arrival times $(\widehat{T}_{j})_{j \in \mathbb{N}}$ via the formula
	\begin{align*}
		\widehat{T}_{j} &= \inf \{ R \geq 0 \, \mid \, \exists x \in [-R,R] \, \, \text{such that} \, \, \theta(y) = \theta_{*} \, \, \text{for each} \, \, y \in [x, x + j) \}.
	\end{align*}

	\begin{lemma} \label{L: egoroff type lemma} If \eqref{E: minimization thing redux} holds, then $\mathbb{P}\{ \widehat{T}_{j} < \infty\} = 1$ for each $j \in \mathbb{N}$. In particular, there is a sequence $( \widehat{L}_{j})_{j \in \mathbb{N}}$ such that
		\begin{equation*}
			\lim_{j \to \infty} \widehat{L}_{j} = \infty, \quad \sum_{j = 1}^{\infty} \mathbb{P}\{T_{j} \geq \widehat{L}_{j}\} < \infty.
		\end{equation*} 
	\end{lemma}
	
		\begin{proof} The criterion \eqref{E: minimization thing redux} clearly implies the first statement. To construct $( \widehat{L}_{j})_{j \in \mathbb{N}}$, recall that since $\lim_{ \ell \rightarrow \infty } \P\{ \widehat{T}_j < \ell \} = \mathbb{P}\{\widehat{T}_{j} < \infty\} = 1$, it is possible to choose a scale $ \widehat{L}_{j} \geq j$ such that $\mathbb{P}\{\widehat{T}_{j} \geq \widehat{L}_{j}\} \leq 2^{-j}$. \end{proof}
		
	By the Borel-Cantelli Lemma, we can fix a $J \in \mathbb{N}$ such that if $\Omega_{0} \in \mathcal{F}$ is the event
		\begin{equation} \label{E: xi event}
			\Omega_{0} = \{ \widehat{T}_{j} < \widehat{L}_{j} \, \, \text{for all} \, \, j \geq J\},
		\end{equation}
	then $\mathbb{P}(\Omega_{0}) \geq \frac{1}{2}$. 
	
With the sequence $(\widehat{T}_{j})_{j \in \mathbb{N}}$ and event $\Omega_{0}$ just defined, we now prove the proposition.

	\begin{proof}[Proof of Proposition \ref{P: one dimensional weird thing helper}] We first work with a specific sequence of $ \gamma $'s to make use of the above lemma. Later on, we pass to a continuum choice of scales.
	
	Let us write $ \gamma_{n} = 2^{-n} $. Our goal is to find a sequence $(\mathcal{R}_{n})_{n \in \mathbb{N}}$ such that 
		\begin{equation*}
			\lim_{n \to \infty} \gamma_{n} \mathcal{R}_{n} = \infty
		\end{equation*} 
	and the statement of the theorem holds provided we define $\gamma \mapsto \mathcal{R}(\gamma)$ by the rule:
		\begin{align} \label{E: definition of tau appendix}
	{ \rm Given }
	\quad \gamma_{ n + 1 } \leq \gamma < \gamma_{n} ,
	\quad { \rm define } \quad
	\mathcal{R} ( \gamma ) = \mathcal{R}_{ n }.
\end{align}
		

To prove this, we begin by defining the sequence $(\mathcal{R}_{n})_{n \in \mathbb{N}}$, choosing a suitable event of probability one, and then proving it has the desired properties.  

In the proof that follows, we use hats on quantities that live on the microscopic scale, that is, relating to properties of $\theta$.  Quantities without hats are on the mesoscopic scale, i.e., relating to the rescaled field $\theta(\gamma^{-1} \cdot)$.

\textit{Step 1 (Choice of $(\mathcal{R}_{n})_{n \in \mathbb{N}}$).} Let $( \widehat{L}_{j})_{j \in \mathbb{N}}$ be the sequence from Lemma \ref{L: egoroff type lemma}. Choose a sequence $( \widehat{j}_{n})_{n \in \mathbb{N}} \subseteq \mathbb{N}$ such that $\lim_{n \to \infty} \widehat{j}_{n} = \infty$ and $ j_n \coloneqq \gamma_{n} \widehat{j}_{n} \geq n$ for each $n \in \mathbb{N}$. Define the sequence $(\mathcal{R}_{n})_{n \in \mathbb{N}}$ by $ \mathcal{R}_{n} = 2 ( \widehat{L}_{ \widehat{j}_{n}} + \widehat{j}_{n} + 1 )$. By construction, $\gamma_{n} \mathcal{R}_{n} \to \infty$ as $n \to \infty$.

With this relabeling of the scales $( \widehat{L}_{j})_{j \in \mathbb{N}}$, it is convenient to also relabel the arrival times $( \widehat{T}_{j})_{j \in \mathbb{N}}$. In particular, define $( \widehat{S}_{n})_{n \in \mathbb{N}}$ by $ \widehat{S}_{n} = \widehat{T}_{\widehat{j}_{n}}$. By definition of $( \widehat{T}_{j} )_{j \in \mathbb{N}}$, the sequence $( \widehat{S}_{n})_{n \in \mathbb{N}}$ is also determined by the rule
	\begin{align} \label{E: arrival time thing}
		\widehat{S}_{n} &= \inf \left\{ R \geq 0 ~ \middle| ~  \exists x \in [-R,R] \, \, \text{such that} \, \, \theta(y) = \theta_{*} \, \,  \text{for a.e.} \, \, y \in [x, x + \widehat{j}_{n}) \right\}.
	\end{align}
	
\textit{Step 2 (Invoking the ergodic theorem).} Recall the event $\Omega_{0}$ defined in \eqref{E: xi event} above. Let $\tilde{\Omega}_{0}$ be the following event:
	\begin{equation*}
		\tilde{\Omega}_{0} = \left\{ \omega \in \Omega \, \mid \,  \lim_{r \to \infty} (2r)^{-1} |\{y \in [ - r, r] \, \mid \, \tau_{y} \omega \in \Omega_0 \}| = \mathbb{P}(\Omega_{0}) \right\}.
	\end{equation*}
The ergodic theorem implies that $\mathbb{P}(\tilde{\Omega}_{0}) = 1$.

By definition of the set $ \tilde\Omega_0 $ we learn that for every $ \omega \in \tilde\Omega_0 $ the following holds: There exists a random variable $ N_* = N_* ( \omega ) \in \mathbb{N}$ such that for every $ n \geq N_* $ there is a $ - \gamma_n^{-1} \leq \widehat{a}_n \leq \gamma_n^{-1} $ such that $ \tau_{ \widehat{a}_n } \omega \in \Omega_0 $. We define
\begin{align*}
a_n \coloneqq \gamma_n \widehat{a}_n
\quad { \rm so ~ that } \quad
| a_n | \leq 1,
\quad
\tau_{ \gamma_n^{-1} a_n } \omega \in \Omega_0.
\end{align*}
In particular, by the definition of $\Omega_{0}$, there is a deterministic $N_{0} \in \mathbb{N}$ such that 
	\begin{equation*}
		\widehat{S}_{n}(\tau_{\widehat{a}_{n}} \omega) < \widehat{L}_{\widehat{j}_{n}} \quad \text{for each} \quad n \geq \max\{ N_{0},N_{*}\} .
	\end{equation*}
This means it is possible to find a sequence $ ( \widehat{x}_n )_{ n \in \N } $ such that $ | \widehat{x}_n - \widehat{a}_{n} | < \widehat{L}_{ \widehat{j}_n } $ and $ \theta = \theta_* $ in $ [ \widehat{x}_n , \widehat{x}_n + \widehat{ j }_n ] $. Setting $ x_n \coloneqq \gamma_n \widehat{x}_n $ this implies $ | x_n - a_n | < \widehat{L}_{ \widehat{j}_n } \gamma_n $ and
\begin{align*}
	\theta(\gamma_n^{-1} y) = \theta_{*} \quad \text{for a.e.} \quad y \in [ x_{n}, x_{n} + j_{n}),
\end{align*}
where we recall that $ j_n = \gamma_{n} \widehat{j}_{n} \geq n $ by construction. 

\textit{Step 3 (Introducing the scaling $\gamma \mapsto R(\gamma)$).} Suppose as in \eqref{E: definition of tau appendix} that $\gamma > 0$ satisfies $\gamma_{n+1} \leq \gamma \leq \gamma_{n}$ for some $n \geq \max\{N_{0},N_{*}\}$. First, note that by the above construction (recall also $ \widehat{j}_n = \frac{j_n}{ \gamma _n } $), if $x_{n} > 0$, then
\begin{align*}
[ \gamma \gamma_n^{-1} x_n , \gamma \gamma_n^{-1}  x_n + \gamma \gamma_{n}^{-1} j_n ] 
& \subseteq [ 2^{-1} x_n , x_n + j_n ] \\
& \subseteq [ 2^{-1} a_n - 2^{-1} \widehat{L}_{ \widehat{j}_n } \gamma_n , a_n + \widehat{L}_{ \widehat{j}_n } \gamma_n + j_n ]  \\
& \subseteq [ - ( \widehat{L}_{ \widehat{j}_n } + \widehat{j}_n + \gamma_n^{-1} ) \gamma_n , ( \widehat{L}_{ \widehat{j}_n } + \widehat{j}_n + \gamma_n^{-1} ) \gamma_n ] \\
& \subseteq [ - 2 ( \widehat{L}_{ \widehat{j}_n } + \widehat{j}_n + \gamma_n^{-1} ) \gamma , 2 ( \widehat{L}_{ \widehat{j}_n } + \widehat{j}_n + \gamma_n^{-1} ) \gamma ] .
\end{align*}
Similarly, if $x_{n} < 0$, a similar chain of inclusions leads to the same deduction:
	\begin{equation*}
		[ \gamma \gamma_n^{-1} x_n , \gamma \gamma_n^{-1}  x_n + \gamma \gamma_{n}^{-1} j_n ] \subseteq  [ - 2 ( \widehat{L}_{ \widehat{j}_n } + \widehat{j}_n + \gamma_n^{-1} ) \gamma , 2 ( \widehat{L}_{ \widehat{j}_n } + \widehat{j}_n + \gamma_n^{-1} ) \gamma ].
	\end{equation*}

Let $\mathcal{R}(\gamma)$ be defined by \eqref{E: definition of tau appendix} so that, in particular, $ \mathcal{R}(\gamma) = 2 ( \widehat{L}_{ \widehat{j}_n } + \widehat{j}_n + 1 ) $. The result of the previous paragraph implies that, for each $n \geq \max\{N_{0}, N_{*}(\omega)\}$,
	\begin{equation} \label{E: minimum inclusion}
		[  x , x + j ] 
		\subseteq [ - \mathcal{R}(\gamma) , \mathcal{R}(\gamma) ], 
		\quad { \rm where } \quad 
		x \coloneqq \gamma \gamma_n^{-1} x_n, \, \, j \coloneqq \gamma \gamma_{n}^{-1} j_n ,
	\end{equation}
and, further,
	\begin{align} \label{E: long intervals}
	\theta(\gamma^{-1} y) = \theta_{*} \quad \text{for a.e.} \quad y \in [ x, x + j) .
\end{align}

Suppose that $R \mapsto \gamma(R) $ is any scaling for which there holds
	\begin{equation} \label{E: another weird scaling condition}
		\lim_{R \uparrow \infty} \gamma(R) = \lim_{R \uparrow \infty} \frac{ \gamma(R) \mathcal{R}(\gamma(R)) }{ R } = 0.
	\end{equation}
Notice that \eqref{E: minimum inclusion} and \eqref{E: long intervals} imply that there is a random scale $R_{*}(\omega) \geq 0$, which is finite almost surely, such that if $R \geq R_{*}(\omega)$, then, with probability one, there is an $x(R) \in \mathbb{R}$ and $j(R) \geq 2^{-1} | \log_{2} \gamma(R) | $ such that 
	\begin{align*}
		\theta(\gamma(R)^{-1} y ) = \theta_{*} \quad \text{for a.e.} \, \, y \in [ x(R), x(R) + j(R) ) \subseteq [-R \varrho, R \varrho ] .
	\end{align*}
In particular, $j(R) \to \infty$ as $R \to \infty$.

\textit{Step 4 (Restatement in terms of $\delta(\epsilon)$).}  Finally, assume, as in the statement of the theorem, that $\epsilon \mapsto \delta(\epsilon)$ is any scaling such that
	\begin{equation*}
		\lim_{\epsilon \downarrow 0} \epsilon^{-1} \delta(\epsilon)
		= \lim_{\epsilon \downarrow 0} \delta(\epsilon) \mathcal{R}(\epsilon^{-1} \delta(\epsilon)) = 0.
	\end{equation*}
Having in mind the change-of-variables $R = \epsilon^{-1}$, define $\gamma(\epsilon) = \epsilon^{-1} \delta(\epsilon)$.  In what follows, we abuse notation by also writing $\gamma(R)$.

Observe that, with the above change of variables, our choice of $\epsilon \mapsto \delta(\epsilon)$ yields
	\begin{align*}
		\lim_{ R \uparrow \infty } \frac{ \gamma(R) \mathcal{R}(\gamma(R)) }{ R } = \lim_{ \epsilon \downarrow 0 } \epsilon \gamma( \epsilon ) \mathcal{R} ( \gamma ( \epsilon ) ) = \lim_{ \epsilon \downarrow 0 } \delta( \epsilon ) \mathcal{R} ( \epsilon^{-1} \delta ( \epsilon ) ) = 0
	\end{align*}
and similarly $\gamma(R) \to 0$ as $R \to \infty$.  In particular, the scaling $R \mapsto \gamma(R)$ satisfies \eqref{E: another weird scaling condition}.  Therefore, from the previous step of the proof (in particular, \eqref{E: long intervals}), for any deterministic $M \in \mathbb{N}$ and $\varrho > 0$, there is a random scale $\mathcal{E}(M,\varrho) \geq 0$, which is positive almost surely, such that, for any $\epsilon < \mathcal{E}(M,\varrho)$, there is a point $T_{\epsilon} \in [-\varrho,\varrho]$ for which we know that
	\begin{align} \label{E: long intervals redux}
		\theta(\delta(\epsilon)^{-1} x) = \theta_{*} \quad \text{for each} \quad x \in [T_{\epsilon} - M \epsilon, T_{\epsilon} + M \epsilon] \subseteq [-\varrho,\varrho].
	\end{align}\end{proof}

\bibliographystyle{plain}
\bibliography{bibliography}

\end{document}